\documentclass[10pt,namelimits,sumlimits]{amsart}
\usepackage{amssymb,amsmath}
\usepackage[mathscr]{eucal}
\usepackage{color}
\usepackage{enumitem}
\usepackage[utf8]{inputenc}
\usepackage[leftcaption]{sidecap}
\usepackage{tikz}
 \usepackage{tikz-cd}
\usetikzlibrary{matrix,arrows,decorations.pathmorphing}

\usetikzlibrary{positioning}
\usetikzlibrary{matrix,arrows,decorations.pathmorphing, shapes.geometric}

\usepackage{psfrag}

\usepackage{epstopdf}
\usepackage{graphicx}
\usepackage{verbatim}
\usepackage{amsmath,amsthm}
\usepackage{graphics}
\usepackage{color}
\usepackage{epsfig}
\usepackage{fullpage}
\usepackage{amssymb,amsmath}
\usepackage[mathscr]{eucal}
\usepackage{ upgreek }

\newtheorem{theorem}[equation]{Theorem}
\newtheorem{lemma}[equation]{Lemma}
\newtheorem{prop}[equation]{Proposition}

\newtheorem{corollary}[equation]{Corollary}

\newtheorem{definition}[equation]{Definition}

\theoremstyle{remark}
\newtheorem{remark}[equation]{Remark}
\newtheorem{notation}[equation]{Notation}
\newtheorem{convention}[equation]{Convention}
\newtheorem{assumption}[equation]{Assumption}

\numberwithin{equation}{section}

\newcommand{\supp}{\text{supp}}

\newcommand{\Thetacyl}{\Theta_{\cyl}}

\newcommand{\what}{\widehat{w}}
\newcommand{\vhat}{\widehat{v}}
\newcommand{\uhat}{\widehat{u}}
\newcommand{\hhat}{\widehat{h}}

\newcommand{\Lhat}{\widetilde{L}}

\newcommand{\tauo}{\tau_{1}}
\newcommand{\taupol}{\tau_{pol}}
\newcommand{\tauhat}{\widehat{\tau}}
\newcommand{\tautilde}{\widetilde{\tau}}
\newcommand{\Dcal}{{\mathcal{D}}}

\newcommand{\gammahat}{{\widehat{\gamma}}} 
\newcommand{\betahat}{{\widehat{\beta}}} 
\newcommand{\varphihat}{\widehat{\varphi}}
\newcommand{\varphinl}{{\varphi_{nl}}}
\newcommand{\varphiinit}{{\varphi_{init}}}

\newcommand{\phieq}{{\widehat{\phi}_{eq}}}

\newcommand{\phie}{{\phi_{even}}}
\newcommand{\phio}{{\phi_{odd}}}

\newcommand{\gammagl}{{\alpha}}

\newcommand{\osc}{{{osc}}}
\newcommand{\ave}{{{avg}}}

\newcommand{\gtilde}{{\widetilde{g}}}
\newcommand{\Gtilde}{{\widetilde{G}}}
\newcommand{\chitilde}{{\widetilde{\chi}}}
\newcommand{\stilde}{{\widetilde{\sss}}}
\newcommand{\sbar}{\underline{\sss}}
\newcommand{\vunder}{\underline{v}}
\newcommand{\kunder}{\underline{k}}

\newcommand{\mupol}{{\mu_{pol}}}
\newcommand{\Vpol}{{V_{pol}}}
\newcommand{\zetapol}{{\zeta_{pol}}}
\newcommand{\Wpol}{{W_{pol}}}

\newcommand{\thetatilde}{{\widetilde{\theta}}}

\newcommand{\atilde}{{\widetilde{a}}}
\newcommand{\btilde}{{\widetilde{b}}}

\newcommand{\mutilde}{{\widetilde{\mu}}}

\newcommand{\symphi}{\textstyle\frac12 (\, \partial_+\phi(\sss_i) - \partial_-\phi(\sss_i)\, )}
  
\newcommand{\phitilde}{{\widetilde{\phi}}}
\newcommand{\phitildeq}{{\phitilde_{eq}}}

\newcommand{\phihat}{{{\widehat{\phi}}}}  
\newcommand{\upphihat}{{\widehat{\upphi}}}
\newcommand{\skernel}{\mathscr{K}}
\newcommand{\skernelv}{\widehat{\mathscr{K}}}
\newcommand{\val}{\mathscr{V}}

\newcommand{\R}{\mathbb{R}}

\newcommand{\Z}{\mathbb{Z}}
\newcommand{\N}{\mathbb{N}}

\newcommand{\Sph}{\mathbb{S}}
\newcommand{\Spheq}{\mathbb{S}^2_{eq}}
\newcommand{\Cir}{\mathbb{P}}

\newcommand{\cat}{\mathbb{K}}

\newcommand{\PiSph}{\Pi_{\Spheq}}

\newcommand{\Ric}{\operatorname{Ric}}

\newcommand{\sech}{\operatorname{sech}}

\newcommand{\cunder}{\underline{c}\,}

\newcommand{\phiunder}{{\underline{\phi}}}

\newcommand{\junder}{{\underline{j}}}

\newcommand{\epsilonunder}{\underline{\epsilon}}

\newcommand{\Lcal}{{\mathcal{L}}}
\newcommand{\Lcalp}{{\mathcal{L}'}}

\newcommand{\Acal}{{\mathcal{A}}}

\newcommand{\Bcal}{{\mathcal{M}}}
\newcommand{\Rcal}{{\mathcal{R}}}
\newcommand{\Jcal}{{\mathcal{J}}}
\newcommand{\Ecal}{{\mathcal{E}}}
\newcommand{\Ecalinv}{{\mathcal{E}^{-1}_L}}

\newcommand{\dbold}{{\mathbf{d}}}

\newcommand{\zetabold}{{\boldsymbol{\zeta}}}
\newcommand{\zetaboldpol}{{\widetilde{\zetabold}}}

\newcommand{\zetaboldhat}{{\widehat{\boldsymbol{\zeta}}}}

\newcommand{\mubold}{{\boldsymbol{\mu}}}
\newcommand{\muboldtilde}{{\widetilde{\mubold}}}
\newcommand{\zerobold}{{\boldsymbol{0}}}
\newcommand{\xiw}{{{w}}}
\newcommand{\xibold}{{\boldsymbol{\xi}}}

\newcommand{\avg}{\operatornamewithlimits{avg}}
\newcommand{\sss}{{{\ensuremath{\mathrm{s}}}}}
\newcommand{\sbold}{{\boldsymbol{\sss}}}

\newcommand{\sym}{{sym}}
\newcommand{\xx}{{{\ensuremath{\mathrm{x}}}}}

\newcommand{\shat}{{{\widehat{\sss}}}}
\newcommand{\shatbar}{{{\widehat{\sbar}}}}

\newcommand{\Acalsssi}{{\Acal_{\sss_i}\!}}
\newcommand{\Rcalsssi}{\mathcal{R}_{\sss_i}\!}

\newcommand{\cyl}{\ensuremath{\mathrm{Cyl}}}

\newcommand{\partialx}{\partial}
\newcommand{\Cmk}{c_2}

\newcommand{\yy}{\ensuremath{\mathrm{y}}}
\newcommand{\zz}{\ensuremath{\mathrm{z}}}
\newcommand{\rr}{r}
\newcommand{\sssunder}{{\underline{\sss}}}

\newcommand{\xxtilde}{{\ensuremath{\widetilde{\mathrm{x}}}}}

\newcommand{\yytilde}{{\ensuremath{\widetilde{\mathrm{y}}}}}

\newcommand{\domTheta}{\text{Dom}_\Theta}

\newcommand{\Pp}{{\widehat{G}}}
\newcommand{\Ghat}{{\widehat{G}}}
\newcommand{\Phat}{{\widehat{\Phi}}}
\newcommand{\phat}{{\widehat{\phi}}}

\newcommand{\phip}{\phi'}
\newcommand{\Phip}{\Phi'}

\newcommand{\grouptwo}{{\mathscr{G}_{\Spheq, m}}}

\newcommand{\Gcal}{\mathcal{G}}

\newcommand{\groupthree}{{\mathscr{G}_{\Sph^3,m}}}

\newcommand{\Xhat}{\widehat{X}}

\newcommand{\Mhat}{\widehat{M}}

\newcommand{\Fcal}{{\mathcal{F}}}
\newcommand{\sssX}{\widehat{\mathsf{X}}}
\newcommand{\YYY}{\widehat{\mathsf{Y}}}
\newcommand{\ZZZ}{\widehat{\mathsf{Z}}}
\newcommand{\bsigma}{\boldsymbol{\sigma}}
\newcommand{\bsigmaunder}{\underline{\bsigma}}
\newcommand{\abold}{\boldsymbol{a}}
\newcommand{\Fbold}{\boldsymbol{F}}
\newcommand{\Fboldunder}{\underline{\Fbold}}
\newcommand{\xbar}{\underline{\sssX}}
\newcommand{\ybar}{\underline{\YYY}}
\newcommand{\zbar}{\underline{\ZZZ}}
\newcommand{\xbartilde}{{\underline{\mathsf{X}}}}
\newcommand{\ybartilde}{{\underline{\mathsf{Y}}}}
\newcommand{\zbartilde}{{\underline{\mathsf{Z}}}}

\newcommand{\ytilde}{{\mathsf{Y}}}

\newcommand{\disjun}{\textstyle\bigsqcup}

\newcommand{\mmer}{{m_{mer}}}
\newcommand{\mpar}{{m_{par}}}

\newcommand{\Lmer}{{L_{mer}}}
\newcommand{\Lpar}{{L_{par}}}

\newcommand{\psicut}{{\psi_{cut}}}
\newcommand{\Psibold}{{\boldsymbol{\Psi}}}

\newcommand{\pointp}{{(F_1, \bsigmaunder)}}
\newcommand{\pointppol}{{(F_1, \bsigmaunder, \tautilde)}}

\begin{document}

\title[Doubling]{Minimal Surfaces in the Round Three-Sphere by Doubling the Equatorial Two-Sphere, II}

\author[N.~Kapouleas]{Nikolaos~Kapouleas}
\author[P.~McGrath]{Peter~McGrath}

\address{Department of Mathematics, Brown University, Providence,
RI 02912} \email{peter\_mcgrath@math.brown.edu} \email{nicos@math.brown.edu}

\date{\today}


\begin{abstract}
In \cite{kap} 
new closed embedded smooth minimal surfaces in the round three-sphere $\mathbb{S}^3(1)$ were constructed, 
each resembling two parallel copies of the equatorial two-sphere $\Spheq$ joined by small catenoidal bridges,  
with the catenoidal bridges concentrating along two parallel circles, 
or the equatorial circle and the poles.  
In this sequel 
we generalize those constructions 
so that the catenoidal bridges can concentrate along an arbitrary number of parallel circles, 
with the further option to include bridges at the poles. 
The current constructions follow the Linearized Doubling (LD) methodology developed in \cite{kap}  
and the LD solutions constructed here can be modified readily for use 
to doubling constructions of rotationally symmetric minimal surfaces with asymmetric sides \cite{LDa}. 
In particular they allow us to develop in \cite{LDa} doubling constructions for 
the catenoid in Euclidean three-space, the critical catenoid in the unit ball, 
and the spherical shrinker of the mean curvature flow.  

Unlike in \cite{kap}, our constructions here 
allow for sequences of minimal surfaces where the catenoidal bridges tend to be ``densely distributed'',  
that is do not miss any open set of $\Spheq$ in the limit. 
This in particular leads to interesting observations which seem 
to suggest that it may be impossible to construct embedded minimal surfaces with isolated singularities 
by concentrating infinitely many catenoidal necks at a point. 
\end{abstract}
\maketitle

\section{Introduction}
\label{S:intro}
$\phantom{ab}$
\nopagebreak
\subsection*{The general framework}
$\phantom{ab}$
\nopagebreak

This article is the second one 
in a series in which gluing constructions
for closed embedded minimal surfaces in the round three-sphere $\Sph^3(1)$
by doubling the equatorial two-sphere $\Spheq   $ are discussed. 
Doublings of the equatorial two-sphere $\Spheq   $ are important 
as a test case for developing the doubling methodology and also 
because their area is close to $8\pi$ 
(the area of two equatorial two-spheres), 
a feature they share with the celebrated surfaces 
constructed by Lawson in 1970 \cite{L2}.
The classification of the low area 
closed embedded minimal surfaces in the round three-sphere $\Sph^3(1)$, 
especially of those of area close to $8\pi$ 
or less, 
is a natural open question.
This is further motivated by the recent resolutions of 
the Lawson conjecture by Brendle \cite{brendle} 
and the Willmore conjecture by Marques and Neves \cite{neves} 
where they also characterize the Clifford torus and the equatorial sphere
as the only examples of area $\le2\pi^2$. 
We refer to \cite{Brendle:survey} for a survey of existence and uniqueness
results for minimal surfaces in the round three-sphere.

The general idea of doubling constructions by gluing methods 
was proposed and discussed in 
\cite{kapouleas:survey,kapouleas:clifford,alm20}.
Gluing methods have been applied extensively
and with great success in Gauge Theories by Donaldson, Taubes, and others.
The particular kind of gluing methods used in this article relate most
closely to the methods developed in \cite{schoen} and \cite{kapouleas:annals},
especially as they evolved and were systematized in
\cite{kapouleas:wente:announce,kapouleas:wente,kapouleas:imc}.
We refer to \cite{kapouleas:survey} for a general discussion of this gluing methodology 
and to \cite{alm20} for a detailed general discussion of doubling by gluing methods.

Roughly speaking, in such doubling constructions,  
given a minimal surface $\Sigma$,
one constructs first a family of smooth, embedded, and approximately minimal, initial surfaces. 
Each initial surface 
consists of two approximately parallel copies of $\Sigma$ with a number of discs removed and replaced by approximately catenoidal bridges.
Finally, one of the initial surfaces in the family is perturbed to minimality by Partial Differential Equations methods.
Understanding such constructions in full generality seems beyond the immediate horizon at the moment.
In the earliest such construction \cite{kapouleas:clifford} 
where doublings of the Clifford torus are constructed, 
there is so much symmetry imposed, 
that the position of the catenoidal bridges is completely fixed 
and all bridges are identical modulo the symmetries.
Moreover the bridges are uniformly distributed, 
that is when their number is large enough, 
there are bridges located inside any preassigned domain of $\Sigma$.
Wiygul \cite{wiygul:t,Wiygul:s} has extended that construction to multiple doublings 
with more that two copies of the Clifford torus involved (and some less symmetric doublings also), 
where the symmetries do not determine the vertical (that is perpendicular to $\Sigma$) 
position of the bridges.

In a previous article \cite{kap} doubling constructions where 
the horizontal position of the bridges is not determined by the symmetries, 
or there are more than one bridge modulo the symmetries, 
were carried out for the first time.  
To realize such constructions an intermediate step in the construction was introduced.  
In this intermediate step singular solutions of the linearized equation, 
called \emph{linearized doubling} (LD) solutions, 
are constructed and studied.  
This new approach which we call Linearized Doubling (LD),  
provides a systematic methodology for dealing with the obstructions involved 
and also provides a detailed understanding of the regions further away from the catenoidal bridges.
It can also be generalized to higher dimensions \cite{kapouleas:high:doubling}. 
In \cite{kap} the conversion of suitable continuous families of LD solutions on $\Spheq$ 
into minimal surfaces whose catenoidal bridges ``replace'' the singularities of the LD solutions was realized in general as part of the LD methodology. 
This reduced the construction of the minimal surfaces to the construction and estimation of the LD solutions, 
which remained however a very difficult problem.  

In \cite{kap} the only LD solutions which were constructed had their singularities either on two parallel 
circles symmetrically arranged around the equatorial circle of $\Spheq$, or 
at the poles and the equatorial circle of $\Spheq$.  
This way two (discrete) families of minimal surfaces were obtained, 
a family with the catenoidal bridges concentrating on two parallel circles symmetrically arranged around the equatorial circle of $\Spheq$
and a family where there are two bridges at the poles and the rest concentrate on the equatorial circle of $\Spheq$.  

\subsection*{Brief discussion of the results}
$\phantom{ab}$
\nopagebreak

In this article we study LD solutions with singularities on an arbitrary number (but $\ge2$) of parallel 
circles of $\Spheq$ (subject to the same symmetries as in \cite{kap}), 
with the option to have singularities at the poles also.  
This leads to the construction of minimal doublings of $\Spheq$ with the catenoidal 
bridges replacing the singularities of the LD solutions 
(see Theorems \ref{Tmain}, \ref{Tmainpol}, \ref{Tspheq}, and \ref{Tspheqpol})
and thus concentrating along an arbitrary number $\ge2$ of parallel circles and optionally at the poles.  

In particular we obtain for the first time (unlike in \cite{kap}) 
sequences of minimal doublings of $\Spheq$ where 
the number of the parallel circles tends to infinity 
and therefore the number of bridges contained in any 
fixed in advance open subset of $\Spheq$ tends also to infinity, 
that is the catenoidal bridges become ``densely distributed'' in the limit. 
We observe then the interesting phenomenon 
that the size of the catenoidal bridges on each minimal surface of such a sequence 
tends to become uniform,  
in the sense that 
$$
\lim_{\mpar\to\infty} \, 
\lim_{\mmer\to\infty} \, \tau_{max}/\tau_{min} =1, 
$$
where $\mpar$ is the number of the parallel circles, $\mmer$ the number of bridges on each circle, 
and
$\tau_{max}/\tau_{min}$ the ratio of the maximum size over the minimum size of the bridges  
(see Remarks \ref{rtausize} and \ref{rtausizepol}). 
This happens even in the case where nearby catenoidal bridges experience very different geometries, 
as when the bridge at a pole is surrounded by a very large number of bridges on nearby parallel circles.  
This suggests that when there is a very large number of catenoidal bridges close to a point, 
even in asymmetric situations, the sizes of the bridges would tend to become uniform. 
This would imply a negative answer to the very important question on whether embedded minimal surfaces 
with isolated singularities can be constructed by concentrating an infinite number of catenoidal 
bridges in the vicinity of a singular point. 

Finally we remark that the families of LD solutions we construct here find immediate application in \cite{LDa} 
to constructions of complete minimal surfaces in the Euclidean space by doubling the catenoid, 
constructions of free boundary minimal surfaces in the unit ball by doubling the critical catenoid, 
and of self-shrinkers of the mean curvature flow by doubling the spherical self-shrinker: 
Although the LD methodology as applied in \cite{kap} and in the current paper assumes 
that the minimal surface being doubled possesses a symmetry fixing it pointwise and exchanging 
its sides, as $\Spheq$ does, 
it is possible 
\cite[Remark 3.21]{kap}  
to extend the methodology to the case of asymmetric sides, 
as in the case of the catenoid and the spherical self-shrinker.  
Because the catenoid is conformally isometric to $\Spheq$ and the spherical self-shrinker is round 
we can use the LD solutions we derive here with small modifications to carry out the constructions 
in \cite{LDa}.

\subsection*{Outline of the approach}
$\phantom{ab}$
\nopagebreak

In this article, as in \cite{kap}, in order to construct a family of initial surfaces 
one of which is later perturbed to minimality by a fixed point theorem,  
we first construct a family of LD solutions 
(defined in \cite[Definition 3.1]{kap} and \ref{dLD0}). 
The LD solutions are converted to 
\emph{matched linearized doubling} (MLD) solutions 
(defined in \cite[Definition 3.4]{kap} and \ref{DMLD}),  
which are then converted to initial surfaces. 
In both articles the construction and estimation of the LD solutions relies 
on the rotational invariance of $\Spheq$. 
Each LD solution $\varphi$ is related to a solution $\phi$ obtained by averaging $\varphi$
on the circles of latitude  
(see \cite[Lemma 5.8]{kap} and \ref{Lphiavg}). 
$\phi$ then belongs to a class of rotationally invariant solutions which in this article 
we call 
\emph{rotationally invariant linearized doubling} (RLD) solutions (see  \ref{RL}). 
The RLD solutions studied in this article are rather more complicated compared to the ones studied in \cite{kap}. 
Much of the progress achieved involves their detailed understanding. 
We summarize schematically the main steps in the construction in the following diagram 
(see Remark \ref{Rsummary} for a more detailed outline): 
\begin{figure}[h]
\label{Foutline}
\begin{center}
\begin{tikzcd}
RLD \arrow{r} & LD \arrow{r} & MLD \arrow{r} & \text{Initial Surface} \arrow{r}& \text{Minimal Surface} \\[-.5cm]
\phat & \Phi & \varphi & M & \Mhat_\upphihat
\end{tikzcd}
 \end{center}
 \caption{The main steps of the construction}
 \end{figure}
 
We discuss now the main innovations of this paper. 
A large part of the effort is in understanding and estimating in detail the RLD solutions. 
Achieving this is helped by the observation that the class of LD solutions is invariant 
under conformal changes of the intrinsic metric of the surface and therefore RLD solutions can 
be considered as defined on the flat cylinder instead of the round sphere. 
We also employ a dimensionless flux $F^\phi$, 
which amounts to the logarithmic derivative of the RLD solution $\phi$, 
to carefully study the RLD solutions. 
Such a quantity was used also in \cite{kap}, 
but here we study $F^\phi$ much more carefully 
and we establish that it satisfies a Riccati differential equation and has several useful monotonicity properties.  

The fact that unlike in \cite{kap} we are dealing with situations where there are more than one (modulo the symmetries) circles of latitude where 
the catenoidal bridges concentrate, makes the balancing and unbalancing questions for the RLD solutions harder to study. 
The fluxes $F^\phi$ are the main tool which allows us to study these questions systematically. 
We also have developed a careful and systematic approach to study the parameters related to unbalancing (see \ref{RLquant}).

The current article incorporates various improvements in the application of the LD methodology: 
First, 
we estimate our LD solutions on the flat cylinder instead of the equatorial two-sphere as in \cite{kap}  
by making use of the conformal invariance of the LD solutions. 
The uniformity and further symmetries of the cylinder allow us to obtain stronger estimates.  
Second, we improve the analysis of the relation between the LD solutions and the corresponding RLD solutions. 
In particular, the decomposition 
$\Phi=\Pp+\Phat+\Phip$ in \ref{dPprime} of an LD solution $\Phi$ into a singular part $\Pp$, 
a rotationally invariant part $\Phat$, and an error term $\Phip$, 
is an improvement over the two decompositions \cite[Definitions 5.16 and 5.25]{kap} it replaces. 
Finally, we employ a simplified definition for $\skernel[L]$ in \ref{dkernel}. 

\subsection*{Organization of the presentation}
$\phantom{ab}$
\nopagebreak

In Section \ref{Notation}, we review definitions and notation from \cite{kap} 
relating to the elementary geometry of the objects we are interested in 
and catalog a useful conformal diffeomorphism \eqref{Esphcyl} between the cylinder and the twice-punctured two-sphere.
We also discuss some special rotationally invariant solutions and Green's functions for the linearized equation. 

The main new features in this paper which refine the approach in \cite{kap} take place in Sections \ref{S:RLD} and \ref{S:LD}.  
In Section \ref{S:RLD}, we define in \ref{RL} a class of rotationally invariant solutions of the linearized equation (RLD solutions), 
establish appropriate criteria for their existence and uniqueness (see \ref{existence}), 
and prove estimates governing their geometry in \ref{Prl}, 
and behavior under small perturbations of initial data in \ref{PODEest}.

In Section \ref{S:LD}, we more generally study linearized doubling solutions (LD solutions) on the cylinder which have prescribed logarithmic singularities at $L$. 
More precisely in Lemma \ref{Lphiavg} we convert RLD solutions $\phat$ to corresponding LD solutions $\Phi$. 
We introduce then in \ref{dPprime} the decomposition $\Phi=\Pp+\Phat+\Phip$ of an LD solution $\Phi$ into a singular part $\Pp$, 
a rotationally invariant part $\Phat$, and an error term $\Phip$.  
Much of the remaining work in the section is devoted to estimating $\Phip$ in \ref{LPhip}.

In Section \ref{S:MLD},
we convert the LD solutions we have constructed in Section \ref{S:LD}
to MLD solutions which satisfy the appropriate linear and nonlinear matching conditions.  
Using the earlier estimates, we prove in \ref{LLD} the estimates we need for these MLD solutions.  
In Section \ref{S:Main}, 
we convert these families of MLD solutions into families of smooth initial surfaces with small mean curvature.
by recalling work from \cite{kap}, 
and we then apply a fixed point theorem to perturbations of the families of initial surfaces constructed in Section \ref{S:MLD} 
to construct our minimal surfaces.  
Finally in Section \ref{S:poles} we modify the construction to include catenoidal bridges at the poles or on the equatorial circle of $\Spheq$ or both.

\subsection*{Acknowledgments}
\phantom{ab}
\nopagebreak

The authors would like to thank Richard Schoen for his continuous support and interest in the results of this article. 
NK would like to thank the Mathematics Department and the MRC at Stanford University
for providing a stimulating mathematical environment and generous financial support during Spring 2016.
NK was also partially supported by NSF grant DMS-1405537.

\section{Elementary geometry and notation}
\label{Notation}

\subsection*{H{\"o}lder norms and cut-off functions}
\label{sub:holder}
$\phantom{ab}$
\nopagebreak

We will find the following notation useful. 
\begin{definition}
\label{Dsimc}
We write $a\sim_c b$ to mean that 
$a,b\in\R$ are nonzero of the same sign, 
$c\in(1,\infty)$, 
and $\frac1c\le \frac ab \le c$. 
\end{definition}

We use the standard notation $\left\|u: C^{r,\beta}(\,\Omega,g\,)\,\right\|$ 
to denote the standard $C^{r,\beta}$-norm of a function or more generally
tensor field $u$ on a domain $\Omega$ equipped with a Riemannian metric $g$.
Actually the definition is completely standard only when $\beta=0$
because then we just use the covariant derivatives and take a supremum
norm when they are measured by $g$.
When $\beta\ne0$ we have to use parallel transport along geodesic segments 
connecting any two points of small enough distance
and this may be a complication if small enough geodesic balls are not convex.
In this paper we take care to avoid situations where such a complication
may arise and so we will not discuss this issue further.

We adopt the following notation from \cite{kap} for weighted H\"{o}lder norms.  
\begin{definition}
\label{D:newweightedHolder}
Assuming that $\Omega$ is a domain inside a manifold,
$g$ is a Riemannian metric on the manifold, 
$r\in \N_0$, 
$\beta\in[0,1)$, $u\in C^{r,\beta}_{loc}(\Omega)$ 
or more generally $u$ is a $C^{k,\beta}_{loc}$ tensor field 
(section of a vector bundle) on $\Omega$, 
$\rho,f:\Omega\to(0,\infty)$ are given functions, 
and that the injectivity radius in the manifold around each point $x$ in the metric $\rho^{-2}(x)\,g$
is at least $1/10$,
we define
$$
\left\|u: C^{r,\beta} ( \Omega,\rho,g,f)\right\|:=
\sup_{x\in\Omega}\frac{\,\left\|u:C^{r,\beta}(\Omega\cap B_x, \rho^{-2}(x)\,g)\right\|\,}{f(x) },
$$
where $B_x$ is a geodesic ball centered at $x$ and of radius $1/100$ in the metric $\rho^{-2}(x)\,g$.
For simplicity we may omit any of $\beta$, $\rho$, or $f$, 
when $\beta=0$, $\rho\equiv1$, or $f\equiv1$, respectively.
\end{definition}

$f$ can be thought of as a ``weight'' function because $f(x)$ controls the size of $u$ in the vicinity of
the point $x$.
$\rho$ can be thought of as a function which determines the ``natural scale'' $\rho(x)$
at the vicinity of each point $x$.
Note that if $u$ scales nontrivially we can modify appropriately $f$ by multiplying by the appropriate 
power of $\rho$.  Observe from the definition the following multiplicative property: 
\begin{equation}
\label{E:norm:mult}
\left\| \, u_1 u_2 \, : C^{k,\beta}(\Omega,\rho,g,\, f_1 f_2 \, )\right\|
\le
C(k)\, 
\left\| \, u_1 \, : C^{k,\beta}(\Omega,\rho,g,\, f_1 \, )\right\|
\,\,
\left\| \, u_2 \, : C^{k,\beta}(\Omega,\rho,g,\, f_2 \, )\right\|.
\end{equation}

Our arguments will require extensive use of cut-off functions, and it will be helpful to adopt the following. 
\begin{definition}
\label{DPsi} 
We fix a smooth function $\Psi:\R\to[0,1]$ with the following properties:
\begin{enumerate}[label=(\roman*).]
\item $\Psi$ is nondecreasing.

\item $\Psi\equiv1$ on $[1,\infty)$ and $\Psi\equiv0$ on $(-\infty,-1]$.

\item $\Psi-\frac12$ is an odd function.
\end{enumerate}
\end{definition}

Given $a,b\in \R$ with $a\ne b$,
we define smooth functions
$\psicut[a,b]:\R\to[0,1]$
by
\begin{equation}
\label{Epsiab}
\psicut[a,b]:=\Psi\circ L_{a,b},
\end{equation}
where $L_{a,b}:\R\to\R$ is the linear function defined by the requirements $L(a)=-3$ and $L(b)=3$.

Clearly then $\psicut[a,b]$ has the following properties:
\begin{enumerate}[label=(\roman*).]
\item $\psicut[a,b]$ is weakly monotone.

\item 
$\psicut[a,b]=1$ on a neighborhood of $b$ and 
$\psicut[a,b]=0$ on a neighborhood of $a$.

\item $\psicut[a,b]+\psicut[b,a]=1$ on $\R$.
\end{enumerate}

Suppose now we have two sections $f_0,f_1$ of some vector bundle over some domain $\Omega$.
(A special case is when the vector bundle is trivial and $f_0,f_1$ real-valued functions).
Suppose we also have some real-valued function $d$ defined on $\Omega$.
We define a new section 
\begin{equation}
\label{EPsibold}
\Psibold\left [a,b;d \, \right](f_0,f_1):=
\psicut[a,b\, ]\circ d \, f_1
+
\psicut[b,a]\circ d \, f_0.
\end{equation}
Note that
$\Psibold[a,b;d\, ](f_0,f_1)$
is then a section which depends linearly on the pair $(f_0,f_1)$
and transits from $f_0$
on $\Omega_a$ to $f_1$ on $\Omega_b$,
where $\Omega_a$ and $\Omega_b$ are subsets of $\Omega$ which contain
$d^{-1}(a)$ and $d^{-1}(b)$ respectively,
and are defined by
$$
\Omega_a=d^{-1}\left((-\infty,a+\frac13(b-a))\right),
\qquad
\Omega_b=d^{-1}\left((b-\frac13(b-a),\infty)\right),
$$
when $a<b$, and 
$$
\Omega_a=d^{-1}\left((a-\frac13(a-b),\infty)\right),
\qquad
\Omega_b=d^{-1}\left((-\infty,b+\frac13(a-b))\right),
$$
when $b<a$.
Clearly if $f_0,f_1,$ and $d$ are smooth then
$\Psibold[a,b;d\, ](f_0,f_1)$
is also smooth.


\subsection*{The parametrizations $\Theta$ and $\Thetacyl$ and the coordinates $(\xx,\yy,\zz)$ and $(\sss, \theta)$}
$\phantom{ab}$
\nopagebreak

We consider now the unit three-sphere $\Sph^3(1)\subset\R^4$.
We denote by $(x_1,x_2,x_3,x_4)$ the standard coordinates of $\R^4$
and we define by
\begin{equation}
\label{Eeqtwo}
\Spheq:=\Sph^3(1)\cap\{x_4=0\}
\end{equation}
an equatorial two-sphere in $\Sph^3(1)$.
To facilitate the discussion we fix spherical coordinates $(\xx,\yy,\zz)$ on $\Sph^3(1)$
by defining a map $\Theta:\R^3\to\Sph^3(1)$ by
\begin{equation}
\label{ETheta}
  \Theta(\xx,\yy,\zz) = 
(\cos\xx \cos\yy \cos\zz, \cos\xx \sin\yy \cos\zz, \sin\xx\cos\zz,\sin\zz).
\end{equation}
Note that in the above notation we can think of $\xx$ as the geographic latitude on $\Spheq$
and of $\yy$ as the geographic longitude.
We will also refer to
\begin{equation}
\label{Eeqone}
\begin{gathered}
\Cir_0:= \Spheq\cap \{x_3=0\}=
\Theta(\{\xx=\zz=0\}),
\\
p_N:=(0,0,1,0)=\Theta(\pi/2,\yy,0),
\\
p_S:=(0,0,-1,0)=\Theta(-\pi/2,\yy,0),
\end{gathered}
\end{equation}
as the equator circle, the North pole, and the South pole of $\Spheq$ respectively.

Clearly, the standard metric of $\Sph^3(1)$ is given in the coordinates of \eqref{ETheta}
by
\begin{equation}
\label{EThetag}
    \Theta^* g = \cos^2\zz\,( \, d\xx^2 +\cos^2\xx \,d\yy^2 \,)\, + d\zz^2.
\end{equation}
Finally we define a nearest-point projection 
$\PiSph:\Sph^3(1)\setminus\{(0, 0, 0, \pm 1)\}\to\Spheq$
by
\begin{equation}
\label{EPiSph}
\PiSph(x_1,x_2,x_3,x_4)=
\frac1{|(x_1,x_2,x_3,0)|}
(x_1,x_2,x_3,0).
\end{equation}
Clearly we have
\begin{equation}
\label{EPiSphTheta}
\PiSph\circ \Theta(\xx,\yy,\zz)
=
\Theta(\xx,\yy,0).
\end{equation}

The study of the RLD and LD solutions can be simplified by the observation that $\Spheq\setminus \{p_N, p_S\}$ is conformally equivalent to a flat cylinder $\R\times \Sph^1$. 
To be precise, let
\begin{align}
\label{Ecyl}
\cyl : = \R\times \Sph^1 
\end{align}
be the cylinder endowed with the flat product metric $\chi$,  
where 
\begin{equation}
\label{ELchi}
\chi: = d\sss^2 + d\theta^2 
\end{equation}
and $(\sss,\theta)$ are the standard coordinates on $\cyl$. 
Consider the map $\Thetacyl:  \cyl \rightarrow \Spheq\setminus \{p_N, p_S\}$ defined by 
\begin{align}
\label{Esphcyl}
\Thetacyl(\sss, \theta) = ( \sech \sss \cos \theta, \sech \sss \sin \theta, \tanh \sss ) .
\end{align}
Clearly $\Thetacyl$ is a diffeomorphism and from now on we will use it to identify 
$\Spheq\setminus \{p_N, p_S\}$ with $\cyl$.  
$\sss$ and $\theta$ can then be considered as coordinates on 
$\Spheq\setminus \{p_N, p_S\}$ and by \eqref{ETheta}
we have 
\begin{align}
\label{Exs}
\sin \xx = \tanh \sss , \qquad \cos \xx = \sech \sss, \qquad
\sss = \log\frac{1+\sin\xx}{\cos\xx},\qquad
 \text{and}\quad  \yy = \theta.
\end{align}
Straightforward computations show 
\begin{align}
\label{Echig2}
 g = (\sech^2 \sss) \chi,  
\qquad \frac{d \sss}{d \xx} = \frac{1}{\cos \xx}, 
\qquad \frac{d\xx}{d\sss} = \sech \sss.
\end{align}

We finally introduce some convenient notation.

\begin{notation}
\label{Nsym2}
We use the shorthand notation $\{ \sss \in (a, b)\}$ to denote the annular region 
$\{ (\sss, \theta) \subset \cyl: \sss \in (a, b)\}$ and similar abbreviations with regard to other types of intervals. 
\end{notation}

\begin{notation}
\label{NT}
For $X$ a subset of $\Spheq$ and $h$ a Riemannian metric on $\Spheq$ or on $\Spheq\setminus \{p_N, p_S\}$, 
we write $\dbold^h_X$ for the distance function from $X$, with respect to $h$.  We define a tubular neighborhood of $X$ by
\[ D^h_X(\delta):=\left \{p\in\Spheq:\dbold^h_X(p)<\delta\right\}. \] 
If $X$ is finite we just enumerate its points in both cases. 
For example, $\dbold^g_q(p)$ is the geodesic distance between $p$ and $q$ and
$D^g_q(\delta)$ is the geodesic disc in $\Spheq$ of center $q$ and radius $\delta$.
\end{notation}

\subsection*{The symmetries and the configurations}
\label{sub:sym}
$\phantom{ab}$
\nopagebreak

To study the symmetries of the parametrization $\Theta$,
we first define
reflections
of its domain $\domTheta$
$\xbar$, $\ybar_c$, $\ybar:=\ybar_0$, and $\zbar$,
and translations
$\YYY_c$,
where $c\in\R$, by 
\begin{equation}
\label{Edomsym}
\begin{aligned}
\xbar(\xx,\yy,\zz):=(-\xx,\yy,\zz),\qquad
\ybar_c(\xx,\yy,\zz)&:=(\xx,2c-\yy,\zz),\qquad
\zbar(\xx,\yy,\zz):=(\yy,\xx,-\zz),
\\
\YYY_c(\xx,\yy,\zz)&:=(\xx,\yy+c,\zz).
\end{aligned}
\end{equation}

We also define corresponding
reflections
$\xbartilde$, $\ybartilde_c$, $\ybartilde:=\ybartilde_0$, and $\zbartilde$
and rotations $\ytilde_c$,
of $\R^4 \supset \Sph^3(1)$ by
\begin{equation}
\label{Esphsym}
\begin{aligned}
\xbartilde(x_1,x_2,x_3,x_4):=&(x_1,x_2,-x_3,x_4),
\\
\ybartilde(x_1,x_2,x_3,x_4):=&(x_1,-x_2,x_3,x_4),
\\
\zbartilde(x_1,x_2,x_3,x_4):=&(x_1,x_2,x_3,-x_4),
\\
\ybartilde_c(x_1,x_2,x_3,x_4):=&(x_1\cos 2c + x_2\sin 2c ,  x_1 \sin 2c- x_2 \cos 2c,x_3, x_4)
\\
\ytilde_c(x_1,x_2,x_3,x_4):=&(x_1\cos c - x_2\sin c, x_1\sin c + x_2\cos c,x_3, x_4).
\end{aligned}
\end{equation}

Note that $\xbartilde$, $\ybartilde$, $\zbartilde$, and $\ybartilde_c$
are reflections with respect to the $3$-planes
$\{x_3=0\}$, $\{x_2=0\}$, $\{x_4=0\}$, and $\ytilde_c(\{x_2=0\})$,
respectively.
$\zbartilde$ exchanges the two sides of $\Spheq$ which it fixes pointwise.
Clearly $\ytilde_{2\pi}$ is the identity map.
We record the symmetries of $\Theta$ in the following lemma:

\begin{lemma}
\label{LThetasymmetries}
$\Theta$ restricted to
$$
\domTheta:=
\left(-\frac{\pi}{2},\frac{\pi}{2}\right) \times \R \times \left(-\frac{\pi}{2},\frac{\pi}{2}\right)
$$
is a covering map onto 
$\Sph^3(1)\setminus\{x_1=x_2=0\}$.
Moreover the following hold:
\begin{enumerate}[label=(\roman*).]
\item
The group of covering transformations is generated by
$\YYY_{2\pi}$.

\item
$\xbartilde\circ\Theta=\Theta\circ\xbar$,
$\ybartilde_c\circ\Theta=\Theta\circ\ybar_c$,
$\zbartilde\circ\Theta=\Theta\circ\zbar$,
and
$\ytilde_c\circ\Theta=\Theta\circ\YYY_c$.
\end{enumerate}
\end{lemma}


The symmetry group of our constructions depends on a large number $m\in\N$
which we assume now fixed.
We define $\Lmer=\Lmer[m]\subset\Spheq$ to be the union of $m$ meridians symmetrically arranged: 
\begin{equation}
\label{ELmer}
\Lmer=\Lmer[m]:=\Theta\left(\{(\xx,\yy,0):\xx\in[-\pi/2,\pi/2],\yy=2\pi i/m, i\in\Z\}\right).
\end{equation}

\begin{definition}[The symmetry groups] 
\label{Dgroup}
We denote by $\groupthree$ and $\grouptwo$ the groups of 
isometries of $\Sph^3(1)$ and $\Spheq$ respectively
which fix $\Lmer[m]$ as a set.
\end{definition}

Clearly $\groupthree$ is a finite group and is generated by the reflections
$\xbartilde$, $\ybartilde$, $\zbartilde$ and $\ybartilde_{\pi/m}$, that is  
\begin{align}
\label{Esymmetries}
\groupthree = \left\langle \xbartilde, \ybartilde, \zbartilde, \ybartilde_{\pi/m} \right\rangle,
\end{align}
and moreover $\grouptwo$ can be identified with the subgroup of $\groupthree$ generated by $\xbartilde, \ybartilde$ and $\ybartilde_{\pi/m}$.
The configuration of our constructions also depends on a number $k\in \N$ whose values are restricted in terms of $m$ (see \ref{Aratio} below).

\begin{definition}
\label{Dlpar}
Given $\sbar \in [0, \infty)$, we define 
\begin{align*}
 L_{par}[\sbar] = \{ (\sss, \theta) \in \cyl : \sss = \pm \sbar\}.
\end{align*}
Furthermore, given 
$\sbold := (\sss_1, \dots, \sss_k)\in \R^k_+$ 
such that 
$0 <\sss_1<\cdots<\sss_k < \infty$, 
we define 
 \[ L_{par}[\sbold] = \bigcup_{i=1}^k L_{par}[\sss_i].\]
Finally given also 
a domain $\Omega\subset \Spheq$, 
we will denote by $\Omega^\sbold$ the 
``subdivision'' of $\Omega$ by $L_{par}[\sbold]$: 
More precisely $\Omega^\sbold$ is the abstract surface 
which is the disjoint union of the $\Omega\cap A$'s,  
where $A$ is the closure of any connected component 
(a disk or an annulus) of $\Spheq\setminus L_{par}[\sbold]$.  
Clearly functions on $\Omega$ can be thought of as functions on $\Omega^\sbold$ as well. 
\end{definition}

Note for example that a function defined on $\Omega$ which 
is in $C^\infty(\Omega^\sbold)$ is also in $C^0(\Omega)$ but not necessarily 
in $C^1(\Omega)$; it is ``piecewise smooth'' on $\Omega$.

\begin{definition}
\label{dL}
For $m$ as in \ref{ELmer} and 
$\sbar$ and $\sbold$ as in \ref{Dlpar}, 
we define  
\begin{align*}
L[\sbar ; m]  : =& L_{mer}[m]\cap L_{par}[\sbar ]  = \grouptwo (\sbar,0),
\\ 
L = L[\sbold; m] :=& L_{mer}[m]\cap L_{par}[\sbold] = \bigcup_{i=1}^k L[\sss_i;m].  
\end{align*}
For $i \in\{1,...,k\}$ we define $p_i:=(\sss_i,0)$, $L_i=L[\sss_i;m]$, 
and given also a $\grouptwo$-symmetric function $\tau:L[\sbold; m]\rightarrow\R$, 
$\tau_i := \tau(p_i)$.  
\end{definition}

%
\subsection*{The linearized equation and the cylinder} 
\label{sub:conformal}
$\phantom{ab}$
\nopagebreak

A major step in the construction is to estimate solutions of the Jacobi equation $\Lcalp u  = 0$, where
\begin{equation}
\label{Ejacobi}
\Lcalp = \Delta_{g} + |A|^2 + \Ric(\nu, \nu)= \Delta_{g}+ 2.
\end{equation}
We define also a version of $\Lcal'$ useful 
in the cylindrical picture by 
$$
\Lcal_\chi := \Delta_\chi + 2\sech^2 \sss = \sech^2 \sss \, \Lcal', 
\qquad
\text{ where }  
\qquad
\Delta_\chi := \frac{\partial^2}{\partial \sss^2} +  \frac{\partial^2}{\partial \theta^2}. 
$$ 

It will be easier to state some of our estimates if we use a scaled metric $\gtilde$ on $\Spheq$ and corresponding scaled coordinates $( \xxtilde,\yytilde)$ defined by 
\begin{align}
\label{Egtilde}
\gtilde:=m^2 g_{\Spheq},
\qquad
\xxtilde=m\xx,
\qquad
\yytilde=m\yy.
\end{align}
In the same fashion, we define a scaled metric $\chitilde$ on $\cyl$ and scaled coordinates
  $(\, \stilde, \thetatilde\, )$
defined by
\begin{align}
\label{Echitilde}
\chitilde:= m^2 \chi, 
\qquad 
\stilde = m \sss, 
\qquad 
\thetatilde = m \theta.
\end{align}
We also define corresponding scaled linear operators
\begin{equation} 
\begin{aligned}
\Lcal_{\gtilde} : = \Delta_{\gtilde} + 2m^{-2} = m^{-2} \Lcalp, 
\qquad 
\Lcal_{\chitilde}: = \Delta_{\chitilde}+2m^{-2}\sech^2 \sss = m^{-2}\Lcal_{\chi}. 
\end{aligned}
\end{equation}

For future reference, we record the following global parametrization of a catenoid with unit waist size:
\begin{equation}
\label{Ecatenoid}
\Xhat_{cat}(\sss, \theta) = \left( \cosh \sss \cos \theta, \cosh\sss  \sin \theta, \sss\right).
\end{equation}
For $\tau>0$,  $\tau \Xhat_{cat}$ parametrizes the catenoid with waist size $\tau$.

\begin{definition}
\label{dchi}
Let $\sbar \in \R_+$.  We define a shifted coordinate $\shat = \shat\, [\sssunder]$ by 
\[\shat := \stilde - m \sbar.\] 
\end{definition}
 
\begin{definition}
\label{dOmega}
For convenience we define 
\begin{align}
\label{Edelta}
\delta := {1}/(9m).
 \end{align}
Given $\sbar \in [0, \infty)$, we have 
nested open sets $D^\chi_{L[\sbar;m]}(3\delta) \subset \Omega'[\sbar;m]\subset \Omega[\sbar;m]$ where 
 \begin{align*}
 \Omega[\sbar; m] : = D^\chi_{L_{par}[\sbar]}\left(3/m\right),\qquad
 \Omega'[\sbar; m] := 
 D^\chi_{L_{par}[\sbar]}\left(2/m\right).
 \end{align*}
 \begin{figure}[h]
\centering
\includegraphics[scale=.7]{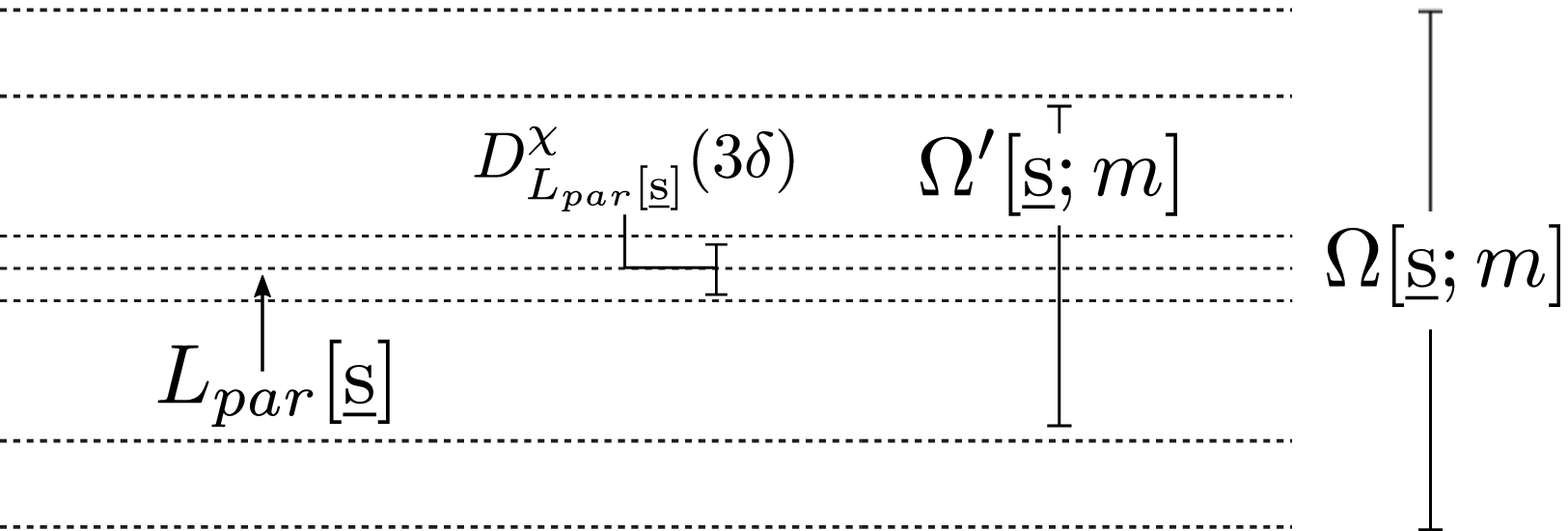}
\caption{A schematic of connected components of the neighborhoods of $\Lpar[\sbar]$ (defined in \ref{dOmega}) near latitude $\sssunder$.}
\end{figure}
\end{definition}

\begin{definition}[Antisymmetry operators] 
\label{Dantisym} 
Given a domain $\Omega\subset\cyl$ which is invariant under the coordinate reflection through some $\sbar\in\R$, 
we define a reflection operator $\Rcal_{\sbar}: C^0\left(\Omega\right)\rightarrow C^0\left(\Omega\right)$ 
by 
\[ \Rcal_{\sbar} u(\sbar+\sss', \theta) = u(\sbar- \sss', \theta), \quad \text{for} \quad (\sbar+ \sss', \theta) \in \Omega \]
and an antisymmetry operator $\Acal_{\sbar}: C^0\left(\Omega\right)\rightarrow C^0\left(\Omega\right)$ by 
\[ \Acal_{\sbar} u = u - \Rcal_{\sbar} u.\] 
\end{definition}

\begin{lemma}
\label{Rprod}
Let $\sbar \in \left(\frac{3}{m}, \infty\right)$.  The following hold:
\begin{enumerate}[label=(\roman*).]
\item  For all $u, v \in C^0_\sym(\Omega[\sbar; m])$,  
\[ \Acal_{\sbar}\left[ u v\right] = u \, \Acal_{\sbar} [v] + \Acal_{\sbar} [u] \, \Rcal_{\sbar} [v].\] 

\item  For all $u \in C^2_\sym(\Omega[\sbar; m])$,
\begin{align*}
\Acal_{\sbar}\left[ \Lcal_\chi u\right] &= \Lcal_\chi\left[ \Acal_{\sbar} u\right] + 2\Acal_{\sbar} [\sech^2 \sss] \, \Rcal_{\sbar} [u], \\
\Acal_{\sbar}\left[ \Lcal_\chitilde u\right] &= \Lcal_\chitilde\left[ \Acal_{\sbar} u\right] + 2m^{-2}\Acal_{\sbar} [\sech^2 \sss] \, \Rcal_{\sbar} [u].
\end{align*}

\item 
 \begin{enumerate}[label=(\alph*)]
 \item  Let $\Omega \subset \cyl$ be a domain.  Then 
$\left\| \sech^2 \sss : C^j\left( \Omega , \chi\right) \right\| \le C(j) \left \| \sech^2 \sss : C^0\left(\Omega, \chi \right)\right\|$. 
\item $\left\| \Acal_{\sbar} \left[ \sech^2 \sss \right] : C^j\left( \Omega[\sbar; m], \chi \, \right)\right\| \le \frac{C(j)}{m}\left\| \sech^2 \sss : C^0\left( \Omega[\sbar; m], \chi \, \right)\right\|$.
\end{enumerate}
\end{enumerate}
\end{lemma}
\begin{proof}
(i) follows from a straightforward computation, and (ii) follows from (i) and a similar computation, using the fact that $\Delta$ commutes with $\Acal_{\sbar}$.  

For (iii).(a), observe that for each $j\geq 1$,  $\partial^j\left( \sech^2\sss \right)$ is a polynomial expression in $\sech^2 \sss$ and $\tanh \sss$ each term of which contains a factor of $\sech^2 \sss$.  (iii).(b) is a discrete version of (iii).(a) which follows from the mean value theorem.  More precisely, fix $\sssunder\in \R_+$ and let $\shat \in (-3, 3)$, where $\shat = \shat\, [\sssunder]$ is as in \ref{dchi}.   By the mean value theorem, there exists $\sss' \in (-\shat , \shat)$ such that
\begin{align*}
\Acal_{\sbar}[ \sech^2 \sss]\left( \sssunder +\frac{ \shat}{m}\right) &= \sech^2\left( \sssunder +\frac{ \shat}{m}\right) - \sech^2\left( \sssunder  - \frac{ \shat}{m}\right)\\
&= \frac{4 \, \shat}{m} \sech^2\left( \sssunder +\frac{ \sss'}{m}\right)\tanh\left( \sssunder +\frac{ \sss'}{m}\right).
\end{align*}
Estimating the right hand side of the preceding, it follows that
\[
\left\| \Acal_{\sbar} \left[ \sech^2 \sss \right] : C^0\left( \Omega[\sbar; m], \chi\, \right)\right\| \le 
\frac{C}{m} \left\| \sech^2 \sss  : C^0\left( \Omega[\sbar; m], \chi\, \right)\right\|.
\] 
Estimates on higher order derivatives follow from the mean value theorem in a similar way.
\end{proof}

\subsection*{Rotationally invariant functions}
\label{sub:rot}
$\phantom{ab}$

We call a function defined on a domain of $\cyl$ which depends only on the coordinate $\sss$ a  rotationally invariant function.  
The linearized equation $\Lcal_\chi \phi = 0$ amounts to an ODE when $\phi$ is rotationally invariant. 
Motivated by this, we introduce some notation to simplify the presentation:  

\begin{notation}
\label{Nsym}
Consider a function space $X$ consisting of functions defined on a domain
$\Omega\subset \Spheq$.
If $\Omega$ is invariant under the action of $\grouptwo$ (recall \ref{Dgroup}),
we use a subscript ``$\sym$'' to denote the subspace $X_\sym\subset X$
consisting of those functions  $f\in X$ which are covariant under the action of $\grouptwo$.
If $\Omega$ is a union of parallel circles, we use a subscript ``$\sss$'' to denote 
the subspace of functions $X_{\sss}$ consisting of 
rotationally invariant functions, which therefore depend only on $\sss$.
If moreover $\Omega$ is invariant under reflection with respect to $\{\sss = 0\}$,
we use a subscript ``$|\sss|$'' to denote 
the subspace of functions $X_{|\sss|}= X_\sss \cap X_\sym$
consisting of those functions which depend only on $|\sss|$.
\end{notation}

For example, we have
$C^0_{|\sss|}(\cyl)\subset C^0_\sym(\cyl) \subset C^0(\cyl)$
and $C^0_{|\sss|}(\cyl) \subset C^0_\sss(\cyl)$,
but $C^0_\sss(\cyl)$ is not a subset of $C^0_\sym(\cyl)$.
We also have the following. 

\begin{definition}
\label{davg}
Given a function $\varphi$ on some domain $\Omega\subset \Spheq$, we define
a rotationally invariant function $\varphi_\ave$ on 
the union $\Omega'$ of the parallel circles on which $\varphi$ is integrable
(whether contained in $\Omega$ or not),
by requesting that on each such circle $C$, 
\[\left. \varphi_\ave \right|_C
:=\avg_C\varphi.\]
We also define $\varphi_\osc$ on $\Omega\cap\Omega'$ by
\[ \varphi_\osc:=\varphi-\varphi_\ave.\]
\end{definition}

\begin{notation}
\label{Npartial}
If $\Omega \subset \cyl$ is a domain and $u \in C^0_\sss(\Omega)$ has one-sided partial derivatives at $\sss=\sbar$, 
then we denote these partial derivatives by using the notation
$$
\partial_{+\,}u(\sbar) :=\left. \frac{\partial u  }{\partial \sss}\right|_{\sss=\sbar+},
\qquad\qquad
\partial_{-\,} u(\sbar) :=-\left. \frac{\partial u  }{\partial \sss}\right|_{\sss=\sbar-}.
$$
If $u$ is $C^1$, we use the abbreviation $\partialx u := \frac{\partial u}{\partial \sss}$.  In that case, $\partial u = \partial_{+}u = -\partial_{-}u$.  
\end{notation}

If $\phi \in C^\infty_{\sss}(\cyl)$, the equation $\Lcal_\chi \phi  =0$ (recall \ref{ELchi}) is equivalent to 
\begin{align}
\label{ELchirot}
\frac{d^2\phi}{d\sss^2} + 2\sech^2 \sss \, \phi = 0.
\end{align}
\begin{definition}
\label{dphie}
Define rotationally invariant functions $\phie \in C^\infty_{|\sss|}(\cyl)$ and $\phio \in C^\infty_\sss(\cyl)$ by
\begin{align}
\phie(\sss) = 1-\sss \tanh \sss, \quad \phio(\sss) = \tanh \sss.
\end{align}
\end{definition}

\begin{lemma}
\label{Lphie}
$\phie$ and $\phio$ are even and odd in $\sss$ respectively and satisfy $\Lcal_\chi \phie = 0$ and $\Lcal_\chi \phio = 0$.  $\phie$ is strictly decreasing on $[0, \infty)$ and has a unique root $\sss_{root}\in (0, \infty)$.  $\phio$ is strictly increasing.  The Wronskian $W[ \phie, \phio]$ satisfies
\[ W[\phie, \phio](\sss): = \phie(\sss)\partial \phio(\sss) - \partial \phie(\sss) \phio(\sss) = 1.\]
\end{lemma}
\begin{proof}
Straightforward calculation using Definition \ref{dphie} and \eqref{ELchirot}.  
\end{proof}

\begin{remark}
When written in $\xx$ coordinates (recall \eqref{Exs}), $\phie$ and  $\phio$ satisfy
\begin{align}
\label{Ephiex}
\phio=\sin\xx,
\qquad
\phie=
1- \sin\xx \, \log\frac{1+\sin\xx}{\cos\xx}, 
\end{align}
and therefore our Definition 
\ref{dphie} is equivalent to \cite[Definition 2.18]{kap}. 
\end{remark}

It will be helpful to introduce the following auxiliary ODE solutions:
\begin{definition}
\label{dauxode}
Given $\atilde,\btilde\in\R$, and $\sbar \in \R_+$ we define
\begin{equation*}
\begin{aligned}
\phiunder=\phiunder\left [\atilde,\btilde;\sbar\right]\in \,\,
&
C^\infty_\sss\left( \, \{\sss\in[0,\infty)\} \,\right )
\bigcap
C^0_{|\sss|}( \cyl  ),
\\
\junder=\junder\left[\btilde;\sbar\right]\in \,\,
&
C^\infty_\sss\left( \{\sss\in[\sbar,\infty)\} \,\right)
\bigcap 
C^\infty_\sss\left(  \{\sss\in(0,\sbar]\} \,\right)
\bigcap 
C^0_{|\sss|}( \cyl),
\end{aligned}
\end{equation*}
by requesting they satisfy the initial data 
$$
\phiunder(\sbar)=\atilde,
\qquad
 \partial \phiunder (\sbar) =  \btilde ,
\qquad
\junder(\sbar)=0,
\qquad
\partial_{+}\junder(\sbar)=\partial_{-}\junder(\sbar)=m \btilde,
$$
and the ODEs $\Lcal_{\chitilde}\phiunder=0$ on
$\{\sss \in[0,\infty)\}$, 
and $\Lcal_{\chitilde}\junder=0$ on $\{\sss\in[\sbar,\infty)\}$ and on
$\{\sss\in[0,\sbar]\}$.
\end{definition}

\begin{remark}
\label{Rauxode}
Note that $\phiunder$ depends linearly on the pair $(\atilde, \btilde)\in \R^2$ and $\junder$ depends linearly on $\btilde\in \R$.  
\end{remark}

\begin{lemma}
\label{Lode}
Let $\sbar \in \left(\frac{3}{m}, \infty\right)$. The following estimates hold (recall \ref{Nsym}).  
\begin{enumerate}[label=(\roman*).]
\item
$\left\|  \, \phiunder[1, 0 ;\sbar] -1 \, : \, C_\sym^{j}( \, \Omega[\sbar; m]\,,\chitilde \, )\, \right\|\, \le \, C(j)/m^2$.

\item
$\left\|  \, \junder[1;\sbar] - |\, \shatbar \, |  \, : \, C_\sym^{j}( \, \Omega[\sbar; m] \setminus \Lpar[\sbar] \,,\chitilde \, ) \, \right\| \, \le \,  C(j)/m^2\, $.

\item
$\left\| \, \Acal_{\sbar} \,   \phiunder[1, 0 ;\sbar] \, : C_\sym^{j}( \, \Omega[\sbar; m]\, , \chitilde \, ) \, \right\|\le \, C(j)/m^3$.

\item
$\left\|  \, \Acal_{\sbar} \,   \junder[1;\sbar] \, : \, C_\sym^{j}( \, \Omega[\sbar; m] \setminus \Lpar[\sbar]  \,,\chitilde \, ) \, \right\| \, \le  \,  C(j)/m^3 $.
\end{enumerate}
\end{lemma}
\begin{proof}
Denote in this proof $\phiunder = \phiunder[1, 0; \sbar]$.  In $\shat$-coordinates (where $\shat = \shat\, [ \sssunder]$ is as in \ref{dchi}), the equation $\Lcal_{\chitilde} \phiunder =0$  is equivalent to
\begin{align}
\label{Ephiode}
\partial^2_{\, \shat} \, \phiunder + \frac{2}{m^2} \sech^2\left( \frac{\shat}{m}+ \sbar\right) \phiunder = 0. 
\end{align}
Since $\sech^2t$ is decreasing on $(0, \infty)$, it is easy to see that for $m\geq 6$, $\phiunder>0$ on $\Omega[\sbar; m]$, hence 
\[ \partial^2_{\, \shat}\,  \phiunder + \frac{2}{m^2} \phiunder > 0 \quad \text{on} \quad \Omega[\sbar; m].\] 
Integrating this differential inequality and matching the initial data implies
\begin{align*}
\left\| 1 - \phiunder : C^0_\sym\left(\Omega[ \sbar;m], \chitilde\,\right)\right\| 
\le \left\| 1 -  \cos \left( \frac{\sqrt{2}}{m}\, \shat \right) : C^0_\sym\left(\Omega[ \sbar;m], \chitilde\,\right)\right\| \le \frac{C}{m^2}.
\end{align*}
In conjunction with \eqref{Ephiode}, the $C^0$ bound above implies $\left|  \partial^2_{\, \shat} \, \phiunder\right| < C/m^2$ on $\Omega[\sbar, m]$.  Integrating this bound with respect to $\shat$ implies $\left| \partial_{\, \shat}\,  \phiunder \right|< C/m^2$.  Together, these bounds imply the $C^2$ bound in (i).  Higher derivative bounds follow from the $C^2$ bound by differentiating \eqref{Ephiode}.  

For (ii), denote $\junder = \junder[1; \sbar]$.  Arguing as above, we have on $\Omega[\sbar;m]\setminus \Lpar[\sbar]$
\begin{align}
\label{Ejode1}
\partial^2_{\, \shat} \, \junder + \frac{2}{m^2} \sech^2\left( \frac{\shat}{m}+ \sbar\right) \junder = 0 \quad \text{and} \quad
 \partial^2_{\, \shat} \, \junder + \frac{2}{m^2} \junder > 0.
\end{align}
A similar comparison argument establishes that
\begin{align*}
\left\| \junder -|\, \shat\, | : C^0_\sym ( \Omega[ \sbar; m]\setminus \Lpar[\sbar] , \chitilde\,)\right\|
&\le  \left\|  \frac{m}{\sqrt{2}} \sin \left( \frac{\sqrt{2}}{m}\, | \,\shat\, | \right)-\left|\, \shat\, \right| : C^0_\sym\left( \Omega[ \sbar; m]\setminus \Lpar[\sss_i] , \chitilde\, \right)\right\|\\
&\le \frac{C}{m^2}.
\end{align*}
The $C^0$ bound above implies $\left| \partial^2_{\, \shat } \, \junder \right| < C/m^2$.
For $t\in (-3, 3)$, we find by integrating this bound that
\begin{align*}
\left| \partial_{\, \shat} \junder (\, \shat\, ) - \partial_{\, \shat} | \, \shat\,  | \right|
= \left| \partial_{\, \shat} \junder (t) -  \partial_{\, \shat} \junder (0)\right| 
= \left| \int_{0}^{t} \partial^2_{\, \shat} \, \junder \, d\, \shat\, \right| \le \frac{C}{m^2}.
\end{align*}
This with the preceding bounds implies the $C^2$ bound in (ii).  Estimates on the higher derivatives follow by differentiating \eqref{Ejode1}.

 By Lemma \ref{Rprod}.(ii), $\Acal_{\sbar} \, \phiunder_i$ satisfies the equation
\begin{align}
\label{Eaphi}
 \partial^2_{ \, \shat}\, \Acal_{\sbar}\,  \phiunder+ \frac{2}{m^2} \sech^2 \left( \frac{\shat}{m}+ \sbar\right) \Acal_{\sbar}\,  \phiunder+ \frac{2}{m^2} \Acal_{\sbar}\left[ \sech^2 \sss\right] \left( \frac{\shat}{m}+ \sss_i\right) \Rcal_{\sbar} \, \phiunder = 0.
\end{align}
It follows from (i) that 
$\left\|  \, \Acal_{\sbar}\,  \phiunder \, : \, C_\sym^{2}( \, \Omega[\sbar; m]\,,\chitilde \, )\,\right \|\, \le \, C/m^2$.  Using this bound and the estimate on $\Acal_{\sbar}\left[ \sech^2 \sss\right]$ from Lemma \ref{Rprod}.(iii) in \eqref{Eaphi}, we find $ \left| \partial^2_{ \, \shat} \, \Acal_{\sbar}\,  \phiunder\right| < C/m^3$ on $\Omega[\sbar;m]$.  Integrating this bound twice with respect to $\shat$ twice and using the fact that $\partial_{\, \shat} \Acal_{\sbar} \,  \phiunder(\sbar) = \Acal_{\sbar} \, \phiunder(\sbar) =0$ establishes the $C^2$ bounds in (iii); as before, estimates on the higher derivatives follow from differentiating \eqref{Eaphi}.  The proof of (iv) is almost exactly the same as the proof of (iii), so we omit it.
\end{proof}


\subsection*{Green's functions and the Green's function for $\Lcal_\chi$}
$\phantom{ab}$
\nopagebreak 

\begin{definition}
\label{Dgreengeneral}
Let $(\Sigma^2, g)$ be Riemannian, $V\in C^\infty(\Sigma)$, and $p\in \Sigma$.  If there exists a domain $\Omega \subset \Sigma$ containing $p$ and $G_p \in C^\infty \left( \Omega \setminus \{p\} \right)$ satisfying 
\begin{enumerate}[label=(\roman*).]
\item $\left( \Delta_g + V\right) G_p= 0$ 
\item  For some neighborhood $\Omega'\subset\Omega$, 
$G_p-\log\circ \dbold^g_p$ is bounded on $\Omega'\setminus \{p\}$,
\end{enumerate}
we say  $G_p$ is a \emph{Green's function for $\Delta_g + V$ centered at $p$}. 
\end{definition}

For any $(\Sigma, g)$, $p$, and $V$ as above, standard theory guarantees the existence of a Green's function for $ \Delta_g + V$ with center at $p$.

\begin{lemma}
\label{Lremovablesing}
Suppose $G_p$ and $\Gtilde_p$ are Green's functions for $\Delta_g + V$, where  $p, \Omega, V$, and $(\Sigma, g)$ are as in Definition \ref{Dgreengeneral}.  Then $G_p - \Gtilde_p$ has a unique extension to $C^\infty(\Omega)$.
\end{lemma}
\begin{proof}
In this proof, we denote $\Lcal = \Delta_g + V$.
Definition \ref{Dgreengeneral}.(ii) implies that $u: = G_p - \Gtilde_p$ satisfies $ \left\|  u : C^0\left( \Omega \setminus \{ p \} , g\right) \right\| \le C$.  Since $\Lcal u = 0$, the $C^0$ bound and Schauder estimates imply that
\begin{align}
\label{Eremovesing}
 \left\| u : C^j\left( \Omega \setminus \{ p \} , g\right) \right\| \le C(j). 
 \end{align}

By standard regularity theory, to prove the existence of the extension,  it suffices to prove that $u$ solves $\Lcal u = 0$ weakly in $\Omega$.  To this end, let $\varphi \in C^\infty_c(\Omega)$.  Compute
\begin{align*}
\int_{\Omega} u \Lcal \varphi  \, d\mu_g 
&=  \lim_{\epsilon \searrow 0} \int_{\Omega\setminus D^g_p(\epsilon)} \!\varphi \Lcal u \, d\mu_g + \lim_{\epsilon \searrow 0}\int_{\partial D^g_p(\epsilon)} \! u \frac{\partial \varphi}{\partial \eta} - \varphi \frac{\partial u }{ \partial \eta} \, d\mu_g
= 0,
\end{align*}
where the last equality follows from \eqref{Eremovesing} and that  $\Lcal u = 0$.
\end{proof}

\begin{convention}
\label{Cgreen}
Given two Green's functions $G_p, \Gtilde_p$ as in \ref{Lremovablesing}, we abuse notation by writing $G_p -\Gtilde_p$ for the smooth extension to $C^\infty(\Omega)$ of  $G_p - \Gtilde_p$, whose existence and uniqueness is asserted in Lemma \ref{Lremovablesing}.
\end{convention}

For future reference, we recall \cite[Lemma 2.20]{kap} basic properties of the Green's function used in \cite{kap} to construct initial surfaces. 
\begin{lemma}[{\cite[Lemma 2.20]{kap}}]
\label{LGp}
There is a function 
$G^{\Sph^2}\in C^\infty((0,\pi))$
uniquely characterized by (i) and (ii) 
and moreover satisfying (iii-vii) below.
We denote by $\rr$ the standard coordinate of $\R^+$.
\begin{enumerate}[label=(\roman*).]
\item For small $\rr$ we have $G^{\Sph^2}(\rr)=(1+O(\rr^2))\,\log \rr$.

\item
For each $p\in\Spheq$
we have $\Lcalp \left(G^{\Sph^2} \circ  \dbold^g_p\right) =0$
where
$G^{\Sph^2}\circ  \dbold^g_p\in C^\infty\left(\Spheq\setminus\{p,-p\}\right)$
(recall \ref{NT}).

\item
$G^{\Sph^2} \circ \dbold^g_{p_N}= (\log 2-1) \, \phio+\phie$ (recall \ref{Eeqone}).

\item
$G^{\Sph^2}(r)\, = \, 1+ \, \cos r \, \left(\,-1+\log\frac{2\sin r }{1+\cos r}\,\right)$. 

\item
$\frac{\partial G^{\Sph^2}}{\partial r} (r) \, = \, -\sin r \, \log\frac{2\sin r }{1+\cos r} 
+ \frac1{\sin r} + \frac{\sin r \cos r }{1+\cos r} $.  

\item
$
\left\|\, 
G^{\Sph^2} - 
\,\cos r \, \log \rr
\,
: C^{j}\left(\,
(0,1)\,,\,
r, dr^2,r^{2}\,\right)\,\right\|
\le 
\, C(j) \, .
$ 

\item
$
\left\|\, 
G^{\Sph^2} 
\,
: C^{j}\left(\,
(0,1)\,,\,
r, dr^2,\, |\log r| \,\right)\,\right\|
\le 
\, C(j) \, .
$ 
\end{enumerate}
\end{lemma}

It will also be helpful to have Green's functions $G_{p}$ for $\Lcal_\chi$ well adapted to the cylinder.
\begin{lemma}
\label{Lgreen}
Given $p = (\sbar, \underline{\theta})\in \cyl$, there exists $G_p \in C^{\infty}\left(D^{\chi}_p(\frac{1}{2})\setminus \{p\}\right)$ satisfying: 
\begin{enumerate}[label=(\roman*).]
\item $\Lcal_\chi G_p = 0$ on $D^\chi_p(\frac{1}{2})\setminus \{p\}$. 

\item For $q$ near $p$, $G_p(q) = \log r + O(r^2|  \log r|)$, where $r(q) = \dbold^{\chi}_p(q)$.

\item $\left\| \Acal_{\sbar} G_{p} : C^{j} \left(D^{\chi}_p(\frac{1}{2}) \setminus \{ p \}, r ,\chi ,r \right)\right\| \le C(j).$ 
\end{enumerate}
\end{lemma}
\begin{proof}
As an auxiliary step, we consider solutions of the equation
\begin{align}
\label{Elaplaceode}
\Delta_\chi u + 2\left(\sech^2\sbar\right) u = 0.
\end{align}
Let  $r = \dbold^\chi_p$ be the polar radius on $D^\chi_p(\frac{1}{2})$.  When $u = u(r)$ is radial, \eqref{Elaplaceode} is equivalent to the ODE
\begin{align}
\label{Eoderot}
\frac{d^2 u}{d r^2} + \frac{1}{r} \frac{d  u}{d r} + 2\left(\sech^2\sbar\right) \, u = 0.
\end{align}
The solution space to \eqref{Eoderot} is spanned by
\[  \left\{ J_0\left(\sqrt{2}\left(\sech\sbar\right)  r\right),  Y_0\left(\sqrt{2}\left(\sech\sbar\right)  r\right)\right\}, \] 
where $J_0$ and $Y_0$ are the Bessel functions (cf. \cite[Section 9.1]{abramowitz}) defined by 
\begin{equation}
\label{Ebessel}
\begin{aligned}
J_0(x) &= \sum_{j = 0}^\infty \frac{(-1)^j}{4^j (j!)^2} x^{2j}, \\ 
Y_0(x) &= \frac{2}{\pi}\left\{ \left(\log \frac{x}{2}+ \gamma\right) J_0(x) + \sum_{j=1}^\infty (-1)^{j+1} \left(\sum_{l=1}^j\frac{1}{l} \right)\frac{x^{2j}}{4^j (j!)^2}\right\},
\end{aligned}
\end{equation}
and $\gamma$ is the Euler constant.  
By a short computation, the function $G'_p \in C^\infty\left( D^\chi_p(\frac{1}{2}) \setminus \{ p \}\right)$ defined by
\begin{align}
\label{DG'}
G'_p :=  \frac{\pi}{2} Y_0\left( \sqrt{2}\left(\sech\sbar\right)  r\right) - \left(\gamma+\log \frac{\sech\sbar}{\sqrt{2}}\right) J_0\left(\sqrt{2}\left(\sech\sbar\right)  r\right)
\end{align}
satisfies
\begin{align}
\label{EG'}
\Delta_\chi G'_p + 2\left(\sech^2 \sbar\right) \, G'_p =0 \quad \text{and} \quad G'_p= \log r + O(r^2| \log r|).
\end{align}
Let $w_p \in C^{2, \beta}(D^\chi_p(\frac{1}{2}))$ be the unique solution of the Dirichlet problem
\begin{align}
\label{Ew}
\begin{cases}
\Lcal_\chi w_p &= 2(\sech^2 \sbar - \sech^2 \sss)G'_p  \hfill \quad\text{on} \quad D^\chi_p(\frac{1}{2})\\
w_p &= 0 \hfill \text{on} \quad \partial D^\chi_p(\frac{1}{2})
\end{cases}
\end{align}
and define (recall Definition \ref{dauxode})
\begin{align}
\label{Ew'}
w'_p: = w_p + \phiunder\left[ -w_p(p), \frac{\partial w_p}{\partial \sss}( p); \sbar\right].
\end{align}
Finally, define
\begin{align}
\label{EGdecomp}
G_p = G'_p + w'_p.
\end{align}
By combining \eqref{EG'}, \eqref{Ew}, and \eqref{Ew'} we find $\Lcal_\chi G_p = 0$, which yields (i). 

Now note that the right hand side of \eqref{Ew} is in $C^{0, \beta}(D^\chi_p(\frac{1}{2}))$ for any $0< \beta < 1$.  It follows that $w'_p \in C^{2, \beta}(D^\chi_p(\frac{1}{2}))$.  In light of \eqref{Ew}, $w'_p(p) = \left.\nabla w'_p\right|_{p}=0$.  Therefore, the second order Taylor series for $w'_p$ and the Schauder estimates imply
\begin{align}
\label{Ew'est}
\left\| w'_p : C^j(D^\chi_p(\textstyle{\frac{1}{2}}))\setminus \{ p\}, r, \chi, r)\right\| \le C(j).
\end{align}

(ii) then follows from combining \eqref{EG'}, \eqref{EGdecomp}, and \eqref{Ew'est}. 

Note from \eqref{DG'} that $\Acal_{\sbar} G_p =  \Acal_{\sbar} w'_p$.  (iii) then follows from this and \eqref{Ew'est}.
\end{proof}

The estimates in Section \ref{S:LD}, and in particular the decomposition in \ref{dPprime}, are adapted to the Green's functions $G_p$, for $p\in L$.  
Since for such $p$, $\varphihat_p$ is defined (recall  \ref{dLD0}.(ii)) in terms of the Green's function $G^{\Sph^2}\circ \dbold^g_p$, the following lemma will be useful.
\begin{lemma}
\label{Lgdiff}
 Fix $\sbar \in \R$ and let $p = (\sbar, 0)\in \cyl$.  We have (recall Convention \ref{Cgreen})
\begin{align}
\label{Egdiff}
\left( G_{p} - G^{\Sph^2}\circ \dbold^g_{p} \right)(p) = -\log \sech\sbar,  \quad 
d_{p} \left(G_{p} - G^{\Sph^2}\circ \dbold^g_{p} \right)(p) = \frac{1}{2}\tanh\sbar\,   d\sss. 
\end{align}
\end{lemma}
\begin{proof}
Let $q = (\sss, 0)$, where $\sss$ is close to $\sbar$.  For convenience, in this proof denote $r = \dbold^\chi_{p}(q)$.  Recalling \eqref{Echig2}, we have
\begin{multline*} 
\dbold^g_{p}(q) = \left| \int_{\sbar}^{\sss} \sech t dt\right| 
= \left| \sech \sbar (\sss- \sbar) + \frac{1}{2}(\sss - \sbar)^2 ( - \tanh \sbar \sech\sbar) + O\left( \sech\sbar(\sss-\sbar)^3\right)\right|\\
=(  \sech \sbar )r \left( 1- \frac{1}{2}(\sss- \sbar)\tanh \sbar + O((\sss - \sbar)^2)\right).
\end{multline*} 

Consequently, recalling from \ref{LGp}.(i) that $G^{\Sph^2}(t) = \log t + O(t^2 |\log t|)$ for small positive $t$, we have
\begin{align*}
G^{\Sph^2}\circ \dbold^g_{p} (q) = \log \sech \sbar + \log r - \frac{1}{2} (\sss-\sbar)\tanh\sbar + O(r^2| \log r|).
\end{align*}

By Lemma \ref{Lgreen}.(ii),  $G_{p}(q)= \log r + O(r^2|\log r|)$ and hence
\begin{align}
\label{Ediffexpansion}
\left( G_{p} - G^{\Sph^2}\circ \dbold^g_{p} \right)(q) =- \log \sech \sbar + \frac{1}{2}(\sss-\sbar)\tanh\sbar + O(r^2 |\log r|).
\end{align}
The conclusion then easily follows from \eqref{Ediffexpansion}.
\end{proof}


\section{Rotationally invariant solutions}
\label{S:RLD}

\subsection*{Basic facts and notation} 
$\phantom{ab}$
\nopagebreak

We will estimate our LD solutions by comparing with rotationally invariant solutions.  It will therefore be useful to define a class of rotationally invariant solutions of the linearized equation.  We begin with some notation.

\begin{definition}
\label{dsigma}
Let $\R^\N = \left\{ (a_i)_{i\in \N}: a_i \in \R\right\}$ and $\R_+^\N = \left\{(a_i)_{i\in \N}: a_i\in \R, a_i> 0\right\}$.  

For any $k\in \N$, we identify $\R^{k}$ with a subspace of $\R^\N$ by the map
\[ 
(a_1, \dots, a_{k}) \mapsto (a_1, \dots, a_{k}, 0, 0, \dots).
\]
We consider the normed space $\left(\ell^1(\R^{\N}), |\cdot |_{\ell^1}\right)$ defined by
\begin{align*}
\ell^1(\R^\N) = \left\{ \abold = (a_i)_{i\in \N} \in \R^\N : \sum_{i=1}^{\infty} |a_i|< \infty \right \}, \quad
|\abold |_{\ell^1} = \sum_{i=1}^{\infty} |a_i|.
\end{align*}

\end{definition} 
\begin{remark}
\label{rsigbij}
If $\bsigma = (\sigma_i)_{i\in \N}\in \ell^{1}\left( \R^{\N}\right)$ and $\xibold = (\xi_i)_{i\in \N} \in \ell^{\infty}\left( \R^{\N} \right)$ satisfies $\left| \xibold \right|_{\ell^\infty} < \frac{1}{10}$ and satisfy
\begin{align*}
e^{\sigma_i} = \frac{F_{i+1+} + F_{i+1-}}{F_{i+} + F_{i-}}, \qquad \xi_i = \frac{F_{i+} - F_{i-}}{F_{i+}+F_{i-}}\quad i\in \N,
\end{align*}
for some positive numbers $F_{i\pm}$, $i\in \N$, then note that for any $1 \leq j \le i < \infty$, 
\begin{align*}
F_{i+} = \frac{1  + \xi_i }{1   + \xi_j} \left( e^{\sum_{l=j}^{i-1}\sigma_l} \right) F_{j+} 
&= \frac{1  + \xi_i }{1   - \xi_j} \left( e^{\sum_{l=j}^{i-1}\sigma_l} \right) F_{j-},
\\ 
F_{i-} = \frac{1  - \xi_i }{1   + \xi_j} \left( e^{\sum_{l=j}^{i-1}\sigma_l} \right) F_{j+} 
&= \frac{1  - \xi_i }{1   - \xi_j} \left( e^{\sum_{l=j}^{i-1}\sigma_l} \right) F_{j-},
\end{align*}
and therefore 
\begin{align*}
\sup\{ F_{i\pm} \}_{i\in \N}  \sim_{ \exp{(|\bsigma|_{\ell^1}+3|\xibold|_{\ell^\infty})}} \inf\{ F_{i\pm} \}_{i\in \N}  .
\end{align*}
\end{remark}

\subsection*{RLD solutions and the scale invariant flux}
$\phantom{ab}$
\nopagebreak

\begin{definition}[Scale invariant flux]
\label{dF}
If  $\Omega = \{\sss \in (a, b)\} \subset \cyl$, $\phi \in C^0_\sss(\Omega)$, and  $\phi$ is piecewise smooth and positive on $\Omega$, 
we define $F^\phi_{\pm}:(a, b)\rightarrow \R$ by (recall \ref{Npartial}) 
\[ F^\phi_{\pm}(\sss) = \frac{\partial_{\pm} \phi(\sss)}{\phi(\sss)}. \] 
\end{definition} 

\begin{remark}
\label{rFlux}
Note that $F^\phi_\pm = F^{c\phi}_\pm$  $\forall c\in \R_+$.  Also, if $\phi$ is $C^1$ at $\sss=\sbar$, then $F^\phi_+(\sbar) = - F^\phi_-(\sbar)$.  \end{remark}

\begin{definition}[RLD solutions]
\label{RL}
We say $\phi\in C^0_{|\sss|} \left( \cyl\right)$ is a \emph{rotationally invariant linearized doubling (RLD) solution} if (recall \ref{Dlpar}) 
\begin{enumerate}[label=(\roman*).]
\item  $\phi >0$. 
\item  There is $k\in \N$ and $\sbold^\phi \in \R^k_+$ as in \ref{Dlpar} 
such that 
$\phi \in C^\infty_{|\sss|} ( \cyl^{\sbold^\phi} )$ 
and $\Lcal_\chi \phi =0$ on $\cyl^{\sbold^\phi}$.  
\item  For $i=1, \dots, k$, $F^\phi_-(\sss^\phi_i)> 0$ and $F^\phi_+(\sss^\phi_i)> 0$. 
\end{enumerate}

We call $\sbold^\phi$ the \emph{jump latitudes} of $\phi$ and $L_{par}[\sbold^\phi]$ the \emph{configuration of $\phi$}.  
If $\phi(0) = 1$, we say $\phi$ is a \emph{unit RLD solution}.  
If $\phi$ is extendible to  $C^{\infty}_{|\sss|} \left(\Spheq \setminus L_{par}[\sbold]\right)$ 
we say $\phi$ is \emph{smooth at the poles}.  
\end{definition}

If $\sss^\phi_i$ is a jump latitude of $\phi$, note that \ref{RL}.(iii) implies $\partial\phi$ is not defined at $\Lpar[\sss_i^\phi]$.  
Thus, the jump latitudes of $\phi$ and their number are uniquely determined by $\phi$. 

\begin{definition}[Quantities associated to RLD solutions]
\label{RLquant}
Let $\phi$ be an RLD solution.  
We define
\begin{multline}
\label{dFlist}
\Fboldunder^\phi := \left( F^\phi_{1-}, F^\phi_{1+}, F^\phi_{2-}, \dots, F^\phi_{k+} \right) \in \R^{2k}_+,
\\
\Fbold^\phi := ( F^\phi_i)_{i=1}^k  \in \R^{k}_+,
\quad \bsigma^\phi := (\sigma^\phi_i)_{i=1}^{k-1} \in \R^{k-1}, 
\quad \xibold^\phi := \left( \xi^\phi_i\right)_{i=1}^k \in \R^k,  
\end{multline}
where for $i=1, \dots, k$  and $j=1, \dots, k-1$,
\begin{equation} 
\label{Exi}
F^\phi_{i\pm} := F^\phi_\pm (\sss_i), \quad 
F^\phi_i := F^\phi_+(\sss_i)+F^\phi_{-}(\sss_i), \quad 
e^{\sigma^\phi_j} = \frac{F^\phi_{j+1}}{F^\phi_j}, \quad 
\xi^\phi_i = \frac{F^\phi_{i+} - F^\phi_{i-}}{F^\phi_{i+} + F^\phi_{i-}}. 
\end{equation} 
We define $\bsigmaunder^\phi: = (\bsigma^\phi, \xibold^\phi) \in \R^{k-1}\times \R^k$ and call the entries of $\bsigmaunder^\phi$ the \emph{flux ratios} of $\phi$.   
If $\bsigmaunder^\phi =  \textbf{0}$ we call $\phi$ \emph{balanced}.        
Finally we define
\begin{align*}
F^\phi_{avg} : = \frac{1}{2k}\left|  \Fbold^\phi\right|_{\ell^1} =  \frac{1}{2k}\left| \Fboldunder^\phi\right|_{\ell^1}.
\end{align*}
\end{definition}

\begin{remark}
\label{rmk: ratios unique}
Using \eqref{Exi} (see also Remark \ref{rsigbij}), we recover $\Fboldunder^\phi$ from $F^\phi_{1}$ and $\bsigmaunder^\phi$: 
\begin{align}
\label{EFxi}
F^\phi_{1\pm} = \frac{1}{2}(1\pm \xi^\phi_1)F^\phi_{1} , 
\qquad F^\phi_{i\pm} =  \frac{1}{2}(1\pm \xi^\phi_i) \left( e^{\sum_{l=1}^{i-1} \sigma^\phi_l} \right)  F^\phi_{1}, 
\qquad i>1.
\end{align}
In Proposition \ref{existence} we construct RLD solutions $\phi$ by prescribing $F^\phi_{1-}$ and $\bsigmaunder^\phi$.
\end{remark}

\subsection*{Existence, uniqueness, and estimates}
$\phantom{ab}$
\nopagebreak

If $\phi$ is an RLD solution, it is immediate from \eqref{ELchirot} that $\partial \phi$ (recall \ref{Npartial}) is decreasing on any domain on which it is smooth.  The scale invariant flux $F^\phi_\pm$ enjoys a similar monotonicity:
\begin{lemma}[Flux monotonicity]
\label{LFmono}
Suppose $\phi \in C^\infty_{\sss}\left(\{\sss \in (a, b)\}\right), \phi> 0$ and $\Lcal_\chi \phi =0$.  For $\sss \in (a, b)$,  
\begin{align} \label{F ODE}
\frac{d F^{\phi}_-}{d\sss}(\sss) = 2\sech^2 \sss + \left(F^\phi_-(\sss)\right)^2>0.
\end{align}

\end{lemma}
\begin{proof}
Let $t\in (a, b)$ and $\Omega = \{ \sss \in (a, t)\}$.  By the divergence theorem and that $\Lcal_\chi \phi =0$, 
\begin{align*}
\int_{\partial \Omega}\left\langle \frac{\nabla \phi}{\phi}, \eta \right\rangle  = \int_\Omega \Delta\left( \log \phi \right) =\int_\Omega \left( \frac{\Delta \phi}{\phi}- \frac{|\nabla \phi|^2}{\phi^2}\right)= -\int_{\Omega}\left( 2\sech^2\sss  +  \frac{|\nabla \phi|^2}{\phi^2}\right).
\end{align*}
Using that $\eta = \frac{\partial}{\partial \sss}$ when $\sss = t$ and differentiating the above with respect to $t$ yields the lemma.  
\end{proof}
\begin{lemma}
\label{Lintid}
Suppose $\phi \in C^\infty_{\sss}\left(\{\sss \in [a, b]\}\right)$, $\phi>0$, and $\Lcal_\chi \phi = 0$. Then
\[ F^\phi_-(b) + F^\phi_+(a) = 2\left( \tanh b - \tanh a\right) + \int_{a}^b \left( F^\phi_-(\sss)\right)^2d\sss.
\]
\end{lemma}
\begin{proof}
Follows directly from integrating \eqref{F ODE} on $(a, b)$. \end{proof}

To study RLD solutions on domains $\{\sss \in (\sss_i, \sss_{i+1})\}$ between successive jumps, we introduce the following auxiliary functions.   
\begin{definition}
\label{dHflux}
Given $\sbar \in (0, \infty)$ and $F> 0$, we define functions
\[ H^+= H^+\left[F; \sbar \right] \in C^{\infty}_\sss \left(\cyl\right),\quad  H^-= H^-\left[F; \sbar\right] \in C^{\infty}_\sss\left(\cyl\right)\] 
by requesting that they satisfy the equations $\Lcal_\chi H^+ = 0$, $\Lcal_\chi H^- =0$ with initial data 
\begin{align*}
H^+(\sbar) = 1, \quad  \quad F^{H^+}_+(\sbar)=F, 
\qquad \qquad 
H^{-}(\sbar) = 1, \quad  \quad F^{H^-}_-(\sbar)=F.
\end{align*}
\end{definition}
\begin{remark}
\label{rH}
By straightforward computations (recall Lemma \ref{Lphie}), 
$H^+[F;\sbar]=
H^-[-F;\sbar]$ and 
$$
H^\pm(\sss) = \left( F_+^{\phio}(\sbar) \mp F\right) \phio(\sbar) \, \phie(\sss) 
+ 
\left( - F_+^{\phie}(\sbar) \pm F\right) \phie(\sbar) \, \phio(\sss). 
$$
Note also that when $\sss \geq 0$, $H^+[F; \sssunder](\sss)  = \phiunder[1, F; \sssunder](\sss)$ (recall \ref{dauxode}).
\end{remark}

\begin{lemma}
\label{LHmono} For any $\sss>\sbar$, $H^+ = H^+[F; \sbar]$ satisfies 
\begin{enumerate}[label=(\roman*).]
\item  $\frac{\partial F^{H^+}_+}{\partial \sbar}(\sss)=\left(2\sech^2 \sbar+F^2\right)\left(\frac{H^+(\sbar)}{H^+(\sss)}\right)^2>0$.  

\item $\frac{\partial F^{H^+}_{+}}{\partial F}(\sss) =\left(\frac{H^+(\sbar)}{H^+(\sss)}\right)^2>0$.
\end{enumerate}
\end{lemma}
\begin{proof}
Applying Lemma \ref{LFmono}  to $H^+[F; \sbar]$ (which we denote below by $H$ for ease of notation) yields
\begin{align}
\label{EHmono}
\frac{d F^{H}_+}{d\sss}(\sss)  + \left(F^H_+(\sss)\right)^2 = - 2 \sech^2 \sss.
\end{align}
Differentiating \eqref{EHmono} with respect to $\sbar$ and switching the order of differentiation gives
\begin{align*}
\frac{\partial}{\partial \sss}\left( \frac{\partial F_+^H}{\partial \sbar}\right)+2  \left(\frac{\partial F_+^H}{\partial \sbar}\right)F_+^H= 0.
\end{align*}
After multiplying through by the integrating factor $H^2$, this is equivalent to 
\begin{align*}
\frac{\partial}{\partial \sss}\left(  \frac{\partial F_+^H}{\partial \sbar} H^2\right) = 0,
\end{align*}
from which we conclude
\[ \frac{\partial F_+^H}{\partial \sbar}(\sss) =\frac{\partial F^H_+}{\partial \sbar}(\sbar)\left(\frac{H(\sbar)}{H(\sss)}\right)^2.\]  
Finally, differentiating both sides of the equation $F^{H^+[F; \sbar]}_+(\sbar) = F$ with respect to $\sbar$ yields 
\[ \frac{\partial F^H_+}{\partial \sbar}(\sbar) = - \frac{\partial F^H_+}{\partial \sss}(\sbar) = 2\sech^2 \sbar+ F^2,\]
where the last equality follows from \ref{LFmono}.  This completes the proof of (i). (ii) follows in an analogous way by differentiating \eqref{EHmono} with respect to $F$ and observing that $\frac{\partial F_+^H}{\partial F}(\sbar) = 1$.
\end{proof}

\begin{lemma}
\label{Lcoeffs}
Suppose 
$a>0$, 
$\Omega = \{ \sss\in (-a, a)\}$, 
$\sbold \in \R^j_+$ is as in \ref{Dlpar} with $\sss_j<a$, 
and 
$\phi$ is a function on $\Omega$ which satisfies $\phi \in C^\infty_{|\sss|}(\Omega^\sbold)$,  
$\phi(0) = 1, \phi>0$, and $\Lcal_\chi \phi = 0$ on $\Omega^\sbold$. 
Finally, suppose that 
$F^\phi_{\pm }(\sss_i) = F_{i\pm}$ for some positive numbers $F_{i\pm}$, $i\in \{1, \dots, j\}.$  
Then
\begin{equation}
\label{Erldcoeff}
\begin{gathered}
\phi=\left\{
\begin{gathered}
A_{0}\phie+B_{0}\phio
\quad\text{on}\quad
\{\sss \in [-\sss_1, \sss_1]\},
\\
A_{1}\phie+B_{1}\phio
\quad\text{on}\quad
\{\sss \in [\sss_1, \sss_2]\},
\\ \dots\\
A_{j}\phie + B_{j}\phio
\quad\text{on}\quad \{\sss\in [\sss_j, a)\},
\end{gathered}
\right., 
\end{gathered}
\end{equation}
where the coefficients $A_{i}, B_{i}$, satisfy the recursive equations 
\begin{equation}
\label{Ecoeffs}
\begin{aligned}
A_{0} = 1  \qquad A_{i} &= A_{i-1}- \phi(\sss_i)(F_{i+}+F_{i-})\phio(\sss_i),\\
B_{0}=0  \qquad B_{i} &= B_{i-1} + \phi(\sss_i)(F_{i+}+F_{i-})\phie(\sss_i) 
\end{aligned}
\quad ( 0< i \leq k) .  
\end{equation}
\end{lemma}

\begin{proof}
Clearly $A_0 =1$ and $B_0 = 0$, since $\phi(0) =1$ and $\phi$ extends evenly across $\sss= 0$.  Now fix $i\in \{1, \dots, j\}$.  By Remark \ref{rH},  $\phi = \phi(\sss_i)H^-\left[ F_{i-}; \sss_i\right]$ on $\{ \sss\in [\sss_{i-1}, \sss_i]\}$ and $\phi = \phi(\sss_i)H^+\left[ F_{i+}; \sss_i\right]$ on $\{\sss \in [\sss_{i} , \sss_{i+1}]\}$ and explicitly,
\begin{align*}
\phi(\sss_i)H^+(\sss) &= \left( F_+^{\phio}(\sss_i)- F_{i+}\right)\phi(\sss_i) \phio(\sss_i)\phie(\sss)+\left( - F_+^{\phie}(\sss_i)+F_{i+}\right)\phi(\sss_i) \phie(\sss_i)\phio(\sss)\\
\phi(\sss_i)H^-(\sss) &= \left( F_+^{\phio}(\sss_i)+ F_{i-}\right)\phi(\sss_i) \phio(\sss_i)\phie(\sss)-\left(  F_+^{\phie}(\sss_i)+F_{i-}\right)\phi(\sss_i) \phie(\sss_i)\phio(\sss).
\end{align*}
Subtracting the second of these equations from the first yields \eqref{Ecoeffs}.
\end{proof}

\begin{prop}[Existence and uniqueness of RLD solutions]
\label{existence}
Given $\sss_1 \in (0, \sss_{root})$ and 
$$ 
\bsigmaunder = (\bsigma, \xibold)  = 
\left( \, (\sigma_i )_{i=1}^{\infty} \, , \, (\xi_j)_{j=1}^\infty \, \right)\in \ell^1\left(\R^\N\right) \oplus \ell^{\infty}\left( \R^\N\right)  
$$
satisfying $|\xibold |_{\ell^\infty} < \frac{1}{10}$,   
there is a unique unit RLD solution 
$\phat = \phat[\sss_1; \bsigmaunder]$ 
satisfying the following. 
\\
(a). 
$\sss_1^\phat=\sss_1$.   
\\
(b). 
$\bsigmaunder^\phat=\left. \bsigmaunder \right|_k$ where 
$k= k[\sss_1; \bsigmaunder]\in \N$ is the number of jump latitudes of $\phat$ (recall \ref{RL})  
and  
$\left. \bsigmaunder \right|_k := \left( \, (\sigma_i )_{i=1}^{k-1} \, , \, (\xi_j)_{j=1}^k \, \right)\in \R^{k-1}\times\R^k$. 

Moreover the following hold.
\begin{enumerate}[label=(\roman*).]
\item  $k[\sss_1; \bsigmaunder]$ is a nonincreasing function of $\sss_1$.  
Further,
for each $\bsigmaunder$ as above 
there exists a decreasing sequence $\{a_{0, \bsigmaunder}:= \sss_{root}, a_{1, \bsigmaunder}, \dots\}$  
such that $k[\sss_1; \bsigmaunder] = k$ if and only if $\sss_1 \in [a_{k, \bsigmaunder}, a_{k-1, \bsigmaunder})$. 
Moreover each $a_{k,\bsigmaunder}$ depends only on 
$\left. \bsigmaunder \right|_k $ (defined as above). 

\item  
$\sss_2^\phat, \dots, \sss_k^\phat$ are increasing smooth functions of $\sss_1$ for fixed $\bsigmaunder$.  

\item
$\phat[a; \bsigmaunder]$ is smooth at the poles if and only if $a = a_{k, \bsigmaunder}$ for some $k\geq 1$. 

\item  The restriction of $\phat[\sss_1; \bsigmaunder]$ on any compact subset of $[0, \infty)$  depends continuously on $\sss_1$ and $\bsigmaunder$.  
\end{enumerate}
\end{prop}

\begin{proof}
We first prove that there is at most one unit RLD solution $\phihat$ satisfying (a) and (b). 
By the symmetries, it follows that $\phihat = \phie$ on $\{ \sss \in [-\sss_1,
\sss_1]\}$.  
Furthermore,
by Remark \ref{rsigbij} and Lemma \ref{Lcoeffs}, $\phihat$ has a unique local
extension beyond $\sss_1$.  
Next, the flux monotonicity---Lemma \ref{LFmono}---and the requirement 
in Definition \ref{RL}.(i) that $\phihat>0$, 
inductively determine all of the jump latitudes $\sbold$ uniquely
using \ref{rsigbij}.  
From this, $\phihat$ is determined uniquely by \ref{Lcoeffs}.

We next construct a family $\phat[\sss_1; \bsigmaunder]$ of RLD solutions, 
parametrized by $\sss_1$ or equivalently $F^{\phat}_1$, with flux ratios $\bsigmaunder$.  
By Lemma \ref{LFmono}, the restriction $\left. F^{\phie}_-\right|_{(0, \sss_{root})}: (0, \sss_{root})\rightarrow (0, \infty)$ is an orientation-preserving homeomorphism.  
Because any unit RLD solution coincides with $\phie$ on $\{\sss \in [0, \sss_1]\}$, 
it follows there is a 1-1 correspondence between choices of 
$\sss_1 \in (0,  a_{0, \bsigmaunder}:= \sss_{root})$ and $F_1 = F_1(\sss_1) : = \frac{2}{1- \xi_1} F^\phie_-(\sss_1)\in (0, \infty)$.  

Let $\sss_1 \in (0, \sss_{root})$.  
By Remark \ref{rsigbij} and Lemma \ref{Lcoeffs}, there is a unique extension $\phat[\sss_1; \bsigmaunder]$ of 
$\left. \phie\right|_{\{\sss\in (-\sss_1, \sss_1)\}}$ to a maximal domain $\{ \sss \in (-a, a) \}$ such that the hypotheses of Lemma \ref{Lcoeffs} hold, 
where the coefficients $F_{i\pm}$ (recall the notation of Lemma \ref{Lcoeffs}) 
are defined by (recall \eqref{EFxi})
\begin{align*}
F_{1\pm} := \frac{1}{2}(1\pm \xi_1)F_{1} , \qquad F_{i\pm} := \frac{1}{2} (1\pm \xi_i) \left( e^{\sum_{l=1}^{i-1} \sigma_l} \right) F_{1}, \quad i>1.
\end{align*}

To show that $\phat$ is an RLD solution, we must show that $a = \infty$ and $\phat$ has finitely many jump latitudes.  
By Remark \ref{rsigbij} and Lemma \ref{Lintid},
\begin{align*}
 2 F_1 & \sim_{ \exp{(|\bsigma|_{\ell^1}+3|\xibold|_{\ell^\infty})}} \left( F_{i+1-} + F_{i+}\right)\\
 & \sim_{ \exp{(|\bsigma|_{\ell^1}+3|\xibold|_{\ell^\infty})}} \left(2\left( \tanh \sss_{i+1} - \tanh \sss_i\right) + \int_{\sss_i}^{\sss_{i+1}} \left( F^{\phat}_-(\sss)\right)^2d\sss\right).
\end{align*}
This implies a lower bound on $\sss_{i+1} - \sss_i$ which is uniform in $i$.  Therefore $a = \infty$.  We next show there are finitely many jump latitudes by estimating an upper bound for $\sss_k$. 
Suppose $\phat$ has a jump at $\sss_{j+1}$. 
On $\{\sss\in (\sss_j, \sss_{j+1})\}$, $ \phat$ coincides with
\begin{multline}
\label{EHj}
 \phat(\sss_j) H^+\left[F_{j+}; \sss_{j}\right] = \left( F^\phio_+(\sss_j) - F_{j+}\right) \phat(\sss_j)\phio(\sss_j) \phie(\sss) \\
  +\left( - F^\phie_+(\sss_j) + F_{j+}\right) \phat(\sss_j)\phie(\sss_j) \phio(\sss).
  \end{multline}
Since $\phat$ has a jump at $\sss_{j+1}$, it follows that
$F^\phio_+ (\sss_j) > F_{j+}$.  
Since
 \[ F^\phio_+ (\sss)= \frac{\sech^2\sss}{\tanh \sss} \searrow 0 \quad \text{as} \quad \sss \rightarrow \infty,\] 
 and $F_{j+} \sim_{\exp{(|\bsigma|_{\ell^1}+3|\xibold|_{\ell^\infty})}} F^\phie_-(\sss_1)$, this implies an upper bound for $\sss_{k-1}$ depending only on $\sss_1$ and $\bsigmaunder$.  
 This establishes the existence of the family of RLD solutions $\phat[\sss_1; \bsigmaunder]$. 
 
We next prove (i).  From Definition \ref{dphie} and Lemma \ref{LFmono}, it
follows that
 \begin{align}
  \label{EA1}
f_1(\sss):=  F^{\phio}_+(\sss) - \frac{1+\xi_1}{1-\xi_1} F^\phie_-(\sss)
 \end{align}
 is monotone on $(0, \sss_{root})$ and moreover satisfies
\begin{align}
\lim_{\sss_1\searrow 0} f_1(\sss) = \infty, \quad \text{and} \quad
\lim_{\sss_1 \nearrow a_{0, \bsigmaunder}} f_1(\sss_1) = - \infty.
\end{align}

We then define $a_{1, \bsigmaunder}$ to be the unique root of $f_1$ in $(0,
a_{0, \bsigmaunder})$.  By \eqref{EA1}, $a_{1, \bsigmaunder}$ depends only
on $\xi_1$.  There are three cases.

\textbf{Case 1:} $\sss_1> a_{1, \bsigmaunder}$.  It follows from Remark
\ref{rsigbij}, \eqref{EA1}, and \eqref{EHj} that $A_1(\sss_1)< 0$.  Since
$B_1>0$ by Lemma \ref{Lcoeffs}, the monotonicity of $\phie$ and $\phio$
imply that $\partial \phi>0$ for all $\sss >\sss_1$.  
Hence $k[\sss_1; \bsigmaunder] = 1$.

\textbf{Case 2:} $\sss_1 = a_{1, \bsigmaunder}$.  Then $A_1(\sss_1) = 0$.
By Lemma \ref{Lcoeffs}, $\phi$ coincides with a positive multiple of
$\phio$ on a maximal domain to the right of $\sss_1$, which must be $\{
\sss\in (\sss_1, \infty)\}$  since $F^\phio_+$ is strictly decreasing and
$\lim_{\sss\rightarrow \infty}F^\phio_+(\sss) = 0$.  Consequently
$\phat[a_{1, \bsigmaunder}; \bsigmaunder]$ is smooth at the poles and
$k[a_{1, \bsigmaunder}; \bsigmaunder] =1$.

\textbf{Case 3:} $\sss_1 < a_{1, \bsigmaunder}$.  Then $A_1 >0$.  Since $\phie$ is monotone and $\lim_{\sss\nearrow \infty}\partial \phie(\sss)= -\infty$, there exists $\sss_{1max}\in (\sss_1, \infty)$ with the property that $\partial \phat(\sss)>0$ when $\sss \in (\sss_1, \sss_{1max})$ and $\phat$ attains a local maximum at $\sss_{1max}$.  In particular, $F^{\phat}_-(\sss_{1max}) = 0$.  By the flux monotonicity, there exists $\sss_2 \in (\sss_{1max}, \infty)$ 
such that $\phat(\sss_1)H^+[F_{1+}; \sss_1](\sss_2) = F_{2-}$.  
Since $\phat$ remains positive, $\sss_2\in L_{par}$ and $\phat$ has its second jump latitude at $\sss_2$.  

We next prove $\sss_2(\sss_1), \sss_3(\sss_1),\dots$ are increasing functions of $\sss_1$.  We first show this for $\sss_2(\sss_1)$.  By Lemma \ref{LFmono}, $F^{\phat}_{1}$ is an increasing function of $\sss_1$.  By \ref{rsigbij}, $F^{\phat}_{2-} = \frac{1-\xi_2}{2} e^{\sigma_1} F^{\phat}_{1}$.  By combining this with both parts of Lemma \ref{LHmono}, it follows that $\sss_2(\sss_1)$ is strictly increasing. Using this fact and replacing $\sss_1$ by $\sss_2$ and $\sss_2$ by $\sss_3$ in the preceding argument shows that $\sss_3(\sss_1)$ is strictly increasing, and inductively in the same way $\sss_j(\sss_1)$ is strictly increasing for $2< j\leq k$.  This proves (ii).  

By the discussion above, $\phat[a_{1, \bsigmaunder}; \bsigmaunder]$ is smooth at the poles and $k[\sss_1; \bsigmaunder] = 1$ for $\sss_1\in I_1$.  
By a straightforward inductive argument, 
there are unique $a_{2, \bsigmaunder}>a_{3, \bsigmaunder}> \dots> 0$, 
where $a_{j, \bsigmaunder}$ depends only on the first $j-1$ entries of $\bsigma$ and the first $j$ entries of $\xibold$, and intervals 
$ [a_{i, \bsigmaunder}, a_{i-1, \bsigmaunder}), j \in \N$ such that for each $i\geq 1$, $\phat[a_{i, \bsigmaunder}; \bsigmaunder]$ is smooth at the poles,  
$k[a_{i, \bsigmaunder}; \bsigmaunder] = i$, and $k[\sss_1; \bsigmaunder] = i$ when $\sss_1\in [a_{i, \bsigmaunder}, a_{i-1, \bsigmaunder})$.  
This concludes the proof of (i) and (iii).  
(iv) follows from the continuous dependence in of the coefficients in the conclusion of Lemma \ref{Lcoeffs} on $\sss_1$ and the continuous dependence of the $\sss_2, \sss_3, \dots$ on $\sss_1$ from (ii).
\end{proof}
\begin{remark}
\label{Rxbalanced}
One construction in \cite{kap} produces surfaces whose configurations consist of points 
on two parallel circles with latitudes whose absolute values are close to (in the notation of \cite{kap}) $\xx_{balanced}\in (0, \xx_{root})$.  
Using \eqref{Exs} to relate $\sss$ and $\xx$ coordinates, we have that $\sech \sss_{root} = \cos \xx_{root}$ and $\sech(a_{1, \zerobold}) = \cos \xx_{balanced}$.  
\end{remark}

If $\phat = \phat[\sss_1; \bsigmaunder]$ is as in Proposition \ref{existence}.(iii), 
recall from  \ref{existence}.(ii) that $\sss^\phat_i$ is an increasing function of $\sss_1$.  
In Lemma \ref{Lsderiv} and Corollary \ref{Csderiv} we estimate the derivatives precisely. 

\begin{lemma}
\label{Lsderiv}
Let $\bsigmaunder = (\bsigma, \xibold)$, 
$\sss_1 \in (0, \sss_{root})$,  
$\phat = \phat[\sss_1; \bsigmaunder]$, 
and $\sbold = \sbold^{\phat[ \sss_1; \bsigmaunder]}$ be as in Proposition \ref{existence}.  
Let $k\in \N$. 
$\phat$ has at least $k$ jumps if and only if 
$\sss_1 \in (0, a_{k-1, \bsigmaunder})$. 
The $k$th jump latitude $\sss_k$ depends then only on $\sss_1$ and $\left. \bsigmaunder\right|_k$.  
$\sss_k$ can be considered as a smooth function defined on the domain $S_k\subset \R\times\R^{k-1}\times\R^k$ where 
$$
S_k:=
\left \{\, \left ( \, \sss_1 \, ,
\, (\sigma_i )_{i=1}^{k-1} \, , \, (\xi_j)_{j=1}^k \, \right) \, : \, 
\sss_1 \in (0, a_{k-1, \bsigmaunder}) \text{ and } |\xi_j|<1/10 \text{ for } j=1,\dots,k\,\right\}.
$$ 
Alternatively we can consider each $\sss_k$ as a smooth function of 
$F_1=F_1^\phat$ and $\left. \bsigmaunder\right|_k$ and then we have for $k=1$ (where $\sss_1$ is a function of $F_1$ and $\xi_1$ only) 
\begin{equation} 
\label{Msderiv0}
\begin{aligned} 
\left(2 \sech^2 \sss_1 + \left( F^\phat_{1-}\right)^2\right)\frac{\partial \sss_1 }{ \partial F_1} &= \textstyle{\frac{1}{2}}(1 - \xi_1), 
\qquad 
\left(2 \sech^2 \sss_1 + \left( F^\phat_{1-}\right)^2\right)\frac{\partial \sss_1 }{ \partial \xi_1} &= - \textstyle{\frac{{1}}{2}} F_1, 
\end{aligned} 
\end{equation} 
and for $k>1$ the recursive formulas (note $S_k\subset S_{k-1}$) 
\begin{multline}
\label{Msderiv}
\left( 2\sech^2 \sss_{k} + \left(F^\phat_{k-}\right)^2\right)\frac{ \partial \sss_{k}}{ \partial F_1}
= 
\left( 2\sech^2 \sss_{k-1} + \left(F^\phat_{k-1+}\right)^2\right)\frac{ \partial \sss_{k-1}}{ \partial F_1}\left( \frac{\phi(\sss_{k-1})}{\phi(\sss_{k})}\right)^2+\\
+ \left( \frac{\phi(\sss_{k-1})}{\phi(\sss_{k})}\right)^2\frac{1}{2} (1+ \xi_{k-1}) \left( e^{\sum_{l=1}^{k-2} \sigma_l} \right)  + 
\frac{1}{2}(1- \xi_{k}) \left( e^{\sum_{l=1}^{k-1} \sigma_l} \right) , 
\end{multline}
\begin{multline}
\label{Msderiv2}
\left( 2\sech^2 \sss_{k} + \left(F^\phat_{k-}\right)^2\right)\frac{ \partial \sss_{k}}{ \partial \xi_j}
= 
\left( 2\sech^2 \sss_{k-1} + \left(F^\phat_{k-1+}\right)^2\right)\frac{ \partial \sss_{k-1}}{ \partial \xi_j}\left( \frac{\phi(\sss_{k-1})}{\phi(\sss_{k})}\right)^2+\\
+\left( \frac{\phi(\sss_{k-1})}{\phi(\sss_{k})}\right)^2\frac{\delta_{j(k-1)}}{2} \left( e^{\sum_{l=1}^{k-2} \sigma_l} \right)  F_1- 
\frac{\delta_{jk}}{2} \left( e^{\sum_{l=1}^{k-1} \sigma_l} \right) F_1,
\end{multline}
\begin{multline}
\label{Msderiv3}
\left( 2\sech^2 \sss_{k} + \left(F^\phat_{k-}\right)^2\right)\frac{ \partial \sss_{k}}{ \partial \sigma_j}
= 
\left( 2\sech^2 \sss_{k-1} + \left(F^\phat_{k-1+}\right)^2\right)\frac{ \partial \sss_{k-1}}{ \partial \sigma_j}\left( \frac{\phi(\sss_{k-1})}{\phi(\sss_{k})}\right)^2+\\
+\left( \frac{\phi(\sss_{k-1})}{\phi(\sss_{k})}\right)^2\frac{F_1}{2}(1+\xi_{k-1})\frac{\partial}{\partial \sigma_j} \left( e^{\sum_{l=1}^{k-2} \sigma_l} \right) 
- \frac{F_1}{2}(1- \xi_{k}) \frac{\partial}{\partial \sigma_j}\left( e^{\sum_{l=1}^{k-1} \sigma_l} \right) .
\end{multline}
\end{lemma}
\begin{proof}
The first statement follows directly from Proposition \ref{existence}.  Below, we compute partial derivatives of $\sss_k$ with respect to $F_1, \sss_{k-1}$, and the entries of $\left.\bsigmaunder\right|_{k}$, from which the smoothness claimed follows immediately.  To this end, we recall from \eqref{EFxi} and \ref{dHflux}, that
\begin{align}
\label{EHbump}
\phi = \phi(\sss_{k-1}) H^+\left[ \frac{1}{2}(1+ \xi_{k-1})  \left( e^{\sum_{l=1}^{k-2}\sigma_l} \right)  F_1; \sss_{k-1} \right] \quad \text{on} \quad \{ \sss \in [\sss_{k-1}, \sss_{k}]\}.
\end{align} 
Denoting $H^+$ as in \eqref{EHbump} by $H$ for simplicity and using \eqref{EFxi}, we find
\begin{align}
\label{EFconstraint}
F^{H}_-(\sss_{k}) = \frac{1}{2}(1-\xi_{k})  \left( e^{\sum_{l=1}^{k-1}\sigma_l} \right)  F_1.
\end{align}
Items \eqref{Msderiv}-\eqref{Msderiv3} then follow by using the chain rule to differentiate \eqref{EFconstraint} and Lemma \ref{LFmono} and both parts of \ref{LHmono} to calculate the partial derivatives of $F^H$.
\end{proof}

\begin{lemma}
\label{Lrldest}
There is a constant $\epsilonunder_1>0$ such that for all $k\in \N, k>1$ and all $\bsigmaunder = (\bsigma, \xibold) \in \R^{k}\times \R^{k+1}$ satisfying $|\bsigma|_{\ell^1}+ |\xibold|_{\ell^\infty}< \epsilonunder_1$ and $\sss_1 \in (a_{k+1, \bsigmaunder}, a_{k, \bsigmaunder})$, $\phat = \phat[\sss_1; \bsigmaunder]$ satisfies the following.
\begin{enumerate}[label =(\roman*).]
\item $F^\phat_\ave < \frac{C}{k}$.
\item $\left\| 1- \phat(\sss) : C^0\left( \{ \sss \in [-\sss_k, \sss_k]\} \right) \right\| < \frac{C}{k}\log k$ and $\left| \log \frac{\phat(\sss_i)}{\phat(\sss_{i-1})}\right| < \frac{C}{k}$ for $2\leq i \leq k$. 
\end{enumerate}
\end{lemma}
\begin{proof}  

Using the monotonicity of the flux in lemma \ref{LFmono} we conclude that the maximum of $|F^{\phat}_-|$ is achieved at the jump latitudes. 
Using also \ref{rsigbij} we conclude 
\begin{align} 
\label{EFcomp} 
\max_{\sss\in[0,\infty)} |F^{\phat}_-(\sss)| = 
\max_{1\leq i\leq k+1} {F^\phat_{i\pm}}\sim_{ \exp{(|\bsigma|_{\ell^1}+3|\xibold|_{\ell^\infty})}} F^\phat_{avg} \sim_{ \exp{(|\bsigma|_{\ell^1}+3|\xibold|_{\ell^\infty})}} \min_{1\leq i\leq k+1}{ F^\phat_{i\pm}}.
\end{align}
By \ref{rH} and the fact that $\phat$ has $k+1$ jumps, it follows that $F^\phio_+(\sss_k) = \frac{\sech^2 \sss_k}{\tanh \sss_k} > F^\phat_{k+}$.  From this and \ref{EFcomp} we estimate
\begin{align}
\label{Efsechest}
 F^\phat_{avg} < C \sech^2 \sss_k. 
\end{align}
By using Lemma \ref{Lintid} on the intervals $\{\sss \in [ 0, \sss_1]\}, \dots \{ \sss\in [\sss_{k-1}, \sss_k]\}$ and summing, we find
\begin{align}
\label{Emostfluxes}
F^\phat_{1-}+ F^\phat_{1+}+ \cdots + F^\phat_{k-} = 2\tanh \sss_k + \int_{0}^{\sss_k} \left( F^\phat_-(\sss)\right)^2 \, d\sss.
\end{align}
Next using \ref{RLquant} and \eqref{EFcomp} and \eqref{Efsechest} to estimate \eqref{Emostfluxes}, we find
\begin{align}
\label{Efavg}
(2k-1)F^\phat_\ave < C\left(1+ \sss_k \sech^2 \sss_k\right)< C,
\end{align}
from which we conclude (i).  For (ii), we observe from \ref{dF} that on any domain on which $\phat$ is smooth, $ F^\phat_+(\sss) = \partial \left( \log \phat\right)$.  Integrating from $\sss_{i-1}$ to $\sss_i$, using the easy to see fact that for $2\leq i \leq k$, $0<\sss_{i}-\sss_{i-1} <C$, and estimating using part (i), 
\begin{align*}
\left| \log\left( \frac{\phat(\sss_i)}{\phat(\sss_{i-1})}\right)\right| \leq \int_{\sss_{i-1}}^{\sss_i} \left|F^\phat_+(\sss)\right| d\sss \leq \frac{C}{k}\int_{\sss_{i-1}}^{\sss_i} d\sss \leq \frac{C}{k}.
\end{align*}
\end{proof}

\begin{corollary}
\label{Csderiv}
Let $\bsigmaunder = (\bsigma, \xibold)$, $\sss_1 \in (0, \sss_{root})$, $\phat = \phat[\sss_1; \bsigmaunder]$, and $\sbold = \sbold^{\phat[\sss_1;\bsigmaunder]}$ be as in \ref{Lsderiv} and suppose moreover that $|\xibold|_{\ell^\infty}<\epsilonunder_1/k$ and that  $(\sss_1, \left. \bsigmaunder\right|_{k+1}) \in S_{k+1}$ (recall \ref{Lsderiv}).  The following estimates hold.
\begin{enumerate}[label = (\roman*).]
\item 
$\displaystyle{  \left( 2\sech^2 \sss_{k} + \left(F^\phat_{k-}\right)^2\right)\frac{ \partial \sss_{k}}{ \partial F_1}
\sim_C k}$.
\
\item $\displaystyle{
	\left| \left( 2\sech^2 \sss_{k} + \left(F^\phat_{k-}\right)^2\right)\frac{ \partial \sss_{k}}{ \partial \sigma_i }\right| < C
	}$, $i=1, \dots, k-1$.
\item
	$\displaystyle{
	\left|\left( 2\sech^2 \sss_{k} + \left(F^\phat_{k-}\right)^2\right)\frac{ \partial \sss_{k}}{ \partial 
	\xi_j}	\right| < \frac{C}{k}
	}$, $j=1, \dots, k$.
\end{enumerate}
\end{corollary}
\begin{proof}
We first prove (i).  To simplify notation in this proof, we denote
\begin{equation*}
\begin{aligned}
P_k &:= \left( 2 \sech^2 \sss_k + \left( F^\phat_{k-}\right)^2\right) \frac{ \partial \sss_k}{\partial F_1}, 
\quad
Q_{k-1} := \left( \frac{\phat(\sss_{k-1})}{\phat(\sss_k)}\right)^2, 
\quad
R_k := \frac{ 2 \sech^2 \sss_k + \left( F^\phat_{k+}\right)^2}{ 2 \sech^2 \sss_k + \left( F^\phat_{k-}\right)^2},\\
T_{k-1} &:=  \frac{1}{2} Q_{k-1} (1+ \xi_{k-1}) \left( e^{\sum_{l=1}^{k-2} \sigma_l} \right)  + 
\frac{1}{2}(1- \xi_{k}) \left( e^{\sum_{l=1}^{k-1} \sigma_l} \right). 
\end{aligned}
\end{equation*}
In this notation, \eqref{Msderiv} from Lemma \ref{Lsderiv} is equivalent to the equation
\begin{align}
\label{Esderivrec}
P_k =  R_{k-1} Q_{k-1}P_{k-1} + T_{k-1}
\end{align}
from which we conclude by applying \eqref{Msderiv} recursively
\begin{align}
\label{Epk}
P_k = P_1 \prod_{i=1}^{k-1} Q_i R_i + \sum_{i=1}^{k-1} \left( T_i \prod_{j=i+1}^{k-1} Q_j R_j\right).
\end{align}
From \eqref{Msderiv0} it follows that $P_1 = \frac{1}{2}(1 - \xi_1)$.  By \ref{Lrldest}, the assumption $|\xibold|_{\ell^\infty}<\epsilonunder_1/k$, and \ref{rsigbij}, the following estimates hold:
\begin{align}
\label{Eqrest}
Q_i \sim_{1+ C/k} 1, \quad R_i \sim_{1+C/k}, \quad T_i \sim_C 1, \qquad i=1, \dots, k-1.
\end{align}
Combining \eqref{Eqrest} with \eqref{Epk} completes the proof of (i).  Proofs of (ii) and (iii) are similar and use respectively \eqref{Msderiv2} and \eqref{Msderiv3} in place of \eqref{Msderiv}, so we omit the details. 
\end{proof}

We are particularly interested in RLD solutions which are smooth at the poles. 
Recall from \ref{existence} that 
$\phat[a_{k,\bsigmaunder};  \bsigmaunder]$ 
is the unique smooth at the poles unit RLD solution  
with $k$ jump latitudes and flux ratios $\bsigmaunder$.  
Moreover by \ref{existence}.(i), 
$\phat[a_{k,\bsigmaunder};  \bsigmaunder]$ 
depends only on 
$\left. \bsigmaunder \right|_k$. 
This motivates the following. 

\begin{notation}
\label{Nphik}
Given $k$, $\bsigmaunder$, and 
$\left. \bsigmaunder \right|_k$ as in \ref{existence}, 
we define 
$$ 
\phat[\bsigmaunder: k]:= 
\phat[\left. \bsigmaunder  \right|_k : k]:= 
\phat[a_{k,\bsigmaunder};  \bsigmaunder],  
\qquad\quad
\sbold = \sbold[\bsigmaunder: k ] := 
\sbold[\left. \bsigmaunder  \right|_k : k ] := 
\sbold^{ \phat[\bsigmaunder: k]   }. 
$$
\end{notation}

\begin{corollary}
\label{Csi}
Fix $\bsigmaunder$ as in Proposition \ref{existence} and let 
$\sbold = \sbold[ \bsigmaunder : k]$ be as in \ref{Nphik}. 
Then 
\begin{enumerate}[label=(\roman*).]
\item For fixed $i$, $ \sss_i = \sss_i[\bsigmaunder :k ]$ is a decreasing function of $k$.
\item If $\bsigmaunder = \zerobold$, and $i$ is fixed, $\sss_{k-i} = \sss_{k-i}[\zerobold : k]$ is an increasing function of $k$. 
\end{enumerate}
\end{corollary}
\begin{proof}
(i) follows from parts (i) and (ii) of Proposition \ref{existence}.  (ii) follows from a dual version of Proposition \ref{existence} where one builds $\phat[ \zerobold: k]$ ``from infinity" by taking the final flux $F^\phat_+(\sss_k)$ as initial data rather than the first flux $F^\phat_-(\sss_1)$ as in the proof of \ref{existence}.  We omit the details of the argument because we do not use part (ii) in the remainder of the paper. 
\end{proof}


 \begin{figure}[h]
\centering
\includegraphics[scale=1]{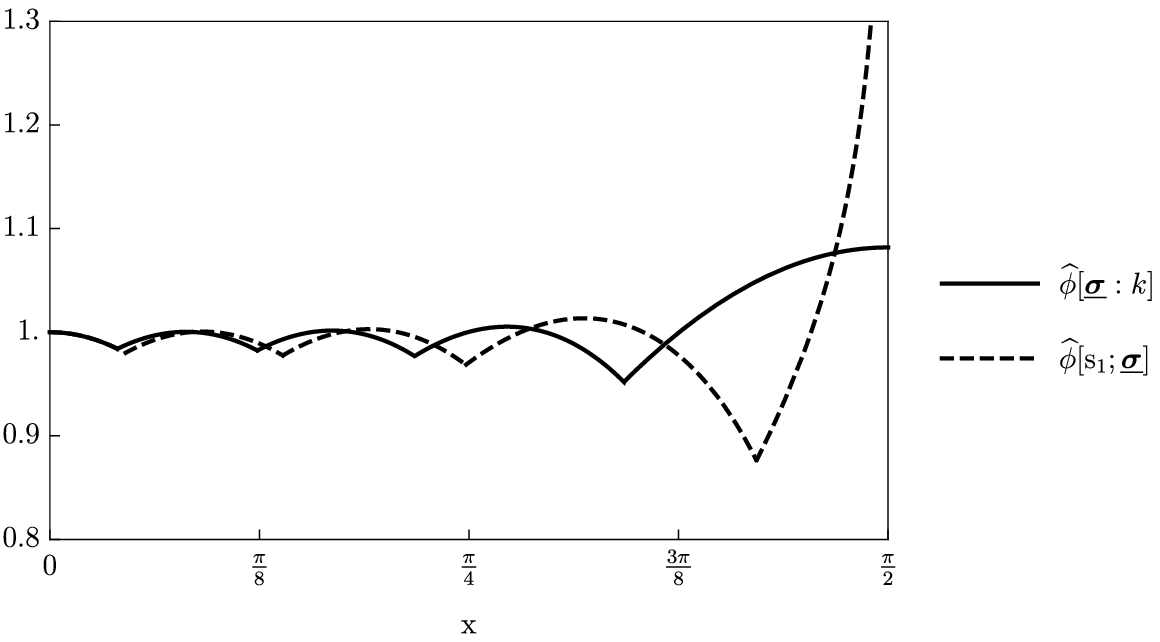}
\caption{Profiles of an RLD solution $\phat[\sss_1; \bsigmaunder]$ for $\sss_1\in (a_{4, \bsigmaunder}, a_{3, \bsigmaunder})$ 
and the smooth at the poles RLD solution $\phat[ \bsigmaunder :k]=\phat[a_{k,\bsigmaunder};  \bsigmaunder]$ (when $k=4$).  
In both cases $\bsigmaunder = \zerobold$.}
\end{figure}

We next establish estimates for smooth at the poles RLD solutions with small flux ratios.  Below, if we write $f = O(w)$ for a number $f$ and some $w>0$, we mean $|f| \le C w$, for some constant $C$ independent of $f$ and $w$. 

\begin{prop}
\label{Prl}
There is a constant $\epsilonunder_1>0$ such that for all $k\in \N$ and all $\bsigmaunder= (\bsigma, \xibold)\in \R^{k-1}\times\R^k$  with $|\bsigma|_{\ell^1}+ |\xibold|_{\ell^\infty} < \epsilonunder_1$,
$\phat =\phat[\bsigmaunder: k]$ satisfies the following.
\begin{enumerate}[label=(\roman*).]
\item  $F^{\phat}_{avg} = \frac{1}{k} +O\left(\frac{1}{k^3}\right)$ and $\sech^2 \sss_k \sim_C \frac{1}{k}$.  
\item \begin{enumerate}[label=(\alph*)]
	\item $ F^{\phat}_{1-}=2\tanh \sss_1 +O\left(\frac{1}{k^3}\right).$
	\item  $F^{\phat}_{i-}+F^\phat_{i-1+} = 2(\tanh \sss_i - \tanh \sss_{i-1}) +O\left(\frac{1}{k^2(k-i+1)}\right)$ for $i=2, \dots, k$.
	\item $F^{\phat}_{k+}=2(1- \tanh \sss_k) + O\left(\frac{1}{k^2}\right)$.
	\end{enumerate}
\item  $\left\|1-\phat(\sss): C^0\left(\cyl\right) \right\| \le \frac{C}{k}\log k$ and  $\left|\log\frac{\phat(\sss_i)}{\phat(\sss_{i-1})}\right| \le \frac{C}{k}$ for $2\leq i \leq k$.    If $\bsigmaunder=\zerobold$, then also $1>\phat(\sss_1)> \cdots >\phat(\sss_k)$.  
\item  There is a constant $C>0$ depending only on $\epsilonunder_1$ such that for any $1\leq i<j\leq k$, 
\begin{align*}
\sss_j - \sss_i &>  \left(1+\frac{C}{k}\right) \frac{j-i}{2k}
\quad \text{and} \quad 
\sech^2 \sss_i - \sech^2 \sss_j \le \frac{C(j-i)}{k}.
\end{align*}
\end{enumerate}
\end{prop}

\begin{remark}
\label{Rsmallk}
The estimates in Proposition \ref{Prl} are substantive only when $k$ is large.  When $k$ is small, by taking $\epsilonunder_1$ small enough, using Proposition \ref{existence}, and smooth dependence of ODE solutions on initial conditions, we maintain coarse control over the family $\phat[\bsigmaunder: k]$.   
More precisely we have the following alternatives for small $k$:  given $\kunder > 0$ and taking $\epsilonunder_1$ small enough in terms of $\kunder$, there exists a constant $c = c(\kunder)> 0$ such that for all $k\in \N, k < \kunder$, \newline
(i). $0< \frac{1}{c}< \phat < c$.\newline 
(ii). $\frac{1}{c}< \sss_1$ and for $i\in \{2, \dots, k\}$, $\sss_i - \sss_{i-1}\sim_c 1$. \newline
(iii). For $i\in \{1, \dots, k\}$,  $F^\phat_{i\pm} \sim_c 1$.
\end{remark}

\begin{proof}[Proof of Proposition \ref{Prl}]

Since $\phat$ is smooth at the poles,
\begin{align*}
F^{\phat}_{k+}= F^{\phio}_+(\sss_k) = \frac{\sech^2 \sss_k}{\tanh \sss_k}.
\end{align*}
Using this with \eqref{EFcomp} and estimating a uniform lower bound for $\tanh \sss_k$, we conclude 
\begin{align}
\label{EupFavg}
\sech^2 \sss_k \sim_{C\exp{ (|\bsigma|_{\ell^1}+|\xibold|_{\ell^\infty})}} F^\phat_{avg}.
\end{align}
By fixing $\epsilonunder_1$ sufficiently small, we can ensure the constant $C$ in \eqref{EupFavg} is less than $10$ and also the constants in the estimates in the statements (i)-(iii) are independent of $k$.  

Using Lemma \ref{Lintid}, we have (where $i=2, \dots, k$)
\begin{equation}
\label{Annuli}
\begin{aligned}
 F^\phat_{1-}&=2\tanh \sss_1 +\int_{0}^{\sss_1} \!  \left(F^\phat_-(\sss)\right)^2d\sss\\
 F^\phat_{i-}+F^\phat_{i-1+} &= 2(\tanh \sss_{i} - \tanh \sss_{i-1}) + \int_{\sss_{i-1}}^{\sss_i} \!  \left(F^\phat_-(\sss)\right)^2d\sss, \\ 
 F^\phat_{k+}&=2(1-\tanh \sss_k) + \int_{\sss_{k}}^{\infty} \!  \left(F^\phat_-(\sss)\right)^2d\sss.
\end{aligned}
\end{equation}
Note $\int_{\sss_{i-1}}^{\sss_i} \frac{\sech^2 \sss}{\sech^2 \sss}d\sss = \int_{\sss_{i-1}}^{\sss_i} \frac{1}{\sech^2 \sss} d\left( \tanh \sss\right)< C(\tanh \sss_i - \tanh \sss_{i-1})/F^\phat_\ave$ and also that
$\int_{\sss_k}^\infty \left( F^\phat_-(\sss)\right)^2 d\sss = \int_{\sss_k}^\infty \frac{\sech^4 \sss}{\tanh^2 \sss} d\sss$.
Using \eqref{EFcomp} and \eqref{EupFavg} we estimate the integrals to obtain 
(where $i=2, \dots, k$)
\begin{equation}
\label{Annuli2}
\begin{aligned}
 F^\phat_{1-}&=2\tanh \sss_1 +\sss_1 O\left(\left(F^\phat_{avg}\right)^2\right)\\
 F^\phat_{i-}+F^\phat_{i-1+} 
  &=2(\tanh \sss_{i} - \tanh \sss_{i-1}) +(\tanh \sss_{i} - \tanh \sss_{i-1})O(F^\phat_{avg}), \\ 
 F^\phat_{k+}&=2(1-\tanh \sss_k) + O\left(\left(F^\phat_{avg}\right)^2\right).
\end{aligned}
\end{equation}
By summing the estimates in \eqref{Annuli2} and dividing through by $k$,  we find that $
F^\phat_{avg} = \frac{1}{k} - \frac{1}{k}O\left(F^\phat_{avg}\right)$,
which implies that $F^\phat_{avg} = \frac{1}{k}+O\left(\frac{1}{k^2}\right)$.  Substituting this estimate for $F^\phat_{avg}$ into \eqref{Annuli2}, we find 
\begin{align}
\label{Etanhest}
2\tanh \sss_1 \sim_C \frac{1}{k}, \quad \tanh \sss_i - \tanh \sss_{i-1} \sim_{C} \frac{1}{k}, \quad 
i=2, \dots, k,
\end{align}
hence
\begin{align}
\label{Etanhest2}
\tanh \sss_i \sim_C \frac{i}{k}, \quad i=1, \dots, k.
\end{align}
Summing the estimates in \eqref{Annuli2} with these improved bounds yields
\begin{align*}
kF^\phat_{avg} = 1+O\left(\frac{1}{k^2}\right),
\end{align*} 
and the first part of (i) follows.  The second statement in (i) follows from the first and \eqref{EupFavg}.
Substituting (i) into the first and last parts of \eqref{Annuli2} yields (ii).(a) and (ii).(c).  For (ii).(b), we substitute (i) into the error term of the second part of \eqref{Annuli2} to get
\begin{align}
\label{E2b}
F^\phat_{i-}+F^\phat_{i-1+} &= 2(\tanh \sss_{i} - \tanh \sss_{i-1}) + \frac{1}{\sech^2 \sss_i}O\left( \frac{1}{k^3}\right).
\end{align}
When $i \leq k/2$, (ii).(b) follows immediately from \eqref{E2b} because \eqref{Etanhest2} implies $\sech^2 \sss_i$ is bounded from below by a constant independent of $k$.  On the other hand, when $j\geq k/2$, \eqref{Etanhest2} implies $\tanh \sss_j \geq C$, so from this and \eqref{Etanhest},
\begin{align}
\label{Esechdiff}
\sech^2 \sss_{j-1} - \sech^2 \sss_j & = (\tanh \sss_j - \tanh \sss_{j-1})(\tanh \sss_j +\tanh \sss_{j-1})
\ge \frac{C}{k}.
\end{align}
Using  \eqref{EupFavg} and summing \eqref{Esechdiff} and, we find for $i\geq k/2$
\begin{align}
\label{Esechbound}
\sech^2 \sss_i \geq \frac{C(k-i+1)}{k}.
\end{align}
Substituting this bound in \eqref{E2b} completes the proof of (ii).(b).  

We now prove (iii).  From Definition \ref{dF}, on any domain on which $\phat$ is smooth,
\begin{align}
\label{EFlog}
 F^{\phat}_+(\sss) = \partial\left( \log \phat \right).
 \end{align}
Integrating \eqref{EFlog} on intervals where $\phat$ is smooth and adding, we find that for any $\sbar \in [0, \sss_k]$,
\begin{align*}
\left|\log\phat(\sbar) \right| &\leq  \int_{0}^{\sbar}\left| F^{\phat}_+(\sss)\right| \,d\sss \leq \frac{C}{k} \int_{0}^{\sss_k} \, d\sss 
\leq \frac{C}{k} \log k,
\end{align*}
where we have used \eqref{EupFavg} in combination with part (i) above to estimate $\sss_k = O\left(\log k\right)$.  Since $\phat$ is smooth at the poles, it coincides with a multiple of $\phio = \tanh \sss$ on $\{ \sss\in (\sss_k, \infty)\}$ and so in particular the estimate in the first part of (iii) holds on $\{ \sss\in (\sss_k, \infty)\}$. 

 For the second statement, it is easy to see that for $2\leq i \leq k$, $0<\sss_{i}-\sss_{i-1} <C$.  Then estimating in the same way as above, we get
\begin{align*}
\left| \log\left( \frac{\phat(\sss_i)}{\phat(\sss_{i-1})}\right)\right| \leq \int_{\sss_{i-1}}^{\sss_i} \left|F^\phat_+(\sss)\right| d\sss \leq \frac{C}{k}\int_{\sss_{i-1}}^{\sss_i} d\sss \leq \frac{C}{k}.
\end{align*}
Next, suppose that $\bsigmaunder = \zerobold$.  Fix $i\in \{ 1, \dots, k-1\}$. Integrating \eqref{EFlog} from $\sss_i$ to $\sss_{i+1}$, we find that 
\begin{align}
\label{Elogphi}
\log\left(\frac{\phat(\sss_{i+1})}{\phat(\sss_{i})}\right) &= - \int_{\sss_i}^{\sss_{i+1}} F^\phat_-(\sss)d\sss. 
\end{align}
Since $\phat$ is balanced, $F^\phat_{i+1-} = F^\phat_{i+}:= F_0$ and the restriction of $F^\phat_-$ to $[\sss_i, \sss_{i+1}]$ satisfies
\begin{align*}
\left.F^\phat_-\right|_{[\sss_i, \sss_{i+1}]}: [\sss_i, \sss_{i+1}]\rightarrow [-F_0, F_0]\quad \text{and} \quad \frac{dF^\phat_-}{d\sss} = 2 \sech^2 \sss+ F^\phat_-(\sss)^2>0.  
\end{align*}
In particular, $F^\phat_-|_{[\sss_i, \sss_{i+1}]}$ is invertible.  Reparametrizing the integral in \eqref{Elogphi} by $\left(F^\phat_-|_{[\sss_i, \sss_{i+1}]}\right)^{-1}$, we have
\begin{align*}
\log\left(\frac{\phat(\sss_{i+1})}{\phat(\sss_{i})}\right) &= - \int_{-F_0}^{F_0}  F\frac{d\sss}{dF} dF\\
&= -\int_{-F_0}^{0} \frac{ F}{2\sech^2 \sss+F^2}dF-\int_{0}^{F_0} \frac{ F}{2\sech^2 \sss+F^2}dF\\
&<0,
\end{align*}
since the first integral term is positive, the second is negative, and $\sech^2 \sss$ is monotonic in $F$ by Lemma \ref{LFmono}.  This completes the proof of (iii).
For (iv), assume $1\leq i<j\leq k$.  By summing instances of (ii).(b) above, we find
\begin{align*}
F^\phat_{i+}+\cdots +F^\phat_{j-} = 2(\tanh \sss_j - \tanh \sss_i) + O\left( \frac{j-i}{k^2(k-i+2)}\right).
\end{align*}
Using \eqref{EFcomp}, we estimate
\begin{align*}
(j-i) \exp{(-|\bsigma|_{\ell^1}-3|\xibold|_{\ell^\infty})}F^\phat_{\ave} \le  (\tanh \sss_j - \tanh \sss_i) + O\left( \frac{j-i}{k^2(k-i+2)}\right).
\end{align*}
Applying the mean value theorem to the right hand side of the preceding gives
\begin{align} 
\label{Eji}
(j-i) \exp{(-|\bsigma|_{\ell^1}-3|\xibold|_{\ell^\infty})}F^\phat_{\ave} < \sss_j - \sss_i + O\left( \frac{j-i}{k^2(k-i+2)}\right)
\end{align}
and the first statement in (iv) follows from combining \eqref{Eji} with (i) above and taking $|\bsigma|_{\ell^1}$ and $|\xibold|_{\ell^\infty}$. small enough.  The second follows from combining \eqref{Esechdiff} with parts (i) and (ii) above.
\end{proof}

\begin{remark}
\label{Rspacing}
Proposition \ref{Prl} implies that the latitudes of $L_{par}$ of RLD solutions which are close to being balanced arrange themselves in a regular way.  Indeed, recalling from \eqref{Exs} that $\tanh \sss = \sin \xx$, Proposition \ref{Prl}.(i) and (ii) together imply that when $\bsigmaunder =\zerobold$,
\begin{equation}
\begin{aligned}
\sin \xx_1 &= \frac{1}{2k} + O\left( \frac{1}{k^3}\right),\\
\sin \xx_i - \sin \xx_{i-1} &= \frac{1}{k} + O\left( \frac{1}{k^2(k-i+1)}\right), \quad i=2, \dots, k-1, \\
1 - \sin \xx_k &= \frac{1}{2k} + O\left(\frac{1}{k^2}\right).
\end{aligned}
\end{equation}
Together with elementary geometric facts about spheres, this means that for large $k$
\begin{enumerate}
\item The (extrinsic $\R^3$) distance between planes corresponding to adjacent circles in $L_{par}$ is approximately $1/k$. 
\item For $i=2, \dots, k$, the area of the annulus $\{\xx\in (\xx_{i-1}, \xx_i) \} \subset \Spheq$ is approximately $2\pi/k$. 
\end{enumerate}
Since the LD solutions we use to construct initial surfaces have configurations arising 
(in a way made precise later in Lemma \ref{Lphiavg}) 
from RLD solutions which are approximately balanced, 
the circles of $L_{par}$ in the minimal surfaces we ultimately construct are also approximately equally spaced. 
\end{remark}

\begin{prop}
\label{PODEest}
 Let $k\in \N$ and suppose that $\bsigmaunder = (\bsigma, \xibold)\in \R^{k-1}\times \R^k , \bsigmaunder'= (\bsigma', \xibold') \in \R^{k-1}\times \R^k$ satisfy $|\bsigma|_{\ell^1}+ |\xibold|_{\ell^\infty}< \epsilonunder_1/k$ and $|\bsigma'|_{\ell^1}+ |\xibold'|_{\ell^\infty}< \epsilonunder_1/k$.  Let $\phat = \phat[ \bsigmaunder: k]$ and $\phat' = \phat[ \bsigmaunder': k]$.  There is a constant $C>0$ independent of $k$ such that: 
 \begin{enumerate}[label =(\roman*).]
\item 
$\displaystyle{ \left| \Fboldunder^{\phat'} - \Fboldunder^\phat \right|_{\ell^\infty} \le 
\frac{C}{k}\left(  | \bsigma' - \bsigma|_{\ell^1}+ |\xibold' - \xibold|_{\ell^\infty}\right) }$. 

\item  
$\displaystyle{ \max_{1\leq i \leq k} \left| \tanh \sss'_i - \tanh \sss_i \right| \le 
\frac{C}{k}\left(  | \bsigma' - \bsigma|_{\ell^1}+ |\xibold' - \xibold|_{\ell^\infty}\right)  }$. 
\end{enumerate}
\end{prop}

\begin{proof}
Observe that (i) follows from the estimate
\begin{align}
\label{Efluxdiff1}
\left| F^{\phat'}_1 - F^{\phat}_1 \right| \le \frac{C}{k}\left(  |\bsigma' - \bsigma |_{\ell^1} + |\xibold' - \xibold|_{\ell^\infty}\right), 
\end{align}
because for any $i\in \{1, \dots, k\}$ we may estimate (assuming \eqref{Efluxdiff1} for a moment)
\begin{align*}
2\left| F^{\phat'}_{i+} - F^\phat_{i+}\right| &= 
\left| (1+ \xi'_i ) \left( e^{\sum_{l=1}^{i-1}\sigma'_l} \right)  F^{\phat'}_1 - (1+ \xi_i)  \left( e^{\sum_{l=1}^{i-1}\sigma_l} \right)  F^\phat_{1}\right| \\
&\le \left|  (1 + \xi'_i)  \left( e^{\sum_{l=1}^{i-1}\sigma'_l} \right)  -  (1 +\xi_i)  \left( e^{\sum_{l=1}^{i-1}\sigma_l} \right) \right| F^{\phat'}_1
+   (1 + \xi_i)  \left( e^{\sum_{l=1}^{i-1}\sigma_l} \right) \left| F^{\phat'}_1 - F^{\phat}_1\right| \\
&\le \frac{C}{k}\left(  |\bsigma' - \bsigma |_{\ell^1} + |\xibold' - \xibold|_{\ell^\infty}\right)
\end{align*}
and a corresponding bound for $\left| F^{\phat'}_{i-} - F^{\phat}_{i-}\right|$ follows in an analogous way.  We now prove \eqref{Efluxdiff1}.

Fix $k\in \N$, let $\epsilonunder_1$ be as in \ref{Prl}, and 
consider the map 
defined by
\begin{align}
\label{EFcal}
\Fcal ( F_1, \bsigmaunder ) = F^{\phat[ \sss_1; \bsigmaunder]}_+(\sss_k) - F^{\phio}_+(\sss_k),
\end{align}
where $\phat[ \sss_1; \bsigmaunder]$ is as in \ref{existence} and $\sss_1$ is chosen so $F^\phat_1 = F_1$.  Clearly, $\Fcal ( F_1, \bsigmaunder) = 0$ if and only if $\phat$ is smooth at the poles.  Now let $\pointp \in \Fcal^{-1}\left(\{0\}\right)$ be arbitrary.

It follows from Lemma \ref{LFmono} and \ref{LHmono} that $\Fcal$ is $C^1$; below we estimate the partial derivatives of $\Fcal$ at $\pointp$.

Differentiating \eqref{EFcal} with respect to $F_1$ and using \eqref{EFxi} and \ref{LFmono}, we compute
\begin{align}
\label{EFF1}
\left. \frac{ \partial \Fcal }{ \partial F_1}\right |_{\pointp} &= \frac{1}{2}(1+ \xi_k) \left( e^{\sum_{l=1}^{k-1} \sigma_l} \right)  + \left(2 \sech^2 \sss_k + \left(F^{\phat}_{k+}\right)^2\right) \left.\frac{ \partial \sss_k}{\partial F_1}\right|_{\pointp}.
\end{align}
By combining \eqref{EFF1} with Corollary \ref{Csderiv}, we estimate that for $j\in \{1, \dots, k\}$ and $i\in \{1, \dots, k-1\}$,

\begin{align}
\label{EF1deriv}
\left.\frac{\partial \Fcal}{\partial F_1}\right|_{\pointp} \sim_C k, \qquad 
\left| \left.\frac{\partial \Fcal}{\partial \sigma_i}\right|_{\pointp} \right| \le C, \qquad
\left|\left. \frac{\partial \Fcal}{\partial \xi_j}\right|_{\pointp} \right| \le \frac{C}{k},
\end{align}
where $C>0$ is a constant independent of $k$.

By the implicit function theorem, locally around $(F_1, \bsigmaunder) \in \Fcal^{-1}\left( \{0\}\right)$, $\Fcal^{-1}\left(\{ 0 \} \right)$ is a graph over $\bsigmaunder$ and moreover (abusing notation slightly), for $i\in \{1, \dots, k-1\}$ and $j\in \{1, \dots, k\}$
\begin{align}
\label{Eift}
\left.\frac{\partial F_1}{\partial \sigma_i}\right|_{\bsigmaunder} = - \left(\left. \frac{ \partial \Fcal}{ \partial F_1}\right|_{\pointp}\right)^{-1} \left.\frac{ \partial \Fcal }{ \partial \sigma_i}\right|_{\pointp}, \qquad
\left.\frac{\partial F_1}{\partial \xi_j} \right|_{\bsigmaunder}= - \left( \left.\frac{ \partial \Fcal}{ \partial F_1}\right|_{\pointp}\right)^{-1}\left. \frac{ \partial \Fcal }{ \partial \xi_j}\right|_{\pointp}.
\end{align}
(i) follows by combining this with the estimates \eqref{EF1deriv}.  The proof of (ii) is similar to the proof of (i), but we omit it since we will not use (ii) in the remainder of the paper. 
\end{proof}

\section{Linearized Doubling (LD) solutions}
\label{S:LD}


\subsection*{LD solutions}
$\phantom{ab}$
\nopagebreak 
\begin{definition}[LD solutions, {\cite[Definition 3.1]{kap}}] 
\label{dLD0}
Given a finite set $L\subset \Spheq$ and a function $\tau:L\to\R$,
we define a \emph{linearized doubling (LD) solution of configuration} $(L,\tau)$ to be a function 
$\varphi\in C^\infty(\Spheq\setminus L)$ which satisfies the following conditions, 
where $\tau_p$ denotes the value of $\tau$ at $p$:
\begin{enumerate}[label=(\roman*).]
\item $\Lcalp \varphi=0$ on $\Spheq\setminus L$. 
\item  $\forall p\in L$ there is a smooth extension across $\{p\}$, 
$\varphihat_p\in C^\infty \left(\, \{p\}\cup(\Spheq\setminus (L\cup\{-p\})\,)\,\right)$,   
such that 
$$\varphihat_p=\varphi-\tau_p \, G^{\Sph^2} \circ \dbold^g_p  \quad\text{on}\quad \Spheq\setminus (L\cup\{-p\}),$$ 
where $G^{\Sph^2} \circ \dbold^g_p$   
is as in \ref{LGp}. 
\end{enumerate}
\end{definition}

\begin{remark}
\label{rLDgreen}
By Lemma \ref{Lremovablesing}, 
the class of LD solutions  
and each corresponding $\tau_p$, 
do not depend on which Green's function we use in Definition \ref{dLD0}.(ii) above. 
Note however that $\varphihat$ and its domain do. 
\end{remark}

\begin{definition}[LD solutions modulo {$\skernel [ L]$}, {\cite[Definition 3.2]{kap}}]
\label{dLD1}
Given $L$ and $\tau$ as in Definition \ref{dLD0}, 
and also a finite dimensional space $\skernel[L] \subset C^{\infty}(\Spheq)$, 
we define a \emph{linearized doubling (LD) solution of configuration} $(L,\tau,w)$ to be a function 
$\varphi\in C^\infty(\Spheq\setminus L)$ which satisfies the same conditions as in \ref{dLD0}, 
except that condition (i) is replaced by the following:
\\ 
$\phantom{abc}$ 
(i$\,{}'$). $\Lcalp \varphi=w\in \skernel[L] \subset C^\infty(\Spheq)$ on $\Spheq\setminus L$.

\end{definition} 

Note that LD solutions in the sense of Definition \ref{dLD0} are LD solutions with $w=0$ in the sense of \ref{dLD1}.  
Existence and uniqueness of $\grouptwo$-symmetric LD solutions modulo $\skernel[L]$ is a consequence of the symmetries of the construction: 

\begin{lemma}[$\grouptwo$-Symmetric LD solutions, {\cite[Lemma 3.10]{kap}}]  
\label{LLDexistence}
We assume given $\grouptwo$-invariant $L,\tau,w$, 
where $L= L[\sbold;m]$ is as in \ref{dL},  
$\tau:L\to\R$, and $w\in \skernel[L]$.  
There is then a unique $\grouptwo$-invariant LD solution modulo $\skernel[L]$ 
of configuration $(L, \tau, w)$,  
$\varphi = \varphi[L, \tau, w]$. 
Moreover, the following hold.
\begin{enumerate}[label=(\roman*).]
\item $\varphi$ and $\varphihat_p$ each depend linearly on $(\tau, w)$.

\item  $\varphi_\ave\in C^0(\Spheq\setminus(L\cap\{p_N,p_S\})\,)$
and $\varphi_\ave$ is smooth on 
$\left(\Spheq\right)^\sbold$ 
where it satisfies the ODE $\Lcalp \varphi_\ave=0$.
\end{enumerate}

If $w=0$, then we also write $\varphi = \varphi[L; \tau]$ and 
$\varphi$ is the unique $\grouptwo$-invariant LD solution of configuration $(L, \tau)$ as in \ref{dLD0}.
\end{lemma}
\begin{proof}
This is Lemma 3.10 of \cite{kap}.
\end{proof}

%
%
\subsection*{Estimates for LD solutions on the cylinder}
$\phantom{ab}$
\nopagebreak 

The next lemma asserts a balancing law for LD solutions $\varphi$ which restricts their averages $\varphi_{avg}$. 

\begin{lemma}[Vertical balancing]
\label{Lvbal}
Suppose $\varphi = \varphi[L[\sbold; m]; \tau]$ is as in 
\ref{LLDexistence}, where $L[\sbold; m]$ and $\tau$ are as in \ref{dL}.  
Then 
\begin{align}
\label{Evbal} 
m \tau_i = \varphi_{avg}(\sss_i) F^{\varphi_{\ave}}_i, \quad i=1, \dots, k. 
\end{align}
\end{lemma}

\begin{proof}
The proof is similar to the proof of \cite[Lemma 3.10]{kap}.  
Fix $i \in \{1, \dots, k\}$.  
For $0<\epsilon_1<<\epsilon_2$ we
consider the domain 
$\Omega_{\epsilon_1,\epsilon_2}:=
D^\chi_{L_{par}[\sss_i]}(\epsilon_2) 
\setminus 
D^\chi_{L_i}(\epsilon_1) 
$.
By integrating $\Lcal_\chi \varphi=0$ on $\Omega_{\epsilon_1,\epsilon_2}$
and integrating by parts we obtain
\begin{align*} \int_{\partial\Omega_{\epsilon_1,\epsilon_2} }
\frac\partial{\partial\eta}\varphi \, + \, 2 \int_{\Omega_{\epsilon_1,\epsilon_2} }\!\left( \sech^2 \sss\right)\,  \varphi = 0,
\end{align*}
where $\eta$ is the unit outward conormal vector field along $\partial \Omega_{\epsilon_1,\epsilon_2}$.  By taking the limit as $\epsilon_1\to 0$ first
and then as $\epsilon_2\to 0$ we obtain using the logarithmic
behavior near $L$ that 
\begin{align}
\label{ELDbal}
m \tau_i = \partial_+ \varphi_{avg}(\sss_i) + \partial_-\varphi_{avg}(\sss_i).
\end{align}
 \eqref{Evbal} follows from  \eqref{ELDbal} after appealing to the definition of 
$F^{\varphi_{avg}}_\pm$ (recall \ref{dF} and \ref{RL}.(iii)).
\end{proof}

\begin{assumption}
\label{Aratio}
We assume from now on $k, m\in \N$ are fixed and $m$ satisfies $m > \Cmk$
where $\Cmk= \Cmk(k)>0$ is a constant which can be taken as large as necessary in terms of $k$.
\end{assumption}  

In the next lemma we convert RLD solutions $\phat$ to LD solutions $\Phi$ whose non-oscillatory part 
in the sense of \ref{davg} is a multiple of $\phat$. 
The overall scale of the LD solutions used is determined later in \ref{dtau1},  
while at the moment we only fix their scale by a suitable normalizing condition. 
The purpose of the rest of this section is to understand and estimate carefully 
$\Phi$ by decomposing it to well described parts $\Pp$ and $\Phat$, 
and an error term $\Phip$ which is estimated in Proposition \ref{LPhip}. 

\begin{lemma}
\label{Lphiavg}
Given $\bsigmaunder = (\bsigma, \xibold)\in \R^{k-1}\times \R^k$ as in \ref{Prl} 
and $m$ as in \ref{Aratio}, 
there is a unique $\grouptwo$-invariant LD solution (recall  \ref{LLDexistence})
\begin{align}
\label{dLDPhi}
 \Phi = \Phi [ \bsigmaunder: k,m] := \varphi[L; \tau'],
 \end{align}
characterized by the requirements that
 \begin{enumerate}[label=(\alph*)]
 \item $\phi = \phi[\bsigmaunder: k,m]:=\Phi_\ave$ is a multiple of $\phat[ \bsigmaunder: k]$
 \item $L = L[ \, \sbold[ \bsigmaunder: k ]\, ; m\,  ]$ (recall \ref{dL}) 
 \item $\tau'_1 = 1$ (normalizing condition).
 \end{enumerate}
 
Moreover,  the following hold.
 \begin{enumerate}[label=(\roman*).]
\item  For $i\in \{1, \dots, k\}$  we have 
$\tau_i' = \displaystyle{  \frac{\phi(\sss_i)}{m}F^\phi_i}$. 
Moreover $\tau_i'$ is independent of $m$ and satisfies 
$\tau_i' = \tau_i' [ \bsigmaunder: k] := 
\displaystyle{  \frac{\phat[\bsigmaunder: k](\sss_i)}{\phat[\bsigmaunder: k ](\sss_1)}  \left( e^{\sum_{l=1}^{i-1}\sigma_l} \right) } $.
\item  $\displaystyle{\phi[\bsigmaunder: k,m] \, = \, 
\frac{m}{\, \phat[\bsigmaunder:k](\sss_1)\,   F^{\phat[\bsigmaunder:k]}_1\,}\, \phat[\bsigmaunder: k ]}$.
\item  On $\Omega[\sss_i;m]$, $\phi = \phiunder_i + \junder_i$, where
\begin{equation}
\label{Ephij}
\begin{aligned}
\phiunder_i &:=  
\phiunder \left[ \phi(\sss_i),\symphi ; \sss_i\right] =  \phi(\sss_i)\, \phiunder \left[1, \textstyle{\frac12}\left(F^\phi_+(\sss_i) - F^\phi_-(\sss_i)\right) ; \sss_i\right],\\
\junder_i &:=  \junder \left [\textstyle{\frac{\tau'_i}{2}}; \sss_i\right].
\end{aligned}
\end{equation}
\end{enumerate}
\end{lemma}

\begin{proof}
Suppose $\Phi$ is as in \eqref{dLDPhi} and satisfies (a)-(c).  
Let $c$ be such that $\phi=c\phat$ and  $i\in \{1, \dots, k\}$. 
Using Lemma \ref{Lvbal} to solve for $\tau'_i$, we immediately conclude $\tau'_i = \phi(\sss_i) F^\phi_i /m$; 
furthermore using Lemma \ref{Lvbal}, (a)-(c) above, and \eqref{Exi}, we compute
 \begin{equation}
  \label{Evbalphi}
 \begin{aligned}
\tau'_i &= \frac{\tau'_i}{\tau'_1} = \frac{\phi(\sss_i)}{\phi(\sss_1)} \frac{F^{\phi}_i}{ F^{\phi}_1} = \frac{\phat(\sss_i)}{\phat(\sss_1)} \frac{ F^\phat_i}{F^\phat_1} = \frac{\phat(\sss_i)}{\phat(\sss_1)}  \left( e^{\sum_{l=1}^{i-1} \sigma_l} \right) ,\\
1 &= \tau'_1 = \frac{\phi(\sss_1)}{m} F^{\phi}_1 = \frac{ c \phat(\sss_1)}{m} F^\phat_1.
 \end{aligned}
 \end{equation}
We conclude from these equations that (a)-(c) imply (i) and (ii).  
In particular, the second equation in (i) determines $\tau'$ and hence uniqueness follows from Lemma \ref{LLDexistence}.   

To prove existence we define $L$ by (b) and $\tau'$ by the second equation in (i). 
We then define $\Phi$ by \eqref{dLDPhi} and we verify that $\Phi_\ave = c\phat$, where $c$ is defined by $c \phat(\sss_1)  F^{\phat}_1 = m$: 
Let $i\in\{1, \dots, k\}$. By Lemma \ref{Lvbal}, it follows that $m \tau'_i = \Phi_\ave(\sss_i) F^{\Phi_\ave}_i$.  
By the definitions of $\tau'_i$ and $c$, we have $ m \tau'_i = c \phat(\sss_i ) F^\phat_i$.  
Since $F^\phat_i = F^{c\phat}_i$, by equating the right hand sides of the preceding equations, 
we conclude that the function $f := c\phat - \Phi_{avg}$ satisfies 
\begin{align}
\partial_+f(\sss_i)+\partial_-f(\sss_i) = 0, \quad i=1, \dots, k.
\end{align}
This amounts to the vanishing of the derivative jumps of $f$ at each $\sss_i$.  
Clearly $f$ is smooth at the poles and satisfies $\Lcal_\chi f = 0$ in between the $\sss_i$.  
Hence we conclude $f\in C^\infty(\Spheq)$ and satisfies $\Lcalp f = 0$ everywhere.  
By the symmetries of $f$, we conclude $f = 0$.

It remains to check (iii).  
The vertical balancing equation \eqref{Evbal} implies
\begin{align*}
\label{eqn: LD vertical balancing}
m \tau'_{i} =  \partial_+ \phi(\sss_i) + \partial_-\phi(\sss_i),
\end{align*}
so from the definition of $\junder$ in \ref{dauxode},
\[ \partial_+\junder_i (\sss_i)  = \partial_-\junder_i (\sss_i)= \frac{\partial_+\phi+\partial_-\phi}{2}(\sss_i). \] 
Therefore, $\phi - \junder_i$ satisfies
\begin{align*}
\partial_+ (\phi - \junder_i)(\sss_i) = \frac{ \partial_+ \phi - \partial_- \phi}{2}(\sss_i) = - \partial_-(\phi - \junder_i)(\sss_i).
\end{align*}
Hence, $\phi - \junder_i\in  C^1_{\sss}\left(\Omega[\sss_i;m]\right)$ and $\Lcal_\chi(\phi-\junder_i)  = 0$.  
By uniqueness of ODE solutions,  $\phi - \junder_i = \phiunder_i$. 
The second expression in \eqref{Ephij} follows from this by \ref{Rauxode}.
\end{proof}

%
%
\begin{definition}
\label{dGhat}
For $i=1, \dots, k$ we define numbers $A_i: = \tau'_i \log \delta$ (recall \eqref{Edelta}).  
We define $\Pp \in C^{\infty}_{\sym} \left(\cyl \setminus L\right)$ by requesting that 
it is supported on $D^\chi_L(3\delta)$ where it is defined by 
\begin{align*}
\Pp = 
  \Psibold[2\delta, 3\delta; \dbold^\chi_{p_i}](\tau'_i G_{p_i} - \phiunder[A_i,0; \sss_i],  0)
 \quad &\text{on } \quad  D^\chi_{p_i}(3\delta), 
\end{align*} 
for $i=1, \dots, k $, 
where the discs $D^\chi_{p_i}(3\delta)$ are disjoint by 
Assumption \ref{Aratio} and Proposition \ref{Prl}. 
\end{definition}

\begin{lemma}
\label{LGhatest}
The following hold.
\begin{enumerate}[label=(\roman*).]
\item  $\left \|\Pp: C_\sym^{j}( \, \cyl \setminus D^\chi_L(\delta)\,,\chitilde \, )\right\|\le C(j)\,$ and $\Pp$ is supported on $D^\chi_L(3\delta)\setminus L$.

\item  $\left\| \Acalsssi \Pp : C_\sym^j(   D^\chi_L(3\delta)\setminus D^\chi_{L}(2\delta) \, , \chitilde \, ) \right\| \le C(j)/m$.
\end{enumerate}
\end{lemma}
\begin{proof}
The statement in (i) on the support follows from Definition \ref{dGhat} and the corresponding estimate follows from Lemma \ref{Lgreen} and Definition \ref{dGhat} (observe the scale is such that the cutoff $\Psibold$ has all derivatives bounded in the $\chitilde$ metric).  For (ii), because $\Pp$ is supported on $D^\chi_L(3\delta)\setminus L$, it will suffice to prove the estimate 
\begin{align*}
\label{Egasymmetry}
\left\| \Acalsssi \Pp : C^j( D^\chi_{p_i}(3\delta)\setminus D^\chi_{p_i}(2\delta), \chitilde \, )\right\| \le C(j)/m.
\end{align*}
Recalling Definition \ref{dGhat}, we have on $D^\chi_{p_i}(3\delta)\setminus D^\chi_{p_i}(2\delta)$
\begin{equation}
\label{EAG}
\begin{aligned}
\Acalsssi \Pp &=  \Psibold[2\delta, 3\delta; \dbold^\chi_{p_i}](\tau'_i \Acalsssi \, G_{p_i}, 0) - 
	 \Psibold[2\delta, 3\delta; \dbold^\chi_{p_i}]( \Acalsssi\,  \phiunder[A_i, 0; \sss_i], 0)\\
	 &:= (I)+(II).
 \end{aligned}
 \end{equation}
From Lemma \ref{Lgreen}.(iii) and the uniform bounds on the cutoff $\Psibold$, we have $\| (I):  C^j( D^\chi_{p_i}(3\delta)\setminus D^\chi_{p_i}(2\delta) \, , \chitilde \, ) \| \le C(j)/m$.  By Lemma \ref{Lode}.(iii), $\| (II):  C^j( D^\chi_{p_i}(3\delta)\setminus D^\chi_{p_i}(2\delta) \, , \chitilde \, ) \| \le C(j) (\log m) m^{-3}$, and these estimates complete the proof of (ii).
\end{proof}

By Definition \ref{dGhat} and Lemma \ref{Lgreen}, 
we have for each $p_i \in L$ that $\Pp = \tau'_iG_{p_i}- \phiunder[A_i, 0; \sss_i]$ on $D^\chi_{p_i}(2\delta)$.  
Using this,  Definition \ref{dLD0}.(ii) and \eqref{dLDPhi}, we see that $\Phi- \Pp$ can be extended smoothly across $L$.

In the next subsection, we will estimate $\Phi$.  
The key tool is that $\Phi$ is well approximated by $\phi$ away from $L$.  
$\phi$ is well understood by the analysis in Section \ref{S:RLD}, in particular by \ref{existence}, \ref{Prl} and \ref{PODEest}.  
Lemma \ref{Lphiavg} motivates the definition of a smooth rotationally invariant function $\Phat$, 
which is a dominant term in the decomposition of $\Phi$ in \ref{dPprime}.
Note that Assumption \ref{Aratio} and Proposition \ref{Prl} imply that $\Omega[\sss_i;m] \cap \Omega[\sss_j;  m] = \emptyset$ when $i\neq j$.

\begin{definition}
\label{DPhat}
We define $\Phat \in C^{\infty}_{|\sss|}(\cyl )$ by requesting that
 \begin{align*}
\Phat := 
\Psibold\left[ \frac{2}{m}, \frac{3}{m}; \dbold^{\chi}_{\Lpar[\sss_i]}\right] \left( \phiunder_i, \phi\right) = \phi - \Psibold\left[ \frac{2}{m}, \frac{3}{m}; \dbold^{\chi}_{\Lpar[\sss_i]}\right] \left(\junder_i, 0\right) 
\end{align*}
 on $\Omega[\sss_i;m], i=1, \dots, k$ (recall \eqref{dOmega}, \eqref{Ephij})
and that $\Phat = \phi$ on $\cyl \setminus D^{\chi}_{L_{par}}(3/m)$.
\end{definition}

We are now ready to define a decomposition $\Phi = \Pp + \Phat + \Phip$.  
The third term $\Phip$ we treat as an error term to be estimated (see Proposition \ref{LPhip}).
\begin{definition}
\label{dPprime}
We define
$\Phi', E'\in C^\infty_\sym(\cyl)$ by requesting that on
$\cyl \setminus L$ 
(recall \ref{dGhat} and \ref{DPhat}),  
\[ \Phi=\Pp+\Phat+\Phip, \qquad E'=-\Lcal_{\chitilde} \left( \Pp+ \Phat\right). \] 
\end{definition}

\begin{remark}
\label{Rdecomp}
The decomposition in Definition \ref{dPprime} in some way combines the decomposition $\Phi = \Pp + \Phi''$  \cite[Definition 5.16]{kap} with the semilocal decomposition $\Phi'' = \Phip + \phiunder[\phi_1 - A_1, \hhat(\xx_1); \xx_1]$ \cite[Definition 5.25]{kap}.  There are the following differences:
\begin{enumerate}
\item 
 Our definition of $\Pp$ transits to zero away from $L$, whereas the $\Pp$ in \cite{kap} transits to $A_1$.
 \item 
 Here, $\Phip$ is defined globally, whereas in \cite{kap}, $\Phip$ is defined only on $\Omega_1[\xx_1;m]$.
 \end{enumerate}
\end{remark}
%
%
\subsection*{Estimating $\Phat$ and $\Phip$}
$\phantom{ab}$
\nopagebreak 

We estimate the average and oscillatory parts of $\Phi$ separately.  Taking averages of both sides of the equation $\Phi =\Pp+ \Phat+\Phip$ (recall Definition \ref{davg}) we have $\Phip_{avg}, \Phip_{osc}, E'_{avg}, E'_{osc} \in C^{\infty}_{sym}(\cyl)$ and 
\begin{equation}
\begin{aligned}
\label{EWdec}
 \quad \Phip= \Phip_{avg} + \Phip_{osc}, \quad E'=E'_{avg}+E'_{osc} \quad \text{on}\quad \cyl.
\end{aligned}
\end{equation}
Furthermore, since $\Lcal_\chi \Phi$ vanishes by Definition \ref{dLD0} and $\Lcal_\chi$ is rotationally covariant,
\begin{equation}
\label{ELW}
\Lcal_{\chitilde} \Phip = E',
\qquad
\Lcal_{\chitilde} \Phip_\ave = E'_\ave,
\qquad
\Lcal_{\chitilde} \Phip_\osc = E'_\osc
\quad\text{on} \quad \cyl.
\end{equation}

Because $\Lcal_\chi \phi =\Lcal_\chi \phiunder = 0$, it follows that $E'_{\ave} = \Lcal_{\chitilde} \Phip_{\ave}$ is supported on \\ $\left(D^\chi_{\Lpar}\left( 3/m\right) \setminus D^\chi_{\Lpar}\left( 2/m\right)\right) \bigcup D^\chi_{\Lpar}\left( 3\delta\right)$.
We have the following characterization of $\Phip_{avg}$.
\begin{lemma}
\label{LPhipave}
$\Phip_\ave$ is supported on $D^\chi_{\Lpar}\left(3/m\right)$.  On $\Omega[\sss_i;m], i=1, \dots, k$,
\begin{align}
\label{Ephipavg}
\Phip_{avg} = \begin{cases}
		  \Psibold\left[ \frac{2}{m}, \frac{3}{m}; \dbold^\chi_{\Lpar[\sss_i]}\right] \left( \junder_i, 0 \right) , \quad \text{on} \quad \Omega[\sss_i;m]\setminus \Omega'[\sss_i;m]\\
	 \junder_i -\Pp_\ave, \quad  \text{on} \quad \Omega'[\sss_i;m].
			\end{cases}
\end{align}
\end{lemma}
\begin{proof}
By taking averages of the equation $\Phi = \Pp + \Phat + \Phip$ and rearranging, we find
\begin{align*}
\Phip_\ave = \phi - \Phat - \Pp_\ave.
\end{align*}
Equation \eqref{Ephipavg} follows from this decomposition after substituting the expression for $\Phat$ from \ref{DPhat} and recalling that $\Pp = 0$ on $\Omega[\sss_i;m]\setminus \Omega'[\sss_i;m]$. 
\end{proof}

The decomposition $\Phi = \Pp+ \Phat + \Phip$ is designed so that $\Phip$ is small in comparison to $\Phat$ (cf. Proposition \ref{LPhip} below).  To estimate $\Phip$, we will use that it satisfies the equation $\Lcal_{\chitilde} \Phip = E'$.  First we establish relevant estimates for $E', E'_{avg}$ and $E'_{osc}$.  
\begin{lemma}
\label{LEest}
$E'$ vanishes on $D^\chi_L(2\delta)$ and $E'_{\osc}$ is supported on $D^\chi_{\Lpar}(3\delta)$. For each $i\in \{1, \dots, k\}$, the following hold.
\begin{enumerate}[label=(\roman*).]
\item 
 \begin{enumerate}[label=(\alph*)]
\item  $\left\|E':C_\sym^{j}( \Omega[\sss_i;m] ,\chitilde \, )\right\|\le C(j)\,$
\item $\left\| E'_\ave:C_\sym^{j}(\Omega[\sss_i;m]  ,\chitilde \, )\right\|\le C(j)$
\item $\left\|E'_\osc:C_\sym^{j}( \Omega[\sss_i;m],\chitilde \, )\right\|\le C(j)$
\end{enumerate}
\item  
 \begin{enumerate}[label=(\alph*)]
\item $\left\|\Acalsssi \, E' :C_\sym^{j}(  D^\chi_{\Lpar[\sss_i]}(3\delta) , \chitilde \, )\right\|\le C(j)/m\,$
\item $\left\|\Acalsssi \, E'_\ave :C_\sym^{j}(   D^\chi_{\Lpar[\sss_i]}(3\delta) \,, \chitilde \, )\right\|\le C(j)/m\,$ 
\item $\left\|\Acalsssi  \, E'_\osc :C_\sym^{j}(   D^\chi_{\Lpar[\sss_i]}(3\delta)\, , \chitilde \, )\right\|\le C(j)/m\,$.
\end{enumerate}
\end{enumerate}
\end{lemma}

\begin{proof}
The statements on the support of $E'$ and $E'_\osc$ follow from \ref{dPprime}, \ref{DPhat}, and \ref{dGhat}.
Next, note that parts (b) and (c) of items (i) and (ii) follow from part (a) of the respective items by taking  averages and subtracting, so it will suffice to prove part (a) of (i) and (iii).
It follows from \ref{dPprime} that on $\Omega[\sss_i; m]$, 
\begin{align}
\label{EE'}
E' = - \Lcal_{\chitilde}\left( \Pp + \Phat\right).
 \end{align}
On $\Omega[\sss_i; m]\setminus \Omega'[\sss_i;m]$, it follows from this, \ref{dGhat} and \ref{DPhat} that
 $E' = \Lcal_{\chitilde}  \Psibold\left[ \frac{2}{m}, \frac{3}{m}; \dbold^\chi_{\Lpar[\sss_i]}\right] \left( \junder_i, 0 \right)$.  Thus, when restricted to $\Omega[\sss_i; m]\setminus \Omega'[\sss_i;m]$, the bound in (i).(a)  follows from \ref{dauxode} and the uniform bounds of the cutoff.  By \ref{DPhat} and \ref{dGhat}, $E'$ vanishes on $\Omega[\sss_i;m]\setminus D^{\chi}_{\Lpar[\sss_i]}(3\delta)$.
On $D^{\chi}_{\Lpar[\sss_i]}(3\delta)$, note that $\Lcal_{\chitilde} \Phat=0$. Since $\Lcal_{\chitilde} \Pp = 0$ on  $D^{\chi}_{L}(2\delta)$, when restricted to $D^{\chi}_{\Lpar[\sss_i]}(3\delta)$, the required bound in (i).(a) follows from Lemma \ref{LGhatest}.(i).  

For (ii).(a), consider $D^\chi_{\Lpar[\sss_i]}(3\delta)$.  Applying $\Acalsssi$ to both sides of \eqref{EE'} yields that $\Acalsssi\,  E' = -\Acalsssi \, \Lcal_{\chitilde}\,  \Pp$ on $D^\chi_{\Lpar[\sss_i]}(3\delta)$.  Since $E'$ vanishes on $D^\chi_{\Lpar[\sss_i]}(2\delta)$, it is only necessary to prove the estimate on $D^\chi_{\Lpar[\sss_i]}(3\delta)\setminus D^\chi_{\Lpar[\sss_i]}(2\delta)$.  Using \ref{Rprod}.(ii) to switch the order of $\Lcal_{\chitilde}$ and $\Acalsssi$, we find
\begin{align*}
\Acalsssi\,  E' = - \Lcal_{\chitilde}\,  \Acalsssi \, \Pp - 2m^{-2} \Acalsssi \,[ \sech^2 \sss] \, \Rcalsssi \, E' \quad \text{on} \quad D^\chi_{\Lpar[\sss_i]}(3 \delta) \setminus D^\chi_{\Lpar[\sss_i]}(2\delta).
\end{align*}
The estimate in (ii).(a) follows from this equation after using Lemma \ref{LGhatest}.(ii) above to estimate the first term and Lemma \ref{Rprod}.(iii) and (i).(a) above to estimate the second. 
\end{proof}

\begin{lemma}
\label{Lsol}
Given 
$E\in C^{0,\beta}_\sym(\cyl)$ with $E_\ave\equiv 0$ and $E$ supported on $D^\chi_{\Lpar[\sss_i]}(3\delta)$ for some $i\in \{1, \dots, k\}$, there is a unique $u \in C^{2,\beta}_\sym(\cyl)$ solving $ \Lcal_{\chitilde} u = E $ and satisfying the following.
\begin{enumerate}[label=(\roman*).]
\item  $u_\ave=0$.
\item  $ \left \|u: C^{2, \beta}_{sym}(\cyl, \chitilde, e^{-m\left| |\sss| - \sss_i \right|} )\right\| \leq C
\left \|  E: C^{0, \beta}_{sym}\left(D^\chi_{\Lpar[\sss_i]}(3\delta), \chitilde\,\right)\right\|.$ 
\item $\left\|\Acalsssi      u:C^{2,\beta}_\sym(D^\chi_{\Lpar[\sss_i]}(3\delta)\, ,\chitilde \, )\right\| 
\le \\ \hspace*{1cm} Cm^{-3} \left\| E : C_\sym^{0,\beta }(D^\chi_{\Lpar[\sss_i]}(3\delta), \chitilde \, )\right \|
+C\left\| \Acalsssi E: C_\sym^{0,\beta }(D^\chi_{\Lpar[\sss_i]}(3\delta)\, ,\chitilde \, )\right\|.$
\end{enumerate}

\end{lemma}
\begin{proof}
Since the kernel of $\Lcal_{\chitilde} : C^{2,\beta}_\sym(\Spheq)\rightarrow C^{0,\beta}_\sym(\Spheq)$ is trivial by the symmetries (recall Lemma \ref{LLDexistence}), the existence and uniqueness of $u$ is clear.  Since $\Lcal_{\chitilde}$ is rotationally covariant, $\Lcal_{\chitilde} u_{avg} = E_{avg} = 0$ and (i) follows. 

The equation $ \Lcal_{\chitilde} u = E $ is equivalent to 
\begin{align}
\label{Ecylu}
\left( \Delta_{\chi}  + 2\sech^2 \sss \right) u =  m^2E.
\end{align}
We use separation of variables to establish the estimate (ii) (see \cite[Proposition 5.15]{Wiygul:s} 
for a similar technique). Recall $\{1, (\cos n \theta)_{n\in \N},  (\sin n \theta)_{n\in \N}\}$ is a basis for $L^2(\Sph^1)$.  Define $L^2$-projections
\begin{equation}
\label{EL2proj}
\begin{aligned}
E_0(\sss) &= \frac{1}{2\pi} \int_0^{2\pi} E(\sss, \theta) d\theta\\
E_{n, even}(\sss) &= \frac{2}{\pi} \int_0^{2\pi} E(\sss, \theta) \cos (n\theta) d\theta\\
E_{n, odd}(\sss) &= \frac{2}{\pi} \int_0^{2\pi} E(\sss, \theta) \sin (n\theta) d\theta.
\end{aligned}
\end{equation}
The assumption $E_{\ave} = 0$ implies $E_0 \equiv 0$.  By the symmetries, $E_{n, odd} \equiv 0$ and $E_{n, even} \equiv 0$ for all $n$ such that $m$ does not divide $n$.  The Fourier expansion of $u$ then satisfies 
\begin{align}
u(\sss, \theta) = \sum_{q=1}^{\infty} \widehat{u}_{mq, even}(\sss) \cos \left(mq\theta\right)+ 
\sum_{q=1}^{\infty} \widehat{u}_{mq, odd}(\sss) \sin \left(mq\theta\right)
\end{align}
for appropriate functions $\widehat{u}_{n, even}(\sss), \widehat{u}_{n, odd}(\sss)$ described as follows.  Separating variables leads us to consider the eigenspace $\left \{  u \in C^\infty_\sss (\cyl): \Lcal_\chi u = n^2 u\right\} $,
which is spanned by
\begin{align}
u_{\pm n} = (\pm n - \tanh \sss) e^{\pm n\sss}. 
\end{align}
From this, it can be checked that for any $n\geq 2$, $\Lcal_\chi - n^2$ has associated Green's function
\begin{align}
\label{EGreen}
\Gcal_{n}(\sss, \xi) = 
\frac{1}{2n(1-n^2)}
\begin{cases}
e^{n(\sss - \xi)} (n+\tanh \xi)(n - \tanh \sss)  \quad \text{for} \quad \sss\leq \xi \\
e^{n(\xi - \sss)}(n - \tanh \xi)(n+ \tanh \sss) \quad \text{for} \quad \sss \geq \xi. 
\end{cases}
\end{align}
For each $q\in \N$, we have then using that $E$ is supported on $\Omega[\sss_i;m]$
\begin{equation}
\begin{aligned}
\label{Eun}
\widehat{u}_{mq, even}(\sss) &=m^2 \int_{-\infty}^{\infty} \Gcal_{mq}(\sss, \xi) E_{mq, even}(\xi) d\xi\\ 
&= \frac{2m^2}{\pi} \int_{ \sss_i - 3\delta}^{\sss_i+3 \delta}\int_{0}^{2\pi} \Gcal_{mq}(\sss, \xi) E(\xi, \theta) \cos(mq \theta) d\theta d\xi ,\\
\widehat{u}_{mq, odd}(\sss) &= \frac{2m^2}{\pi} \int_{ \sss_i - 3\delta}^{\sss_i+3 \delta}\int_{0}^{2\pi} \Gcal_{mq}(\sss, \xi) E(\xi, \theta) \sin(mq \theta) d\theta d\xi.
\end{aligned}
\end{equation}

A straightforward estimate of \eqref{Eun} implies that for all $q\in \N$,
\begin{equation}
\begin{aligned}
\left| \uhat_{mq, even} \right| \le \frac{C}{q^2}\left \| E: C^0_{sym}(D^\chi_{\Lpar[\sss_i]}(3\delta), \chitilde \, )\right\|,\\ 
\left| \uhat_{mq, odd} \right| \le \frac{C}{q^2}\left \| E: C^0_{sym}(D^\chi_{\Lpar[\sss_i]}(3\delta), \chitilde \, )\right\|.
\end{aligned}
\end{equation} 
Therefore, 
\[ \left\|u: C^{0}_{sym}\left(D^\chi_{\Lpar[\sss_i]}(3\delta), \chitilde\,  \right)\right\| \leq C\left\| E:  C^{0}_{sym}\left(\Omega[\sss_i;m], \chitilde\,\right)\right\|.\] 
Since $\Lcal_{\chitilde} u = E$, the $C^0$ bound above implies with the Schauder estimates that
\begin{align}
\label{Eu'schauder}
\left\|u: C^{2, \beta}_{sym}\left(D^\chi_{\Lpar[\sss_i]}(3\delta), \chitilde \right) \right\| \leq C\left\|E:  C^{0, \beta}_{sym}\left(D^\chi_{\Lpar[\sss_i]}(3\delta), \chitilde\,  \right)\right\|.
\end{align} 
Combining \eqref{Eu'schauder} with the exponential decay from \eqref{EGreen} yields
\begin{align}
\label{Eu'cylinderest}
\left\| u: C^{2, \beta}_{sym}\left(\cyl, \chitilde, e^{-m\left| |\sss| - \sss_i \right|}\right)\right\| 
\leq C \left\|  E: C^{0, \beta}_{sym}\left(D^\chi_{\Lpar[\sss_i]}(3\delta), \chitilde\, \right)\right\|.
\end{align}
This proves (ii). 
Applying $\Acalsssi$ to both sides of \eqref{Ecylu} and using Lemma \ref{Rprod}.(iii), we obtain
\begin{align*}
\Lcal_\chitilde [  \Acalsssi u ]=- 2m^{-2} \Acalsssi [ \sech^2 \sss ] \, \Rcalsssi u +  \Acalsssi E.
\end{align*}
Although $\Acalsssi E-  2m^{-2} \Acalsssi \, [ \sech^2 \sss ] \, \Rcalsssi u$ is not supported on $D^\chi_{\Lpar[\sss_i]}(3\delta)$, it has average zero, so a straightforward modification of the argument leading up to \eqref{Eu'schauder} by replacing the assumption that the inhomogeneous term is compactly supported  with the assumption (from (ii) above) that the right hand side has exponential decay away from $D^\chi_{\Lpar[\sss_i]}(3\delta)$, we conclude that
\begin{align*}
\left\| \Acalsssi\,  u: C^{2, \beta}_\sym \left(D^\chi_{\Lpar[\sss_i]}(3\delta) , \chitilde\, \right)\right\|
&\le C\left\| \Acalsssi\,  E: C^{0, \beta}_\sym \left(D^\chi_{\Lpar[\sss_i]}(3\delta) , \chitilde\, \right)\right \|\\
&+ C\left\|m^{-2} \Acalsssi \, [ \sech^2 \sss ] \Rcalsssi u: C^{0, \beta}_\sym \left(D^\chi_{\Lpar[\sss_i]}(3\delta) , \chitilde\, \right)\right \|.
\end{align*}
(iii) follows after using \eqref{E:norm:mult}, Lemma \ref{Rprod}.(iii), and part (i) above to estimate the last term.
\end{proof} 

Recall that $E'_\osc$ is supported on $D^\chi_{L_{par}}(3\delta)$.  With this and Lemma \ref{Lsol} in mind, we make the following decompositions. 
\begin{definition}
\label{Deosc}
For $i=1, \dots, k$, we define $E'_{\osc, i}\in C^{\infty}_{\sym}(\cyl)$ and $\Phip_i \in C^{\infty}_{\sym}(\cyl)$ by requesting that $E'_{\osc,i}$ is supported on $D^\chi_{\Lpar[\sss_i]}(3\delta)$ and that 
\begin{align*}
\sum_{i=1}^k E'_{\osc, i} = E'_\osc, \quad \sum_{i=1}^k \Phip_{\osc, i} = \Phi'_\osc,
\quad \Lcal_{\chitilde} \, \Phip_{\osc, i} = E'_{osc, i}.
\end{align*}
\end{definition}
The functions $\Phi'_i$ are well-defined by Lemma \ref{Lsol}.  By combining Definition \ref{Deosc} with Lemma \ref{Lsol}, we get global estimates for $\Phip_\osc$.  
\begin{lemma}  
\label{LPhip est}
The following estimates hold. 
\begin{enumerate}[label=(\roman*).] 
\item \begin{enumerate}[label=(\alph*)]
	\item $\left\| \Phip_{\osc, i} : C^{j}_\sym\left( \cyl, \chitilde, e^{-m\left| |\sss| - \sss_i\right|}\right)\right\|\le C(j)$.
	\item 
	 $\left\| \Acalsssi \Phip_{\osc, i} : C^{j}_\sym\left(D^\chi_{\Lpar[\sss_i]}(3\delta), \chitilde\, \right)\right\| \le C(j)/m$.
	 \end{enumerate}
\item  $\left\| \Phip_{osc} : C^{j}_\sym\left( \cyl, \chitilde\, \right)\right \| \le C(j).$

\item $\left\| \Acalsssi \Phip_\osc : C^{j}_\sym\left(D^\chi_{\Lpar[\sss_i]}(3\delta), \chitilde\, \right)\right\| \le C(j)/m$. 
\end{enumerate} 
\end{lemma}
\begin{proof}
(i) follows directly from applying Lemma \ref{Lsol} to $E'_{\osc, i}$, using Lemma \ref{LEest} and Schauder regularity for the higher derivative estimates.  
For small $k$, (ii) follows from (i).  On the other hand, for $k$ large enough in absolute terms, Lemma \ref{Prl}.(iv) implies that for all $i, j \in \{1, \dots, k\}$, $|\sss_j - \sss_i| > \frac{3|j-i|}{4k}$.  Using this with part (i) above, we estimate
\begin{align*}
\left\| \Phip_{osc} : C^{j}_\sym\left( \cyl, \chitilde\,\right)\right \| &\le  C(j) \sup_{\sss\in \R} \sum_{i=1}^k e^{-m\left| |\sss| - \sss_i\right|} \\
&\leq C(j) \sum_{l=0}^{k-1} e^{- \frac{3m}{4k} l }\\ 
&\leq C(j),
\end{align*}
where we have used Assumption \ref{Aratio}.  This completes the proof of (ii).

Now fix some $i\in \{1, \dots, k\}.$
As in part (i), we may assume that for all $i, j \in \{1, \dots, k\}$, $|\sss_j - \sss_i| > \frac{3|j-i|}{4k}$.  Using the definitions and (i) above, 
\begin{align*}
\left\| \Acalsssi \Phip_\osc : C^{r}_\sym\left(D^\chi_{\Lpar[\sss_i]}(3\delta), \chitilde\,\right)\right\| &\leq \left\| \Acalsssi \Phip_{\osc, i}: C^{r}_\sym\left(D^\chi_{\Lpar[\sss_i]}(3\delta), \chitilde\,\right)\right\| \\ 
&+\sum_{ j\neq i}\left \| \Acalsssi \Phip_{\osc, j} : C^{r}_\sym\left(D^\chi_{\Lpar[\sss_i]}(3\delta), \chitilde\,\right)\right\| \\
&\leq \frac{C(r)}{m} + C(r)\sum_{ j\neq i}\left \|  \Phip_{\osc, j} : C^{r}_\sym\left(D^\chi_{\Lpar[\sss_i]}(3\delta), \chitilde\,\right)\right\| \\
&\leq \frac{C(r)}{m} + C(r) \sum_{j\neq i} e^{-m\left| \sss_j - \sss_i\right|}\\
&\leq \frac{C(r)}{m}+ C(r) \sum_{l=1}^k e^{-\frac{3m}{4k}l } \\
&\leq \frac{C(r)}{m},
\end{align*}
where we have used Assumption \ref{Aratio}.
\end{proof}
We culminate our understanding of $\Phip$ with the following estimates.  
\begin{prop}  
\label{LPhip}
The following hold. 
\begin{enumerate}[label=(\roman*).]
\item $\| \Phip: C^{j}_\sym\left( \cyl, \chitilde\,\right) \| \le C(j)$. 
\item  For $i\in \{ 1, \dots, k\}$, $\left\|\Acalsssi \Phip : C^{j}_\sym\left( D^\chi_{\Lpar[\sss_i]}(3\delta), \chitilde\, \right)\right \| \le C(j)/m$.
\end{enumerate}
\end{prop}
\begin{proof}
Because of the estimates on $\Phip_\osc$ established in Proposition \ref{LPhip est}.(ii), it is enough to prove the estimate (i) for  $\Phi'_\ave$.  By Proposition \ref{LPhipave}, $\Phip_\ave$ is supported on $D^\chi_{L_{par}}\left(3/m\right)$.  Fix $i\in \{1, \dots, k\}$.  We first establish the estimate on $\Omega[\sss_i;m]$.  
Equation \eqref{Ephipavg} shows that there,
\[ \Phip_\ave = \junder \left [\textstyle{\frac{\tau'_i}{2}}; \sss_i\right]  -\Pp_\ave.\] 
Note that the left hand side is smooth and the discontinuities on the right hand side cancel.  Using that $\Lcal_\chitilde \Phip_\ave = E'_\ave$ from \eqref{ELW}, on  $\Omega[\sss_i;m]$ we have (where $\shat = \shat\, [ \sss_i]$ is as in \ref{dchi})
\begin{align}
\label{Ephipode}
\partial^2_{\,\shat} \Phip_\ave + \frac{2}{m^2}\sech^2\left( \frac{\shat}{m}+ \sss_i\right)  \Phip_\ave =  E'_\ave.
\end{align}
On a neighborhood of $\partial \Omega[\sss_i;m]$, we have that $\Pp_\ave = 0$ from Definition \ref{dGhat}.  This combined with estimates on $\junder$ from Lemma \ref{Lode} implies that $\left|\Phip_\ave\right|<C$ and $\left| \partial_{\, \shat} \, \Phi'_\ave\right|<C$ on $\partial \Omega[\sss_i;m]$.  Using this as initial data for the ODE and bounds of the inhomogeneous term from Lemma \eqref{LEest} yields the $C^2$ bounds in (i).  Higher derivative estimates follow inductively from differentiating \eqref{Ephipode} and again using Lemma \ref{LEest}.  This establishes (i) on $\Omega[\sss_i;m]$.  The proof of the estimate (i)  on $\Omega[\sss_i;m]\setminus \Omega'[\sss_i;m]$ follows in a similar way using \eqref{Ephipavg} but is even easier since there $\Pp_\ave = 0$, so we omit the details.

As in the proof of (i), Proposition \ref{LPhip est}.(iii) implies it is sufficient to prove the estimate in (ii) for $ \Acalsssi \Phip_\ave$.
By Lemma \ref{Rprod}.(iii), $\Acalsssi \Phip_\ave$ satisfies
\begin{align}
\label{Eu}
\partial^2_{\, \shat} \, \Acalsssi \Phip_\ave +  \frac{2}{m^2}\sech^2\left( \frac{\shat}{m}+ \sss_i\right)\Acalsssi \Phip_\ave+
\frac{2}{m^2}\Acalsssi \, [ \sech^2 \sss] \left( \frac{\shat}{m}+\sss_i\right) \Rcalsssi \Phip_\ave = \Acalsssi  \,  E'_\ave. 
\end{align}
The $C^2$ bounds in (ii) follow in a similar way as above, by using Lemma \ref{Lode}.(iii)-(iv) to estimate the initial data on $\partial D^\chi_{\Lpar[\sss_i]}(3\delta)$, estimates on $\Acalsssi  E'_\ave$ from Lemma \ref{LEest}.(iv), and estimates on $\Acalsssi\, [ \sech^2 \sss]$ and $\Phip_\ave$ from Lemma \ref{Rprod}.(iii) and (i) above.  Higher derivative bounds follow inductively from  differentiating \eqref{Eu} and using Lemma \ref{Rprod}.(iii) and Lemma \ref{LEest}.(iv).
\end{proof} 
%
%
\section{Matched Linearized Doubling (MLD) solutions}
\label{S:MLD}
\subsection*{Mismatch and the spaces $\skernel[L]$ and $\skernelv[L]$}
$\phantom{ab}$
\nopagebreak

\begin{definition}[Mismatch of LD solutions, {\cite[Definition 3.3]{kap}}] 
\label{Dmismatch}
We define a vector space $\val[L]$ and 
given $\varphi$ as in Definition \ref{dLD0} with $\tau>0$, the \emph{mismatch} $\Bcal_L \varphi$ of $\varphi$, by
\begin{gather*}
\Bcal_L \varphi : = \left( \Bcal_p \varphi \right)_{p\in L} \, \in \, \val[L]:= \bigoplus_{p\in L} \val[p], 
\\ \text{ where } \quad 
\Bcal_p \varphi : =  \left( \varphihat_p(p) + \tau_p \log (\tau_p/2), d_p \varphihat_p\right) \in \val[p]:=\R\oplus T^*_p\Spheq .
\end{gather*}
\end{definition}

Among all LD solutions modulo $\skernel[L]$, we are mainly interested in the ones which are well matched: 
\begin{definition}[MLD solutions, {\cite[Definition 3.4]{kap}}] 
\label{DMLD}
We define a \emph{matched linearized doubling (MLD) solution modulo $\skernel[L]$ of configuration $(L, \tau, w)$} 
to be some $\varphi$ as in \ref{dLD0} which moreover satisfies the conditions 
$\Bcal_L\varphi = \zerobold $ and $\tau_p >0$ $\forall p \in L$.
\end{definition}

\begin{definition}[The spaces {$\skernel[L]$} and {$\skernelv[L]$}] 
\label{dkernel}
Given $L = L[\sbold; m]$ as in \ref{dL}, 
we define  
\begin{equation}
\begin{aligned}
\label{Ekhat}
\skernel_\sym[L]  = \text{\emph{span}}\left( \left(W_i\right)_{i=1}^k, \left( W'_i \right)_{i=1}^k\right) 
&\subset C^\infty_{\sym}(\Spheq),\\  
\skernelv_\sym[L] = 
\text{\emph{span}}\left( \left(V_i\right)_{i=1}^k, \left( V'_i \right)_{i=1}^k\right) 
&=  
\left\{v\in C^\infty_\sym(\Spheq):\Lcalp v\in\skernel_\sym[L] \right\},
\end{aligned}
\end{equation}
where for $i\in \{1, \dots, k\}$ 
and with $\delta$ as in \eqref{Edelta}, 
$V_i , V'_i, W_i, W'_i \in  C^\infty_\sym(\Spheq)$ are defined by requesting that they are supported on $D^\chi_{L_i}(2\delta)$  
and 
on $D^\chi_{p_i}(2\delta)$ they satisfy   
\begin{equation}
\label{EVV}
\begin{aligned}
V_i &=V_i[\sss_i, m]: = \Psibold[ \delta, 2\delta; \dbold^\chi_{p_i}](  \phiunder[ 1, 0 ; \sss_i], 0), \qquad 
W_i = W_i[\sss_i,m]: = \Lcalp  V_i, 
\\ 
V'_i &=V'_i[\sss_i, m]: =  \Psibold[ \delta, 2\delta; \dbold^\chi_{p_i}]( \phiunder[ 0, 1 ; \sss_i ], 0), \qquad
W'_i = W'_i[\sss_i,m]: = \Lcalp V'_i. 
\end{aligned} 
\end{equation}
\end{definition}

Note that the last equality follows from the symmetries imposed. 
Note also that $\skernel_\sym[L]$ and $\skernelv_\sym[L]$ are both $2k$-dimensional 
with corresponding bases $\left( W_i, W'_i\right)_{i=1}^k$ and $\left( V_i, V'_i\right)_{i=1}^k$.

\begin{definition}
\label{dEcal}
We define a linear map $\Ecal_L : \skernelv_\sym[L] \to \val_\sym[L]$,
where $\skernelv_\sym[L]$ was defined in \eqref{Ekhat} and
$\val_\sym[L]$ is the subspace of $\val[L]$ (recall \ref{Dmismatch})
consisting of those elements which are invariant under the obvious action of $\grouptwo$ on $\val [L]$, by requesting that \[ \Ecal_L (v) = (v(p),d_pv)_{p\in L}\in\val_\sym[L].\] 
\end{definition}

In Lemma \ref{LPhiest} below, we convert estimates established in Section \ref{S:LD} for $\Phi$ on the cylinder into estimates for $\Phi$ on the sphere.  
Before doing this, we need the following lemma which compares the geometry induced by the metrics $\chi$ and $g$. 

 \begin{lemma}

\label{Lsech} Let $i \in \{1, \dots, k\}$.  There is a constant $C>0$---independent of $m$ and $k$---such that:  
 \begin{enumerate}[label=(\roman*).]
\item  For any $a, b \in (\sss_i - 3/m, \sss_i +3/m)$,
\begin{align*}
\sech^2 a \sim_{1+\frac{C}{m}} \sech^2 b. 
\end{align*}
\item When restricted to $\Omega[\sss_i;m]$, 
\[ \dbold^g_{p_i} \sim_{1+\frac{C}{m}} \sech \sss_i \, \dbold^\chi_{p_i}.\] 
\item For large enough $m$ and any $\epsilon\le\delta/2$,
\begin{align*}
D^\chi_{L_i}\left({\epsilon}/{2}\right)\subset D^g_{L_i}\left( \epsilon \sech \sss_i\right) \subset 
D^\chi_{L_i}\left(2\epsilon\right).
\end{align*}
\item  If $f\in C^j(\Omega)$, where $\Omega\subset \cyl$ is a domain such that $\sup_{p\in \Omega}|\sss(p)|  \leq \sss_k+1$
	then 
	\[ \left\| f: C^j( \Omega, g)\right\|  \sim_{C k^{j/2}} \left\| f: C^j(\Omega, \chi)\right\|.\] 
	\end{enumerate}
\end{lemma}
\begin{proof}
Fix $i\in \{1,\dots, k\}$ and suppose $\sss_i - 3/m < a < b< \sss_i + 3/m$.  Trivially, $\sech^2 b< \sech^2 a$.  On the other hand, $\partial( \sech^2 \sss) = -2 \sech^2 \sss \tanh\sss$, so by Gr{\"o}nwall's inequality,
\begin{align*}
\sech^2 a &\leq \left(\sech^2 b\right)\,  e^{\int_{a}^b 2 \tanh \sss\,  d\sss} \leq \sech^2 b\left( 1+ \frac{C}{m}\right). 
\end{align*}
This completes the proof of (i).  (ii) follows easily from (i) and that $g = \left( \sech^2 \sss\right)\chi$ (recall \eqref{Echig2}).
(iii) follows from (ii) by taking $m$ large enough.

(iv) follows by using that $g = \left( \sech^2 \sss\right)\chi$ in combination with part (iii) above and the fact (the second part of Proposition \ref{Prl}.(i)) that $\sech \sss_k \sim_{C} k^{-1/2}$.\end{proof}

\begin{lemma}
\label{LV}
For each $i=1, \dots, k$, $V_i, V_i'\in C^\infty_\sym(\cyl)$ satisfy the following.
\begin{enumerate}[label=(\roman*).]
\item  $\left\| V_i: C^{j}_\sym(\cyl, \chitilde\, )\right \| \leq C(j)$ and $\left\| V_i' : C^{j}_\sym(\cyl, \chitilde\,) \right\| \leq C(j)$.

\item $\Ecal_L$ is an isomorphism and $\left\|\Ecal^{-1}_{L}\right\|\leq C m^{2+\beta} k^\frac{2+\beta}{2}$(recall \ref{dEcal}), where $\left\|\Ecal^{-1}_{L}\right\|$ is the operator norm of $\Ecal^{-1}_L: \val_\sym[L] \rightarrow \skernelv_\sym[L]$ with respect to the $C^{2, \beta}(\Spheq, g)$ norm on the target and the maximum norm subject to the standard metric $g$ of $\Spheq$. 
\end{enumerate}
\end{lemma}
\begin{proof}

(i) follows easily from the bounds on $\phiunder$ in Lemma \ref{Lode} and the uniform bounds on the cutoff in the $\chitilde$ metric.  By the definitions above and \eqref{Echig2}, it is easy to see that $\Ecal_L$ is invertible and that
\begin{align}
\label{Eepsinverse}
\Ecal_L^{-1} \left( \left( a_i , b_i\cos\xx \, d\xx \right)_{i=1}^k\right) =  \sum_{i=1}^k a_i V_i + \sum_{i=1}^k b_i V'_i.
\end{align}
Combining (i) above with Lemma \ref{Lsech}.(iv) to switch to the $g$ metric (recall also \eqref{Egtilde} and \eqref{Echitilde}), we have
\begin{align}
\label{EVnorm}
\left\|V_i: C^{j}_\sym(\Spheq, g\, )\right\|\leq C(j)k^{j/2}m^j, \quad \left\| V_i' : C^{j}_\sym(\Spheq, g\, )\right\| \leq C(j)k^{j/2}m^j.
\end{align}
Recall that $\left\| \Ecalinv\right\|$ is computed with respect to the $C^{2,\beta}(\Spheq,g)$ norm on $\skernelv_\sym[L]$ and  
 the maximum norm on $\val_\sym[L]$, subject to the standard metric $g$ of $\Spheq$. Then (ii) follows by combining \eqref{EVnorm} and \eqref{Eepsinverse}.
\end{proof}

\subsection*{The family of MLD solutions}
$\phantom{ab}$
\nopagebreak

In this subsection we convert the LD solutions we constructed and studied in section \ref{S:LD}  
to MLD solutions. 
We first have to choose the scale of the LD solutions so that we have approximate matching. 
By a heuristic argument which we omit we find that the overall scale $\tauo$ should be given by 
\begin{align}
\label{Etau1} 
\tauo := \frac{1}{m} e^{\zeta_1} e^{-\phi(\sss_1)} = \frac{1}{m} e^{\zeta_1} e^{- \frac{m}{F^\phi_1}}, 
\end{align}
where $\zeta_1$ is an unbalancing parameter used to absorb error terms later. 
The continuous parameters of the construction are then 
$\zetabold := (\zeta_1, \bsigmaunder) = (\zeta_1, \bsigma, \xibold) \in \R \times \R^{k-1} \times \R^k$ 
where we require
\begin{align}
\label{Epbounds}
|\zeta_1| \leq \cunder_1, \quad \left|\bsigma\right|_{\ell^\infty} \leq \frac{\cunder_1}{m}, \quad \left | \xibold \right|_{\ell^\infty} \leq \frac{\cunder_1}{m},
\end{align}
where $\cunder_1>0$ will be fixed later and which we assume may be taken as large as necessary depending on $k$ but independently of $m$.  

With the overall scale $\tauo$ having been chosen, 
we define the MLD solution 
\begin{equation} 
\label{ELD}
\varphi=\tauo\Phi+\vunder 
\end{equation} 
for some $\vunder\in\skernelv[L]$ uniquely determined by the matching condition $\Bcal_L\varphi=\zerobold$: 
By the definitions $\Bcal_L\varphi=\Bcal_L(\tauo\Phi)+\Ecal_L(\vunder)$. 
Using the invertibility of $\Ecal_L$ as in \ref{LV}.(ii),  
the matching condition is equivalent then to 
\begin{equation} 
\label{ELDv}
\vunder = -\Ecal^{-1}_L \Bcal_L (\tauo\Phi).  
\end{equation} 
To record in detail the dependence on the continuous parameters we have the following. 

\begin{definition}[MLD solutions]
\label{dtau1}
We assume $\zetabold$ is given as in \eqref{Epbounds}.
Let $\phi=  \phi[\bsigmaunder: k,m] $, $\Phi = \Phi[\bsigmaunder: k,m]$,  
and $\tau_i' = \tau_i' [ \bsigmaunder: k]$ 
be as in \ref{Lphiavg}.  
We define then 
$\tauo = \tauo[\zetabold;  m]$ 
by \ref{Etau1}, 
an MLD solution 
$ \varphi = \varphi[ \zetabold; m]$ 
of configuration 
$( \, L \, ,  \, \tau \, ,  \, w \, )$ (recall \ref{LLDexistence}) 
by \ref{ELD} and \ref{ELDv}, 
where $L=L[\sbold; m]$,  
$\sbold = \sbold[ \bsigmaunder: k ]$ (recall \ref{Nphik}), 
$\tau= \tau[ \zetabold; m] : L[\sbold; m] \rightarrow \R_+$ 
is $\grouptwo$-invariant satisfying $\tau_i = \tau_1 \tau'_i$ for $i\in \{1, \dots, k \}$,   
and $w= w[ \zetabold; m] := \Lcalp    \vunder$. 
Finally we define 
$\mubold=\mubold[\zetabold; m] = \left(\mu_i[\zetabold; m]\right)_{i=1}^k\in\R^k$ 
and 
$\mubold'=\mubold'[\zetabold; m] = \left(\mu'_i[\zetabold; m]\right)_{i=1}^k\in\R^k$ 
by 
$\vunder= \sum_{i=1}^k \tau_i \mu_i V_i + \sum_{i=1}^k \tau_i \mu'_i V'_i $ 
which also implies 
$w = \sum_{i=1}^k \tau_i \mu_i W_i + \sum_{i=1}^k \tau_i \mu'_i W'_i $.  
\end{definition}

\begin{lemma}
\label{Lmatching}
Let $\varphi$ be as in \ref{dtau1}.  The equation $\Bcal_L \varphi = \zerobold$ is equivalent to the equations
\begin{align}
\label{Ephimatching1} 
0 &= \frac{m}{ F^\phi_{1}} \left( e^{-\sum_{l=1}^{i-1}\sigma_l} - 1\right)+  \frac{\Phip(p_i)}{\tau'_i}+ \mu_i +\zeta_1 + \log\left(\frac{9}{2} \tau'_i\right)- \log \sech \sss_i 
\\ 
\label{Ephimatching2} 
0 &= \frac{1}{\tau'_i}\frac{\partial \Phip}{\partial \sss}(p_i)  +  \frac{m}{2} \xi_i + \mu'_i + \frac{1}{2} \tanh \sss_i
\end{align}
for $i=1, \dots, k$.
\end{lemma}

\begin{proof}
From Definition \ref{Dmismatch}, the condition that $\Bcal_L\varphi = \zerobold$ is equivalent to the conditions that
\begin{equation*}
\begin{aligned}
\frac{1}{\tau_{p_i}} \varphihat_{p_i}(p_i) + \log\left( \frac{\tau_{p_i}}{2}\right) = 0, \qquad
\frac{1}{\tau_{p_i}} \frac{\partial \varphihat_{p_i}}{\partial \sss}(p_i)= 0, \qquad (i=1, \dots, k). 
\end{aligned}
\end{equation*}
By Definitions \ref{Dmismatch}, \ref{dGhat} and \ref{dPprime}, \eqref{ELD}, and Lemma \ref{Lgdiff}, we have on $D^{\chi}_{p_i}(3\delta)$
\begin{align*}
\frac{1}{\tau_i}\varphihat_{p_i} &= \frac{1}{\tau'_i}\Phi - G_{p_i}+ \mu_i V_i + \mu'_i V'_i + \left( G_{p_i} - G^{\Sph^2}\circ \dbold^g_{p_i}\right).
\end{align*}
Expanding $\Phi = \Pp + \Phat+ \Phip$ and using expressions for $\Pp$ and $\Phat$ from Definitions \ref{dGhat} and \ref{DPhat} gives 
\begin{align}
\label{Ephihatdecomp}
\frac{1}{\tau_i}\varphihat_{p_i}&= 
\frac{1}{\tau'_i}\phiunder_i - \phiunder[ \log \delta, 0; \sss_i]
+\frac{1}{\tau'_i}\Phip+\mu_iV_i + \mu'_i V'_i +\left( G_{p_i} - G^{\Sph^2}\circ \dbold^g_{p_i}\right).
\end{align}
Evaluating at $p_i$, using that $V(p_i) = 1$, adding $\log(\tau_{p_i}/2)$ to both sides, using \eqref{Evbal}  to see that $\frac{\phi(\sss_i)}{\tau'_i} = \frac{m}{F^{\phi}_i}$, and using Lemma \ref{Lgdiff} shows that the vertical matching equation is equivalent to
\begin{align*}
0 = \frac{m}{F^{\phi}_{i}}+ \frac{\Phi'(p_i) }{\tau'_i} + \mu_i + \log\left( \frac{\tau'_i \tau_1}{2\delta}\right) - \log \sech\sss_i.
\end{align*}
Simplifying $\frac{m}{F^{\phi}_{i}}$ using \eqref{EFxi} and expanding the last term using \eqref{dtau1} gives \eqref{Ephimatching1}.

Next, note using the definition of $\phiunder_i$ in \ref{Ephij}, the vertical balancing equation and \eqref{Exi}
\begin{align*}
\frac{1}{\tau'_i}\frac{\partial \phiunder_i}{\partial \sss}(\sss_i) &= \frac{\phi(\sss_i)}{\tau'_i} \frac{F^{\phi}_{i+} - F^{\phi}_{i-}}{2}= \frac{m}{2} \frac{F^{\phi}_{i+} - F^{\phi}_{i-}}{F^{\phi}_{i}} = \frac{m}{2}\xi_i.
\end{align*}
\eqref{Ephimatching2} follows from this decomposition, \eqref{Ephihatdecomp}, and Lemma \ref{Lgdiff}.
\end{proof}

\begin{convention}[cf. {\cite[Conventions 2.31, 4.1]{kap}}]
\label{con:alpha}
We fix some $\alpha>0$ which we will assume as small in absolute terms
as needed.  We also fix some $\beta\in(0,1)$ and $\gamma\in (1,2)$ 
satisfying $1-\frac\gamma2>2\alpha$ and $(1-\alpha)\,(\gamma-1)>2\alpha$, for example $\gamma=\frac32$.
We will suppress the dependence of various constants on $\beta$ and $\gamma$.
\end{convention}

\begin{definition}
\label{Edeltapp}
For each $p\in L$ we define
$\delta_p'=\tau_p^\gammagl$ 
where 
$\alpha$ is as in Convention \ref{con:alpha}.  For each $p\in L$, there exists a unique $i\in \{1, \dots, k\}$ such that $p\in L_i$.  
Define then $\delta_p = \left(\sech\sss_i \right) \delta$.
Define also 
\begin{equation*}
\begin{aligned}
\tau_{min}&=\min_{p\in L}\tau_p, \quad 
\tau_{max}=\max_{p\in L}\tau_p,  \quad\\
\delta_{min} &= \min_{p\in L} \delta_p, \quad
\delta_{min}'=\min_{p\in L}\delta_p'=\tau_{min}^\gammagl.
\end{aligned}
\end{equation*}
\end{definition}

\begin{lemma}
\label{LPhiest}
For $\Phi$ as in \ref{Lphiavg}, the following estimate holds.
\begin{align*}
\left\| \Phi : C^{3, \beta}_{\sym}\left( \Spheq \setminus  D^g_{L}(\delta'_{min}), g\right)\right\| 
\le C \left(   (\delta_{min}')^{-3-\beta} | \log \delta_{min}' | k^{\frac{3+\beta}{2}} + m^{4+\beta}k^{\frac{5+\beta}{2}}\right).
\end{align*}
\end{lemma}
\begin{proof}
Recall from Definition \ref{dPprime} that $\Phi = \Ghat + \Phat+ \Phip$. 
By Lemmas \ref{Lgreen} and  \ref{LGhatest}, we have 
\begin{align}
\label{Eghatcirc}
\left\| \Pp : C^{3, \beta}_\sym\left( \cyl\setminus D^\chi_{L}(\delta'_{min}), \chi\right)\right\| \leq C (\delta'_{min})^{-3-\beta}| \log \delta'_{min} | .
\end{align}
Recall from \ref{dPprime} that $\Lcal_\chitilde ( \Phat + \Phip) = - \Lcal_{\chitilde} \Pp$ and 
Since $\Lcal_{\chitilde} \Pp = 0$ on $D^\chi_{L}(2\delta)$,
it follows from \ref{LGhatest}.(i) that
\begin{align*}
\left \| \Lcal_\chitilde ( \Phat + \Phip):  C^{3, \beta}_\sym( \cyl ), \chitilde \, ) \right\| \le C.
\end{align*}
From \ref{LPhip}.(i), \ref{DPhat} and \ref{Lphiavg}.(ii), we have also
\begin{align*}
\left\| \Phat + \Phip :  C^0_\sym( \cyl,  \chitilde \, )\right \| < Cmk, 
\end{align*}
so it follows from the Schauder estimates that
\begin{align}
\label{Ephatp}
\left \| \Phat + \Phip : C^{3, \beta}_\sym( \cyl , \chi) \right \| \le C m^{4+\beta} k .
\end{align}
Now let $\Omega: = \Spheq \cap \{ |\sss| \leq \sss_k + \frac{3}{m}\}$.  After combining \eqref{Eghatcirc} and \eqref{Ephatp} and using Lemma \ref{Lsech}.(iv), we conclude  
\[ \left\| \Phi : C^{3, \beta}_{\sym}\left( \Omega \setminus  D^g_{L}(\delta'_{min}), g\right)\right\| 
\le C \left(   (\delta_{min}')^{-3-\beta} | \log \delta_{min}' | k^{\frac{3+\beta}{2}} + m^{4+\beta}k^{\frac{5+\beta}{2}}\right).
\]
It remains to verify the estimate on $\Spheq \setminus \Omega$. 
On $\Spheq \setminus \Omega$, note that $\Pp =0$ (recall \ref{dGhat}) and also that $\Phat = c \phio$ for $c$ bounded independent of $k$, which follows from \ref{DPhat} and \ref{Prl}.(iii). Using this and the exponential decay of $\Phip$ away from $\Lpar$ from \ref{LPhip est}, we conclude $\| \Phi : C^{3, \beta} ( \Spheq \setminus \Omega, g) \| \le m^{4+\beta}k$.
\end{proof}

The next lemma will be important in controlling certain error terms in the fixed point theorem.
\begin{lemma}
\label{Ltauratio}
For $\Phi$ as in Definition \ref{Lphiavg} and $1\leq j < i \leq k$, we have
\begin{align*}
\frac{\tau'_i}{\tau'_j} = \frac{\phi(\sss_i)}{\phi(\sss_j)}  \left( e^{\sum_{l=j}^{i-1} \sigma_l} \right)  \sim_{1+\frac{C}{k}} 1. 
\end{align*}
\end{lemma}
\begin{proof}
The first equality follows from \ref{Lphiavg}.(i).  
We have then
\begin{equation}
\label{Etauratio1}
\begin{aligned}
\log \frac{\tau'_i}{\tau'_{j}} &= \log\frac{\phi(\sss_i)}{\phi(\sss_{j})} + \sum_{l=j}^{i-1}\sigma_{l}=  O\left(\frac{1}{k}\right) + O\left(\frac{k \cunder_1}{m}\right),
\end{aligned}
\end{equation}
where the estimates follow from Proposition \ref{Prl}.(iii), Definition \ref{dtau1} and \eqref{Epbounds}.
\end{proof}

\begin{notation}
Given $\abold = (a_i)_{i=1}^k \in \R^k, k\geq 2$, we define $\Dcal \abold \in \R^{k-1}$ by requesting that $(\Dcal \abold)_i = a_{i+1} - a_i$, $i=1, \dots, k-1$.
It is useful to think of $\Dcal \abold$ as a discrete derivative of $\abold$.
\end{notation}

\begin{lemma}[Properties of the MLD solutions]
\label{LLD} 
Let $\zetabold$ be as in \eqref{Epbounds} and $\varphi = \varphi[\zetabold; m]$ be as in \ref{dtau1}.
For $m$ large enough as in \ref{Aratio} (depending on $\cunder_1$), the following hold:
\begin{enumerate}[label=(\roman*).]
\item $\tauo = \tauo[\zetabold; m]$ and $(\mubold, \mubold') = (\mubold[\zetabold; m], \mubold'[\zetabold; m])$ depend continuously on $\zetabold$.  

\item
 $\tauo[\zetabold; m] \sim_{C(\cunder_1)} \tauo[\zerobold; m]$ and $C(\cunder_1)>1$ depends only on $\cunder_1$. 
 
\item  (Matching estimates) There is an absolute constant $C$ independent of $\cunder_1$ such that 
	\begin{enumerate}[label=(\alph*)]
	\item  $\left|\zeta_1+\mu_1 \right| \le C$.
	\item  $ \left|  \bsigma - \frac{F^\phi_1}{m} \Dcal \mubold \right|_{\ell^\infty}\le C/m.$
	\item  $\left|\xibold+  \frac{2}{m}\mubold' \right|_{\ell^\infty} \le C/m.$
	\end{enumerate}
\item  $ \left\|  \varphi : C^{3, \beta}_{\sym}(\Spheq \setminus D^g_{L}(\delta'_{min}), g) \right\| \le \tau_{min}^{8/9}.$
\item  On $\Spheq \setminus D^g_{L}(\delta'_{min})$ we have $\tau_{max}^{1+\alpha/5} \le \varphi$.
\item  For all $p\in L$, $(\delta_p)^{-2} \left\| \varphihat_p: C^{2, \beta}\left( \partial D^g_p( \delta_p),\left( \delta_p\right)^{-2} g\right)\right\| \leq \tau_p^{1-\alpha/9}.$
	\end{enumerate}
\end{lemma}
\begin{proof}

 The continuity of the parameter dependence of $\tauo$ and $\mubold$ on $\zetabold$ then follows from Definition \ref{dtau1}, 
Proposition \ref{existence} and Proposition \ref{Prl}.  

We next prove (ii).  For convenience in this proof, denote $\phi =\left(\Phi[\bsigmaunder: k, m]\right)_{\ave}$ and $\phip = \left( \Phi[\zerobold: k,m]\right)_\ave$ and $\sbold = \sss[\bsigmaunder:k], \sbold' = \sbold[\bsigmaunder': k]$ (recall \ref{Nphik}).  From Definition \ref{dtau1}, 
\begin{align}
\label{Etauratio4}
\frac{\tau_1[\zetabold ;m]}{ \tau_1[\zerobold; m]} &= e^{\zeta_1} e^{\phi'(\sss'_1) - \phi(\sss_1)}.
\end{align}
From  \ref{Lphiavg}.(ii), Proposition \ref{PODEest}, Proposition \ref{Prl}.(i), and \ref{Epbounds}, we have 
\begin{equation}
\begin{aligned}
\label{Eheightdiff}
\left|\phi'(\sss'_1) - \phi(\sss_1)\right|&=m\frac{\left| F^\phi_1- F^{\phip}_1\right|}{F^{\phip}_1 F^{\phi}_1} \le C \cunder_1.
\end{aligned}
\end{equation}
This establishes (ii).
We next prove (iii).  
Taking $i=1$ in \eqref{Ephimatching1} we obtain 
\begin{align}
\label{Evert1}
 \mu_1 + \zeta_1=-  \log \left( \frac{9}{2}\right) - \Phip(p_1) + \log \sech \sss_1.
\end{align}
In conjunction with Proposition  \ref{LPhip}.(ii), this proves (iii).(a).  
Now suppose $i\geq 2$.  
Subtracting the instance of \eqref{Ephimatching1} evaluated at $i$ from the instance evaluated at $i-1$ gives
\begin{equation*}
\begin{aligned}
0 &= \mu_{i-1} - \mu_i+  \frac{m}{ F^\phi_{1}} 
\left(  \sigma_{i-1} + O \left( \frac{\cunder^2}{m^2}\right) \right)- \log\left(\frac{\tau'_i}{\tau'_{i-1}}\right) 
- \frac{ \Phip(p_i)}{\tau'_i} + \frac{ \Phip(p_{i-1})}{\tau'_{i-1}} + \log \frac{\sech \sss_{i}}{\sech \sss_{i-1}}.
\end{aligned}
\end{equation*}
Multiplying through by $\frac{ F^\phi_1}{m}$ and rearranging, we find
\begin{equation}
\label{Evererror}
\begin{aligned}
\frac{ F^\phi_1}{m}\left(  \mu_{i-1}-\mu_i\right)+  \sigma_{i-1}
&= \frac{ F^\phi_1}{m}
\left( \log\frac{\tau'_i}{\tau'_{i-1}} +\frac{ \Phip(p_i)}{\tau'_i} - \frac{\Phip(p_{i-1})}{\tau'_{i-1}} - \log \frac{\sech \sss_i}{\sech \sss_{i-1}} \right) 
+ O\left( \frac{\cunder_1^2}{m^3}\right).
\end{aligned}
\end{equation}
We estimate the right hand side of \eqref{Evererror}:
by Lemma \ref{Ltauratio}, Proposition \ref{Prl}, and Proposition \ref{LPhip}, 
\begin{align}
\label{Eevererror1}
 \frac{ F^\phi_1}{m} \left|\log\frac{\tau'_i}{\tau'_{i-1}}\right| \le \frac{C}{mk^2},
\quad 
 \frac{ F^\phi_1}{m}\left| \log \frac{\sech \sss_i}{ \sech \sss_{i-1}}\right| \le \frac{C}{mk},
 \quad
 \frac{F^\phi_1}{m} \left| \frac{\Phip(p_i)}{\tau'_i}  - \frac{ \Phip(p_{i-1})}{\tau'_{i-1}}\right| &\le 
\frac{C}{mk}.
\end{align}
This completes the proof of (iii).(b).

Multiplying \eqref{Ephimatching2} by $\frac{2}{m}$ and rearranging, we estimate that for 
$i\in \{1, \dots, k\}$ we have  
\begin{align}
\label{Ehor0}
\left|  \xi_i+\frac{2}{m}\mu'_i\right| \le  \frac{2}{\tau'_im} \left|\frac{\partial \Phip}{\partial \sss}(p_i)\right| +\frac{\tanh \sss_i}{m}.
\end{align}
(iii).(c) then follows from \ref{LPhip} with a constant $C$ depending on $k$.  

Next, since $\supp(V_i)\cap \supp(V_j) = \emptyset$ when $i\neq j$ and likewise for the functions $V'_i$, we estimate
\begin{align*}
\left\| \sum_{i=1}^{k}\tau_i \mu_i V_i + \sum_{i=1}^k\tau_i \mu'_i V'_i : C^{3, \beta}_\sym( \cyl, \chitilde\,  )\right\| &\le 
\max_{i=1, \dots, k} \left( \left\|\tau_i\mu_i V_i : C^{3, \beta}_\sym( \cyl, \chitilde\,  )\right\| +\left \| \tau_i\mu'_iV'_i : C^{3,\beta}_\sym( \cyl, \chitilde\, )\right\|\right).
\end{align*}
It follows from (iii) above and \eqref{Epbounds} that
 \[ |\tau_i \mu_i| \le C \cunder_1 \tauo, \quad i=1, \dots, k. \] 
 Therefore, using Lemma \ref{LV} to estimate the norm of $V_i$, we find
 \begin{align}
 \label{EVbound}
\max_{i=1, \dots, k} \left \|\tau_i\mu_i V_i : C^{3, \beta}_\sym( \cyl, \chitilde\,  )\right\|  \le C \cunder_1 \tauo .
 \end{align}
 By (iii) above, we have $|\tau_i \mu'_i | \le C \cunder_1 k^{-2} \tauo$.  Using Lemma \ref{LV} to estimate the norm of $V'_i$, we get
\begin{align}
\label{EV'bound}
 \max_{i=1, \dots, k} \left \| \tau_i\mu'_iV'_i : C^{3, \beta}_\sym\left( \cyl, \chitilde\,  \right)\right\|  \le C\cunder_1 \tauo.
 \end{align}
By combining \eqref{EV'bound}and \eqref{EVbound} and switching to the $g$ metric, we estimate using Lemma \ref{Lsech}.(iv)
\begin{align}
\label{EWest}
\left\| \sum_{i=1}^{k}\tau_i \mu_i V_i + \sum_{i=1}^k\tau_i \mu'_i V'_i : C^{3, \beta}_\sym\left( \Spheq, g\right )\right\| \le C \cunder_1 k^{\frac{3+\beta}{2}} m^{3+ \beta}\tauo.
\end{align}
Recalling from \eqref{ELD} that $\varphi = \tauo \Phi+  \sum_{i=1}^{k} \tau_i\mu_i V_i + \sum_{i=1}^k \tau_i\mu'_i V'_i$ and combining \eqref{EWest} with Lemma \ref{LPhiest} yields
\begin{align*}
\left\| \varphi : C^{3, \beta}\left( \Spheq \setminus D^g_{L}(\delta'_{min}), g\right)\right\| 
\le C\left(   (\delta_{min}')^{-3-\beta} | \log \delta_{min}' | k^{\frac{3+\beta}{2}} + m^{4+\beta}k^{\frac{5+\beta}{2}}+ \cunder_1 k^{\frac{3+\beta}{2}}m^{3+\beta} \right)\tau_1.
\end{align*}

From Definition \ref{Edeltapp} and Lemma \ref{Ltauratio}, $\delta'_{min} = \tau_{min}^{\alpha} \sim_C \tau^{\alpha}_1$.  (iv) then follows from the above by taking $m$ large enough. 
For (v),
from \eqref{EVbound}, \eqref{EV'bound}, Lemma \ref{Lgreen}.(ii) and Definition \ref{dtau1},
\begin{equation}
\label{Ephiparts}
\begin{aligned}
\left\| \Pp: C^0\left(  \Spheq \setminus D^g_L(\delta'_{min})\right)\right\|  &\le \alpha Cm k  \\ 
\left\| \sum_{i=1}^k \tau_i\mu_i V_i + \sum_{i=1}^k \tau_i\mu'_i V'_i : C^0_{sym}\left(\Spheq, g\right)\right\| &\le C\cunder_1 \tau_1\\
\left\| \Phip : C^0\left( \Spheq, g\right)\right \| &\le C.
\end{aligned}
\end{equation}
It is easy to see from  \ref{Lphiavg}.(ii) and definition \ref{DPhat} and that there is an absolute constant $c>0$ such that $ \Phat>  cmk $, so (v) follows from \ref{Ephiparts} and \ref{dtau1} by taking $\alpha$ small enough and $m$ large enough.  

Finally, let $i\in \{1, \dots, k\}$.  By \ref{dLD0}, \ref{dPprime} and \ref{Lsech}, on $D^g_{p_i}(\delta_{p_i})$, $\varphihat_{p_i}$ satisfies
\begin{equation}
\begin{aligned}
\varphihat_{p_i} &= \tau_1\left( \Phat + \Phip\right) + \tau_i' \tau_1 \left( G_{p_i} - G^{\Sph^2}\circ \dbold^g_{p_i}\right) + \sum_{i=1}^k \tau_i\mu_i V_i + \sum_{i=1}^k\tau_i \mu_i' V'_i\\
&:= (I) + (II)+(III).
\end{aligned}
\end{equation}
By \ref{Ephatp} and \ref{Lsech}, $\left \| (I): C^{2, \beta}\left( \partial D^g_{p_i}(\delta_{p_i}), (\delta_{p_i})^{-2} g\right)\right\| \le C\tau_1 mk^{2+\beta}  $.  By Lemma \ref{Lgdiff}, $G_{p_i} - G^{\Sph^2}\circ \dbold^g_{p_i}$ satisfies the equation $\Lcal'\left( G_{p_i} - G^{\Sph^2}\circ \dbold^g_{p_i}\right) = 0$ and the $C^0$ estimate 
\begin{align*}
\left\| G_{p_i} - G^{\Sph^2}\circ \dbold^g_{p_i} : C^0\left( D^g_{p_i}(\delta_{p_i}), g\right)\right\| < C|\log \sech \sss_i|
\le C k,
\end{align*}
where the last estimate follows from \ref{Prl}.(i).  
Therefore, by the Schauder estimates, we have 
\\
$\left \| (I): C^{2, \beta}( \partial D^g_{p_i}(\delta_{p_i}), (\delta_{p_i})^{-2} g)\right\| \le C \tau_1 k $.  
Additionally, by \eqref{EWest}, it follows that 
\newline 
$\left \| (III): C^{2, \beta}( \partial D^g_{p_i}(\delta_{p_i}), (\delta_{p_i})^{-2} g)\right\| \le C \cunder_1 k^{\frac{2+\beta}{2}} \tau_1$.  
(vi) then follows by taking $m$ large enough.
\end{proof}



\section{Main results with no necks at the poles or the equatorial circle}
\label{S:Main}
\subsection*{Initial surfaces from MLD solutions}
$\phantom{ab}$
\nopagebreak

In this subsection we discuss the conversion of 
the MLD solutions (constructed in the previous subsection) 
to  initial surfaces.  
We also discuss the mean curvature and the linearized equation on 
the initial surfaces constructed. 
These steps were carried out under generous assumptions for the 
MLD solutions in \cite[Sections 3 and 4]{kap} and therefore we only 
quote the results here and confirm that our MLD solutions 
satisfy the required conditions (see Lemma \ref{con:add}).

\begin{lemma}[cf. {\cite[convention 3.6]{kap}}]
\label{con:L}
For each $p\in L$,
the disks $D^g_p(9\delta_p)$ are disjoint for different points $p\in L$.
\end{lemma}
\begin{proof}
Follows from \ref{Edeltapp}, Proposition \ref{Prl} and Assumption \ref{Aratio} by taking $m$ large enough. 
\end{proof}

\begin{lemma}[cf. {\cite[convention 3.15]{kap}}]
\label{con:add}
Let $\varphi[\zetabold;m]$ be as in \ref{dtau1}.  The following hold. 
\begin{enumerate}[label=(\roman*).]
\item
$\tau_{max}$ is small enough in absolute terms as needed.

\item
$\forall p\in L$ we have $9 \delta_p' < \tau_p^{\alpha/9}< \delta_p \, $. 

\item
$\tau_{max}\le \tau_{min}^{1-\gammagl/9}$.  

\item
$\forall p\in L$ we have 
$(\delta_{p})^{-2}\left \| \, \varphihat_p 
: C^{2,\beta}(\, \partial D^g_p(\delta_p)    ,\, (\delta_p)^{-2} g\,)\,\right\| 
\le
\tau_p^{1-\gammagl/9}$.

\item
$\left\| \varphi:C^{3,\beta}_\sym ( \, \Spheq\setminus\disjun_{q\in L}D^g_q(\delta_q')    \, , \, 
g \, ) \,\right \|
\le
\tau_{min}^{8/9} \, $.

\item
On $\Spheq\setminus\disjun_{q\in L}D^g_q(\delta_q') $ we have
$\tau_{max}^{1+\alpha/5} \le \varphi$.  

\item $\Ecal_L: \skernelv_\sym[L]\to \val_\sym[L]$ is a linear isomorphism and
	$\delta_{min}^{-4} \,\tau_{max}^\alpha \,
	\left\|\Ecalinv\right\| \,
	\le \,1.$
\end{enumerate}
\end{lemma}
\begin{proof}
(i)-(iii) follow from \ref{dtau1} and \ref{Ltauratio} by taking $m$ large enough.  (iv)-(vi) follow from Lemma \ref{LLD}.(v)-(vii).  Finally, (vii) follows from \ref{LV}.(ii).
\end{proof}

Given a $\grouptwo$-symmetric MLD solution $\varphi$ of configuration $(L, \tau, w)$  as 
in \ref{dtau1}, we modify it to 
$\varphinl\in C^\infty_\sym(\Spheq\setminus L)$
as in \cite[Lemma 3.18]{kap}. 
By using cut-off functions then we define as in 
\cite[Definition 3.19]{kap} 
$$
\varphi_{init}
=
\varphi_{init}[L,\tau,w] 
:\Spheq\setminus\textstyle\bigsqcup_{p\in L} D_p(\tau_p)\to[0,\infty), 
$$
and then as in 
\cite[Definition 3.20]{kap} 
the initial smooth surface $M[L,\tau,w]$ as the union of the graphs of 
$\pm\varphiinit$. 
Recall however that our MLD solutions and their configurations 
(defined in \ref{dtau1}) 
are parametrized by $\zetabold$ (which determines also $k$) and $m$. 
Because of this we introduce the notation 
\begin{equation}
\label{dM} 
M= M[[ \zetabold]] = 
M[[ \zetabold; m ]] := M\left[ L, \tau, w\right], \quad \text{ where } 
L=L[[\zetabold]] = L[[ \zetabold; m ]] 
:= L[\, \sbold [ \bsigmaunder: k ]\, ; m], 
\end{equation} 
$\tau= \tau[ \zetabold; m]$, 
and $w= w[ \zetabold; m]$,  
are as in \ref{dtau1}.  
Note that the double brackets are introduced to distinguish from earlier notation. 
Note also that as usual the value of $k$ is implied by $\zetabold$ and we may not mention $m$ 
when it is implied by the discussion. 

In the rest of this section 
we assume throughout that $\zetabold$ satisfies \eqref{Epbounds}, $(k,m) \in \N^2$ is as in Assumption \ref{Aratio}, 
and $m$ is large enough in terms of $\cunder_1$ as needed.
Now that the initial surfaces have been defined, 
we need to discuss their mean curvature, the linearized operator on them, and the nonlinear terms in a small perturbation. 
All these have been studied in \cite[Section 4]{kap} and so we only need to quote a definition and three basic results. 
Note that $\PiSph$ was defined in \eqref{EPiSph}.  
Note also that Convention \ref{con:alpha} and Lemma \ref{con:add} imply 
all the requirements for the applicability of 
\cite[Proposition 4.17]{kap}  
so we do not mention them in \ref{Plinear}.

\begin{definition}[cf. {\cite[Definition 4.12]{kap}}]
\label{D:norm}
For $k\in\N$, $\betahat\in(0,1)$, 
$\gammahat\in\R   $,
and $\Omega$ a domain in $\Spheq$ or
an initial surface $M$ as above, 
we define
$\|u\|_{k,\betahat,\gammahat;\Omega}
:=
\left\|u:C^{k,\betahat}(\Omega ,\rho,g,\rho^\gammahat)\right\|$,
where 
$\rho:= \dbold^g_L \circ \PiSph$ 
when $\Omega\subset \Spheq$
or $\Omega\subset M$, 
and $g$ is the standard metric on $\Spheq$, 
or the metric induced on $M$ by the standard metric on $\Sph^3(1)$.
\end{definition}

\begin{lemma}[cf. {\cite[Lemma 4.15]{kap}}]
\label{LHM}
The modified mean curvature 
$
H-\xiw \circ \PiSph
$
on an initial surface $M=M[[\zetabold]]    $ as above 
is supported on 
$\PiSph^{-1}\left(\bigsqcup_{p\in L}(\,D_p^g(3\delta'_p)\,\setminus D_p^g(2\delta'_p) \,)\right)$.
Moreover it satisfies the estimate 
$$
\|\, H - \xiw \circ \PiSph\, \|_{0,\beta,\gamma-2;M} 
\le
\, \tau_{max}^{1 + {\alpha}/3} \,
.
$$
\end{lemma}

\begin{prop}[{\cite[Proposition 4.17]{kap}}]
\label{Plinear}
If $M=M[[\zetabold]]    $ is an initial surface as above,  
there exists a linear map 
$\Rcal_M: C^{0,\beta}_\sym(M) \to C^{2,\beta}_\sym(M) \times \skernel_\sym[L]$ such that if 
\[ \Rcal_M E:= (u,w) \in C^{2,\beta}_\sym(M) \times \skernel_\sym[L], \] 
then 
\[ \Lcal u= E+ w \circ \PiSph.\] 
Moreover the following hold.
\begin{enumerate}[label=(\roman*).]
\item
$
\|u\|_{2,\beta,\gamma;M}
\le  \, C(b)  \, 
\delta_{min}^{-2-\beta} \,
\left\|\Ecalinv\right\|\,
\|\, E\, \|_{0,\beta,\gamma-2;M}.
$

\item
$
\left\|w  : C^{0,\beta}(\Spheq,g)\right\|
\le  \, C \, 
\delta_{min}^{\gamma-2-\beta} \,
\|\Ecalinv\|\,
\|E\|_{0,\beta,\gamma-2;M}.
$

\item
$
\Rcal_M
$ 
depends continuously on the parameters of $\zetabold$.
\end{enumerate}
\end{prop}

Given $\upphi\in C^1_\sym(M)$, we define the normal perturbation of $\upphi$,  $I_\upphi: M \rightarrow \Sph^3(1)$,  by
\[ I_\upphi(x)=\exp_x(\,\upphi(x)\,\nu(x)\,), \]
where $\nu:M\to T\Sph^3(1)$ is the unit normal to $M$. 
If $\upphi$ is sufficiently small, the normal perturbation $M_\upphi: = I_\upphi(M)$  is an embedded surface.  
Moreover, $M_\upphi$ is invariant under the action of $\groupthree$ on the sphere $\Sph^3(1)$.

We recall the following estimate (\cite[Lemma 4.24]{kap}) on the nonlinear terms of such a perturbation.

\begin{lemma}[{\cite[Lemma 4.24]{kap}}]
\label{Lquad}
If $M$ is as in \ref{Plinear} and 
$\upphi\in C^{2,\beta}_\sym(M)$ 
satisfies $\left\|\upphi\right\|_{2,\beta,\gamma;M} \, \le \, \tau_{max}^{1+\alpha/4} $,  
then $M_\upphi$ is well defined as above and
is embedded.
If 
$H_\upphi$ is the mean curvature of $M_\upphi$ pulled back to $M$ by $I_\upphi$ 
and $H$ is the mean curvature of $M$, then we have 
$$
\left\|\, H_\upphi-\, H - \Lcal \upphi \, \right\|_{0,\beta,\gamma-2;M}
\, \le \, C \, 
\left\|\upphi\right \|_{2,\beta,\gamma;M}^2.
$$
\end{lemma}

\subsection*{The main theorem with no bridges at the poles or the equatorial circle}
$\phantom{ab}$
\nopagebreak

For the purposes of the fixed point theorem, we will need to fix a reference initial surface 
and pull back perturbations of the initial surfaces we consider to the reference surface via appropriate diffeomorphisms:  

\begin{lemma}
\label{Lrefdiff}
There exists a family of diffeomorphisms $\Fcal_\zetabold : M[[\zerobold]] \rightarrow M[[\zetabold]]$, 
where 
$\zerobold = (0, \dots, 0 )\in \R^{2k}$,  
satisfying the following: 
\begin{enumerate}[label=(\roman*).]
\item $\Fcal_\zetabold$ depends continuously on $\zetabold$.
\item $\Fcal_\zetabold$ is covariant under the action of $\groupthree$.
 
\item  For any $u\in C^{2, \beta}(M[[\zetabold]])$ and $E\in C^{0, \beta}(M[[\zetabold]])$ we have the following equivalence of norms: 
\begin{align*}
\| \, u\circ \Fcal_\zetabold \, \|_{2,\beta,\gamma;M[[\zerobold]] }
\, \sim_2 \, 
\| \, u\, \|_{2,\beta,\gamma;M[[\zetabold]] } ,
\qquad
\| \, E\circ \Fcal_\zetabold \, \|_{0,\beta,\gamma-2;M[[\zerobold]] }
\, \sim_2 \, 
\| \, E\, \|_{0,\beta,\gamma-2;M[[\zetabold]] } .
\end{align*}
\end{enumerate}
\end{lemma}

\begin{proof}
As a preliminary step, we construct a family of diffeomorphisms $\Fcal'_{L_{par}}:\Spheq \rightarrow \Spheq$ which depend smoothly on $L_{par}$ and covariant under the action of $\grouptwo$.
For ease of notation, denote the positive $\sss$-coordinates of the circles $L_{par}[[\zerobold]] $ by $\sbold$ and likewise the coordinates of the circles in $L_{par}[[\zetabold]]$ by $\sbold'$.  We define $\Fcal'_{L_{par}}$ by requesting the following:
\begin{enumerate}
\item $\Fcal'_{L_{par}}$ is rotationally invariant in the sense that $\Fcal'_{L_{par}}\left((\sss, \theta)\right)$ depends only on $\sss$. 
\item On $D^\chi_{L_{par}[\sss_i]}(5 \delta)$, we have $\Fcal'_{L_{par}}( (\sss, \theta) ) = (\sss'_i - \sss_i + \sss, \theta)$.

\item On $\Spheq \setminus D^\chi_{L[[\zerobold]]}(5\delta)$, $\Fcal'_{L_{par}}$ we have
\[ \Fcal'_{L_{par}}\left( (\sss, \theta)\right) = (f_{L_{par}}(\sss), \theta)\] 
 for a suitably chosen function $f_{L_{par}}$.   
\end{enumerate}
By choosing $f_{L_{par}}$ carefully, we can ensure that $\Fcal'_{L_{par}}$ depends smoothly on $L_{par}$, hence on $\zetabold$ and is close to the identity in all necessary norms.  Next, we use $\Fcal'_{L_{par}}$ to define $\Fcal_\zetabold$ by requesting the following.

\begin{enumerate}
\item $\forall p\in L[[\zerobold]]$ we define 
$\Fcal_\zetabold$ to map $\Lambda_\zerobold: = M[[\zetabold]] \cap \Pi^{-1}_{\Spheq} \left( D^g_{p_i}(\delta'_{p_i})\right)$ onto 
$\Lambda_\zetabold : = M[[\zetabold]] \cap \Pi^{-1}_{\Spheq} \left( D^g_{q_i}(\delta'_{q_i})\right)$ 
where $q_i = \Fcal'_{L_{par}}(p_i)$.  
On $\Lambda_\zerobold$, $\Fcal_\zetabold$ satisfies
\[ \widehat{\Fcal}_\zetabold \circ Y_\zerobold \circ \Pi_{\cat, p_i} = Y_\zetabold \circ \Pi_{\cat, q_i} \circ \Fcal_\zetabold, \]
where $Y_\zetabold$ (and similarly for $Y_\zerobold$ is the conformal isometry from $\Pi_{\cat, q_i}(\Lambda_\zetabold)$ equipped with the induced metric from the Euclidean metric 
$\left. (\tau[\zetabold, m])^{-2} g \right|_{p_i}$, to  $\{ \sss \in [-\ell_\zetabold, \ell_\zetabold]\} \subset (\cyl, \chi)$, and
\begin{equation*} 
\label{EFhat} 
\widehat{\Fcal}_{\zetabold}: 
\{ \sss \in [-\ell_\zerobold,  \ell_\zerobold ] \} 
\to 
\{ \sss \in [-\ell_\zetabold,  \ell_\zetabold ]\} 
\end{equation*} 
is of the form 
\[ \widehat{\Fcal}_{\zetabold} (\sss,\theta) =
( \, \ell_\zetabold \, \sss \,/ \,\ell_\zerobold , \, \theta\,), 
\] 
where the ambiguity due to possibly modifying the $\theta$ coordinate by adding a constant is removed by the requirement that $\Fcal_\zetabold$ is covariant with respect to the action of $\groupthree$.  
\item We define the restriction of $\Fcal_\zetabold$ on $ M[[\zerobold]] \cap \Pi_{\Spheq}^{-1} \left(\Spheq \setminus \bigsqcup_{p\in L[[\zerobold]]} D^g_p (2\delta'_p)\right)$ to be a map onto $M[[\zetabold]] \cap \Pi_{\Spheq}^{-1} \left(\Spheq \setminus \bigsqcup_{p\in L[[\zetabold]]} D^g_p (2\delta'_p)\right)$ which preserves the sign of the $\zz$ coordinate and satisfies 
\[ \Pi_{\Spheq} \circ \Fcal_\zetabold = \Fcal'_{L_{par}} \circ \Pi_{\Spheq}.\]
\item On the region $M[[\zerobold]] \cap \Pi_{\Spheq}^{-1} \left(\Spheq \setminus \bigsqcup_{p\in L[[\zerobold]]} D^g_p (2\delta'_p)\right)\setminus M[[\zerobold]] \cap \Pi_{\Spheq}^{-1} \left(\Spheq \setminus \bigsqcup_{p\in L[[\zerobold]]} D^g_p (\delta'_p)\right)$ we apply the same definition as in (2) but with $\Fcal'_{L_{par}}$ appropriately modified by using cut-off functions and $\dbold^g_{L[[\zetabold]]}$ so that the final definition provides an interpolation between (1) and (2).  
\end{enumerate}
By construction, $\Fcal_\zetabold$ satisfies (i) and (ii).  
To check the uniform equivalence of the norms, 
first observe from the parametrization of the catenoid \eqref{Ecatenoid}
and the uniform equivalence of the $\tau$s from Lemma \ref{LLD}.(ii), it follows that 
\[ \ell_\zetabold \sim_{1+ C(\cunder_1)/m} \ell_\zerobold. \]
Moreover, from Lemma \ref{LLD}.(ii), 
$\tauo[\zetabold; m] \sim_{C(\cunder_1)} \tauo[\zerobold; m]$. 
From this, it follows from arguing as in Lemma 4.13 of \cite{kap} that when $m$ is large enough in terms of $\cunder_1$, the estimates in (iii) hold.  
\end{proof}

For the convenience of the reader and to motivate the fixed point map $\Jcal$ (see \ref{EJcal}) 
used in the proof of the main theorem \ref{Tmain} 
we recall the main steps of the construction 
omitting smallness conditions and the precise ranges of the parameters: 

\begin{remark}[Outline of the construction] 
\label{Rsummary}
\textit{Step 1: Unit RLD solutions:}
Based on Lemmas \ref{LFmono} and \ref{LHmono} we construct in Proposition \ref{existence} a unit 
RLD solution $\phat[\sss_1; \bsigmaunder]$ for given $\sss_1 \in (0, \sss_{root})$ 
and 
$\bsigmaunder = (\bsigma, \xibold) \in \ell^1\left( \R^\N\right) \oplus \ell^\infty \left( \R^\N\right)$. 
\\
\textit{Step 2: Unit RLD solutions with prescribed number of jumps and flux ratios:}
Based on Proposition \ref{existence} we introduce Notation \ref{Nphik} to describe the 
RLD solutions which are smooth at the poles. 
These solutions are denoted by 
$\phat[\bsigmaunder: k]$ 
and are determined uniquely by the number of jumps $k$ and 
the corresponding $2k-1$ flux ratios 
$\bsigmaunder = (\bsigma, \xibold) \in \R^{k-1}\times \R^k$. 
Their jump latitudes are denoted by $\sbold = \sbold[\bsigmaunder: k ]\in \R^k_+$.  
\\ 
\textit{Step 3: Normalized $\grouptwo$-symmetric LD solutions:} 
Based on Lemmas \ref{LLDexistence} and \ref{Lvbal} we construct in Lemma \ref{Lphiavg} 
an LD solution $\Phi [ \bsigmaunder: k,m]$ 
of configuration $(\, L= L[\,\sbold[\bsigmaunder: k] \, ; m] , \, \tau' \, )$ 
for given $k,m\in\N$ and 
$\bsigmaunder = (\bsigma, \xibold) \in \R^{k-1}\times \R^k$. 
$\Phi = \Phi [ \bsigmaunder: k,m]$ is normalized by the condition $\tau'_1=1$ and 
is characterized by the form of the configuration (recall \ref{dL}) and the requirement that 
$\phi = \phi[\bsigmaunder: k,m]:=\Phi_\ave$ is a multiple of $\phat[ \bsigmaunder: k]$. 
Note that the 
$\tau_i' = \tau_i' [ \bsigmaunder: k]$'s 
do not depend on $m$. 
\\ 
\textit{Step 4: $\grouptwo$-symmetric MLD solutions:} 
In this step we first choose a suitable overall scale $\tau_1$ in 
\eqref{Etau1} for our LD solutions,  
and then modify them by adding $\vunder$ as in \ref{ELDv} to obtain MLD solutions 
$\varphi[ \zetabold; m]$ (see \ref{dtau1}), 
where 
$\zetabold := (\zeta_1, \bsigmaunder) = (\zeta_1, \bsigma, \xibold) \in \R \times \R^{k-1} \times \R^k = \R^{2k}$ 
involves an extra parameter $\zeta_1$ used in unbalancing related to the overall scale $\tau_1$. 
\\ 
\textit{Step 5: $\grouptwo$-symmetric initial surfaces:} 
In this step the MLD solutions are converted to initial surfaces by using the process 
developed in \cite[Section 3]{kap}. 
\\ 
\textit{Step 6: $\grouptwo$-symmetric minimal surfaces:} 
In this remaining step we use a fixed point argument 
to perturb to minimality one of the initial surfaces constructed earlier for given $k,m$.  
\end{remark}

\begin{theorem}
\label{Tmain}
There is an absolute constant $\cunder_1>0$ such that given $(k,m)\in \N^2$ with $m$ large enough in terms of $\cunder_1$ and $k$, 
there is $\zetaboldhat \in \R^{2k}$ satisfying \eqref{Epbounds} such that $\varphi[ \zetaboldhat, m]$ satisfies the conditions of Lemma \ref{LLD} 
and moreover there is $\upphihat \in C^\infty(\Mhat)$, where $\Mhat: = M[[  \zetaboldhat ; m ]]$, 
such that
\begin{align*}
\left\| \upphihat\right \|_{2, \beta, \gamma, \Mhat} \le \tauhat_1^{1+ \alpha/4},
\end{align*}
and further the normal graph $\Mhat_\upphihat$ is a genus $2mk -1$ embedded minimal surface in $\Sph^3(1)$ 
which is invariant under the action of $\groupthree$ and has area $\text{Area}(\Mhat_\upphihat)\rightarrow 8 \pi$ as $m\rightarrow \infty$.  
\end{theorem}

\begin{proof} 
The proof is based on finding a fixed point for a map 
$\Jcal : B \rightarrow C^{2, \beta}_{sym}(M[[\zerobold]])\times \R^{2k}$ 
we will define shortly, 
where 
$B \subset C^{2,\beta}_\sym(\,M[[\zerobold]]\,)\times\R^{2k}$ 
is defined by
$$
B := 
\left\{ \, v\in C^{2,\beta}_\sym(M[[\zerobold]]):\|v\|_{2,\beta,\gamma;M[[\zerobold]]} \, \le \, \tau_1[\zerobold;m]^{1+\alpha}\,\right\} 
\times [-\cunder_1,\cunder_1]\times  \left[-\frac{\cunder_1}{m},\frac{\cunder_1}{m}\right]^{2k-1}.  
$$
To motivate the definition of $\Jcal$, suppose $(v, \zetabold) \in B$.  
Use Proposition \ref{Plinear} to define $(u, w_H):= - \Rcal_{M[[\zetabold]]}\left( H - w\circ \Pi_{\Spheq}\right)$.  
Define also $\upphi \in C^{2, \beta}_{\sym}\left(M[[\zetabold]]\right)$ by $\upphi: = v\circ \Fcal^{-1}_\zetabold + u$.  
We then have: 
\begin{enumerate}
\item $\Lcal u + H = (w+w_H)\circ \Pi_{\Spheq}.$
\item By Lemma \ref{LHM} and Proposition \ref{Plinear}, 
	\begin{align*}
	\left\| w_H: C^{0, \beta}( \Spheq, g)\right\| +\left \| \upphi \right \|_{2, \beta, \gamma; M[[\zeta]]} \leq \tauo^{1+\alpha/4}.
	\end{align*}
Using \ref{Plinear} again, define $(u_Q, \mu_Q): = - \Rcal_{M[[\zetabold]]}(H_\upphi- H - \Lcal \upphi)$.  By definition,
\item $\Lcal u_Q+H_\upphi = H + \Lcal \upphi +w_Q \circ \Pi_{\Spheq}$.\newline
By \ref{Lquad}, we have the following estimate on the quadratic terms:
\item $\left\| w_Q : C^{0, \beta}( \Spheq, g)\right\| + \left\| u_Q\right\|_{2, \beta, \gamma; M[[\zetabold]]} \leq \tauo^{2+\alpha/4}$.\newline
By combining the above, we see
\item $\Lcal( u_Q - v\circ \Fcal^{-1}_{\zetabold}) + H_\upphi = (w+w_H+w_Q)\circ \Pi_{\Spheq}$.  
\end{enumerate}

We then define
\begin{align}
\label{EJcal} 
\Jcal(v, \zetabold) = \left( u_Q \circ \Fcal_{\zetabold}, \zeta_1+\mutilde_1,    \bsigma - \frac{F^\phi_1}{m} \Dcal \muboldtilde , \xibold+\frac{2}{m}\muboldtilde'
\right),
\end{align}
where $\sum_{i=1}^k \tau_i\mutilde_i W_i + \sum_{i=1}^k\tau_i \mutilde'_i W'_i = w+w_H+ w_Q$. 

We are now ready for the fixed-point argument.  Clearly $B$ is convex.  Let $\beta' \in (0, \beta)$.  By the Ascoli-Arzela theorem, the inclusion $B \hookrightarrow C^{2, \beta'}_{sym}\left( M[[ \zerobold]]\right)$ is compact.  Moreover, by \ref{LLD}.(iii) and item (4) above, when $m$ is large enough, $\Jcal(B) \subset B$.   

The Schauder fixed point theorem implies there is a fixed point $(\vhat, \zetaboldhat)$ of $\Jcal$.  By inspection of the defining formula for $\Jcal$, this implies that $\mutilde_i = \mutilde'_i = 0$ for $i=1, \dots, k$, hence $\what+ \what_H+ \what_Q = 0$ and also $\vhat = \uhat_Q \circ F_{\zetaboldhat}$.  By (5), we conclude the minimality of $\Mhat_{\upphihat}$.  

The smoothness follows from standard regularity theory and the embeddedness follows from Lemma \ref{Lquad} and (4) above.  The genus follows because we are connecting two spheres with $2km$ bridges.  Finally, the limit of the area as $m\rightarrow \infty$ follows from the bound on the norm of $\varphihat$ and from the estimates (cf. \cite[Section 3]{kap}) on the function $\varphi_{init}\left[  L[\widehat{\sbold}; m], \tauhat, \what \right]$ used to construct the initial surfaces. 
\end{proof}

\begin{remark}
\label{rtausize}
As $k\rightarrow \infty$, the sizes of the catenoidal bridges on each minimal surface in \ref{Tmain} tends to become uniform in the following sense: 
given $k\in \N$ and $\zetabold\in \R^{2k}$ as in \eqref{Epbounds}, 
it follows from Lemma \ref{Ltauratio} that 
$\limsup_{m\rightarrow \infty} \tau_{max}[\zetabold]/ \tau_{min}[\zetabold] < C/k$ 
where $C$ is a constant independent of $\zetabold$ and $k$.  
Consequently
\begin{align*}
\lim_{k\rightarrow \infty} \lim_{m\rightarrow \infty} \tauhat_{max}/\tauhat_{min} = 1,
\end{align*}
where $\tauhat_{max}/\tauhat_{min} := \tauhat_{max}[\zetaboldhat]/\tauhat_{min}[\zetaboldhat]$ 
is the ratio of the maximum size over the minimum size of the bridges associated with fixed point $(\vhat, \zetaboldhat)$ of $\Jcal$ as in \ref{Tmain}.
\end{remark}

\begin{remark} 
\label{restimprove}
For simplicity in this article we have not attempted to determine explicit bounds 
for $m$ in terms of $k$ (recall \ref{Aratio}) 
under which the conclusion of Theorem \ref{Tmain} holds. 
It would be interesting to extend our results to include the cases that $k\le \cunder  m$ 
with $m$ large depending on $\cunder$ but independently of $k$. 
The case $m$ small with large $k$ would require new ideas as remarked in \cite{kap}. 
\end{remark}

\section{Constructions with necks at the poles or the equator}
\label{S:poles}
$\phantom{ab}$
\nopagebreak
In this section, we construct doublings of $\Spheq$ whose configurations $L$ contain the poles, 
or points equally spaced along the equator circle of $\Spheq$, or both, 
in addition to points distributed symmetrically along $2k$ parallel circles in $\Spheq$ as in the previous sections.  

\subsection*{The case with necks at the poles}
\label{SS:poles}
$\phantom{ab}$
\nopagebreak
\begin{definition}
\label{Linf}
Given $\sbold = (\sss_1, \dots, \sss_k)$ as in \ref{Dlpar} and $L[\sbold;m]$ as in \ref{dL}, define
\begin{equation}
\begin{aligned}
\Lhat &= \Lhat[\sbold; m] = L[\sbold;m]\cup L_{pol} \subset \Spheq,
\end{aligned}
\end{equation}
where $L_{pol} = \{ p_N, p_S\} = \grouptwo\, \left( p_N\right)$.
Given also a $\grouptwo$-symmetric function $\tau:\Lhat[\sbold; m]\rightarrow\R$, we denote 
$\tau_i := \tau(p_i)$ for $i=1, \dots, k$ and $\taupol := \tau(p_N)  = \tau(p_S)$.\end{definition}

\begin{definition}[cf. Proposition \ref{existence}]
\label{dphitilde}
Given $\bsigmaunder = (\bsigma, \xibold) \in \ell^1\left( \R^\N\right) \oplus \ell^\infty \left( \R^\N\right)$ as in \ref{existence}, 
$k\in \N$ and $\sss_1 \in   (a_{k+1, \bsigmaunder}, a_{k, \bsigmaunder}) $ (recall \ref{existence}), 
we modify $\phat[\sss_1; \bsigmaunder]$ to $\phitilde[\sss_1; \bsigmaunder]$ by removing the last discontinuity in the sense that 
$\phitilde[\sss_1; \bsigmaunder] := \phat[\sss_1; \bsigmaunder]$ on $\{ \sss \in [0, \sss_{k+1}]\}$ and 
$\phitilde[\sss_1; \bsigmaunder] $ is smooth on $\{ \sss \in [\sss_k, \infty)\} $. 
By defining 
$A_k = A_k[\sss_1; \bsigmaunder]$ and $B_k = B_k[\sss_1; \bsigmaunder]$ by \eqref{Ecoeffs} as in Lemma \ref{Lcoeffs} 
the last condition is equivalent to 
\begin{align}
\label{Ephitilde}
\phitilde := A_k \phie + B_k \phio \quad \text{on} \quad \{ \sss\in [\sss_k, \infty)\}. 
\end{align}
\end{definition}

 \begin{figure}[h]
\centering
\includegraphics[scale=1]{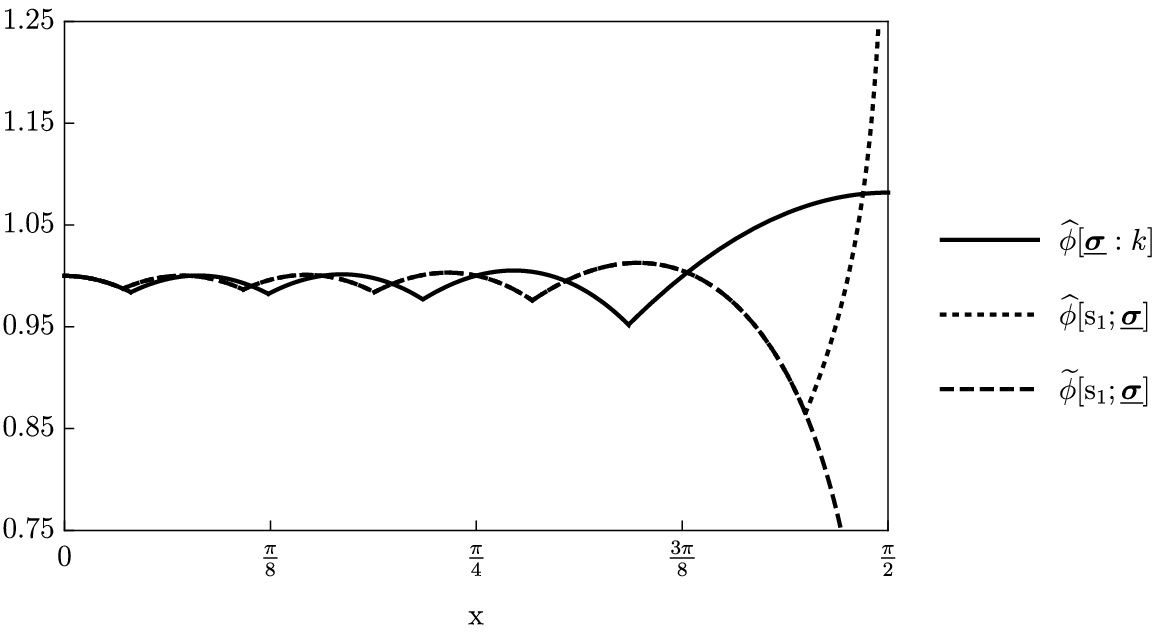}
\caption{Profile curves of some RLD solutions: $\phat[ \bsigmaunder :k]$ (when $k=4$), 
 $\phat[\sss_1; \bsigmaunder]$, and $\phitilde[\sss_1; \bsigmaunder]$, where $\sss_1\in (a_{5, \bsigmaunder}, a_{4, \bsigmaunder})$.  In all cases $\bsigmaunder = \zerobold$.}
\end{figure}

\begin{remark}
In Definition \ref{RL}.(i), we required that RLD solutions $\phi$ be positive.  
This ensures that RLD solutions are meaningful for the linearized doubling apparatus when $L$ consists of points equally spaced on $2k$ parallel 
circles---in particular that each RLD solution $\phi$ is the average of an LD solution 
$\Phi = \varphi[ L; \tau]$ for some $\grouptwo$-invariant $\tau: L \rightarrow \R_+$ (recall Lemma \ref{Lphiavg}).  
When $L$ also contains the poles $\{p_N, p_S\}$, the presence of logarithmic singularities at $\{p_N, p_S\}$ dictates that 
the expansion of the average of an $\grouptwo$-invariant LD solution 
$\Phi = \varphi[ L ; \tau]$ contains a positive multiple of $\phie$ (recall \ref{LGp} and \ref{dLD0}.(ii)) on $\{ \sss \in [\sss_k, \infty)\}$.  
This is ensured by \ref{rH}, \eqref{Erldcoeff} and \ref{dphitilde}.
\end{remark}

\begin{lemma}
\label{Rtautilde} 
There exists $\epsilonunder_2>0$ such that given $k\in \N$ and $\bsigmaunder = (\bsigma, \xibold) \in \ell^1 \left( \R^\N \right) \oplus \ell^\infty\left( \R^\N\right)$ satisfying $|\bsigma|_{\ell^1} + |\xibold|_{\ell^\infty}<\epsilonunder_1/k$ (recall \ref{Prl}), $A_k[\sss_1; \bsigmaunder]$ is a strictly decreasing function of $\sss_1$ on $[ a_{k, \bsigmaunder} - \epsilonunder_2/ k^2, a_{k, \bsigmaunder}]$ which on this interval satisfies $\frac{ \partial A_k [\sss_1; \bsigmaunder]}{\partial \sss_1} \sim_C -k$ and $A_k[ a_{k, \bsigmaunder}; \bsigmaunder] = 0$, where $C>0$ is a constant independent of $k$.
\end{lemma}
\begin{proof}
Let $k\in \N$, $\bsigmaunder = (\bsigma, \xibold) \in \ell^1\left( \R^\N\right) \oplus \ell^\infty \left( \R^\N \right)$ satisfy $|\bsigma|_{\ell^1} + |\xibold|_{\ell^\infty}<\epsilonunder_1/k$, and $\phitilde[\sss_1; \bsigmaunder]$ be as in \ref{dphitilde}, where $\sss_1 \in (a_{k+1, \bsigmaunder}, a_{k, \bsigmaunder})$.  By combining \ref{Prl} with \ref{rsigbij}  we find that
\[ F^{\phie}_-(a_{k, \bsigmaunder}) = F^{\phitilde[ a_{k, \bsigmaunder}; \bsigmaunder]}_{1-} \sim_{\exp(|\bsigma|_{\ell^1} + 3|\xibold|_{\ell^\infty})} F^{\phitilde[ a_{k, \bsigmaunder}; \bsigmaunder]}_{\ave} = \frac{1}{k} + O \left( \frac{1}{k^3}\right).\]

 By a direct calculation, $F^\phie_-(\sss) = \sss + O(\sss^2)$ for small $\sss$.  We we conclude from the preceding in combination with the assumption that $|\bsigma|_{\ell^1}+|\xibold|_{\ell^\infty}<\epsilonunder_1/k$ that 
 $a_{k, \bsigmaunder} - a_{k+1, \bsigmaunder}> C/k^2$.
Assume now that $\sss_1 \in [ a_{k, \bsigmaunder}- \epsilonunder_2/k^2, a_{k, \bsigmaunder}]$ and that $\epsilonunder_2>0$ is small enough that $a_{k, \bsigmaunder}- \epsilonunder_2/k^2> a_{k+1, \bsigmaunder}$.
 From \ref{rH} we have 
\begin{align}
\label{EAk}
A_k[\sss_1; \bsigmaunder] = \phitilde(\sss_k)\left( F^{\phio}_+(\sss_k) - F^{\phitilde}_{k+}\right)\phio(\sss_k).
\end{align}
Differentiating \eqref{EAk} with respect to $\sss_1$, we find
\begin{equation}
\label{EAkd}
\begin{gathered}
\frac{\partial A_k[\sss_1; \bsigmaunder] }{\partial \sss_1} = (I) + (II) \quad \text{where}\\
(I) :=  \frac{\partial \phitilde(\sss_k)}{\partial \sss_1}\left( F^{\phio}_+(\sss_k) - F^{\phitilde}_{k+}\right)\phio(\sss_k), \quad
(II):=
\phitilde(\sss_k) \frac{\partial}{\partial \sss_1}\left[ \left( F^{\phio}_+(\sss_k) - F^{\phitilde}_{k+}\right)\phio(\sss_k)\right].
\end{gathered}
\end{equation}
Using \eqref{Ephiode} and \ref{rsigbij} we find
\begin{align}
\label{E2expand}
(II) = \phitilde(\sss_k)\left( - 2\sech^2 \sss_k \phio(\sss_k) \frac{\partial \sss_k}{\partial \sss_1}
- \frac{\partial F^{\phie}_-(\sss_1)}{\partial \sss_1} \frac{1+\xi_k}{1-\xi_1}\left( e^{\sum_{l=1}^{k-1} \sigma_l}\right) - 
F^\phitilde_{k+} \partial \phio(\sss_k) \frac{\partial \sss_k}{\partial \sss_1}\right).
\end{align}
Using Corollary \ref{Csderiv}, \ref{Prl}.(i), and that $F^\phie_-(\sss) = \sss + O(\sss^2)$ for small $\sss$, we see that $(II) \sim_C -k$.  
To estimate $(I)$, first observe from the estimate on $(II)$ and that $F^\phio_+(\sss_k) - F^\phitilde_{k+} = 0$ when $\sss_1 = a_{k, \bsigmaunder}$ that
\begin{align}
\label{Efdiff}
\left( F^\phio_+(\sss_k) - F^\phitilde_{k+}\right) < Ck( a_{k, \bsigmaunder} - \sss_1) < C \epsilonunder_2/k.
\end{align}

Recall now from \ref{dF} that $\log( \phitilde(\sss_k)) = \int_{0}^{\sss_k} F^\phitilde_+(\sss)d\sss$.  Differentiating under the integral sign, we find
\begin{align}
\label{EPderiv}
\frac{1}{\phitilde(\sss_k)} \frac{ \partial \phitilde(\sss_k)}{\partial \sss_1} = - F^{\phitilde}_{k-} \frac{\partial \sss_k}{\partial \sss_1} -\sum_{i=1}^{k-1} F^\phitilde_i \frac{\partial \sss_i}{ \partial \sss_1}  + 
 \int_{0}^{\sss_k} \frac{\partial F^\phitilde_+(\sss)}{\partial \sss_1} d\sss.
\end{align}
Using Corollary \ref{Csderiv}, we estimate 
\begin{align}
\frac{\partial \sss_i}{\partial \sss_1} < C\left( 2 \sech^2 \sss_i + \left( F^\phitilde_{i-}\right)^2\right)^{-1}  i , \quad i=1, \dots, k,
\end{align}
hence by \ref{rsigbij} and \ref{Prl}.(i)
\begin{align}
\label{Esfderiv}
  F^{\phitilde}_{k-} \frac{\partial \sss_k}{\partial \sss_1} + \sum_{i=1}^{k-1} F^\phitilde_i \frac{\partial \sss_i}{ \partial \sss_1} < Ck^2.
\end{align}
Using \ref{LHmono}, when $\sss\in [\sss_i, \sss_{i+1}]$, $i=1, \dots, k-1$, we have
\begin{align}
\label{Efintderiv}
\frac{\partial F^{\phitilde}_+}{\partial \sss_1} = \left( \frac{ \phitilde(\sss_i)}{\phitilde(\sss)}\right)^2\frac{1}{2}\frac{1+ \xi_i}{1 - \xi_1}\left( e^{\sum_{l=1}^{i-1}\sigma_l}\right) \frac{\partial F^\phie_-}{\partial \sss_1}(\sss_1) +  \left( \frac{ \phitilde(\sss_i)}{\phitilde(\sss)}\right)^2\left(2 \sech^2 \sss_i + \left(F^{\phitilde}_{i+}\right)^2\right)\frac{\partial \sss_i}{\partial F_1}
\end{align}
Using \ref{Lsderiv} and \ref{Csderiv}, we estimate $ \left|\frac{ \partial F^\phitilde_+}{ \partial \sss_1}(\sss)\right |< k$ for $\sss\in [ 0, \sss_k]$.  From this and the estimate that $\sss_k < C\log k $ (recall \ref{Prl}), we find
\begin{align}
\int_{0}^{\sss_k} \frac{ \partial F^\phitilde_+}{ \partial \sss_1}(\sss) d\sss < C k \log k. 
\end{align}
Combining with \eqref{Esfderiv}, we have by estimating \eqref{EPderiv} that 
\begin{align}
\label{Ephideriv1}
\left| \frac{\partial \phitilde(\sss_k)}{\partial \sss_1} \right| < C k^2. 
\end{align}
Combining this with \eqref{Efdiff} we find that  $|(I)|< C k \epsilonunder_2$.  The result then follows from the estimates on $(I)$ and $(II)$ by taking $\epsilonunder_2$ small enough.
\end{proof}

\begin{definition}
\label{Nphitildek}  
Given $k\in \N$, $\bsigmaunder$ as in \ref{Rtautilde}, and $\tautilde \in [0, \epsilonunder_3/k)$, 
where $\epsilonunder_3>0$ is a constant depending only on $\epsilonunder_2$ and the constant $C>0$ in Lemma \ref{Rtautilde},
we define $\phitilde[ \bsigmaunder, \tautilde : k] := \phitilde[ \sss_1; \bsigmaunder]$, 
where $\sss_1$ is uniquely characterized (recall Lemma \ref{Rtautilde}) by the properties that 
$\sss_1 \in [a_{k, \bsigmaunder} - \epsilonunder_1/k^2, a_{k, \bsigmaunder}]$ and $A_k[\sss_1; \bsigmaunder] = \tautilde$.
Note that by \ref{Nphik} $\phitilde[ \bsigmaunder, 0 :k] = \phat[\bsigmaunder: k]$.
We also define $B_{k, \tautilde} := B_k[\bsigmaunder, \tautilde : k]$ as in \eqref{Ephitilde}.
\end{definition}

\begin{definition}
\label{Nsbark}
Given $\phitilde[\bsigmaunder, \tautilde: k]$ as in \ref{Nphitildek}, we define $\stilde_k\in (\sss_k, \infty)$ to be the unique root of $F^\phitilde_+$ in $(\sss_k, \infty)$.  Note that $\stilde_k$ is well defined by Lemma \ref{LFmono}.
\end{definition}

\begin{prop}[cf. Proposition \ref{Prl}]
\label{Prlpoles}
There is a constant $\epsilonunder_1>0$ such that for all $k\in \N$ and all $(\bsigmaunder, \tautilde)= (\bsigma, \xibold, \tautilde)\in \ell^1\left(\R^\N\right)\oplus \ell^\infty\left( \R^\N\right)\times \R_+$ with $|\bsigma|_{\ell^1}+ |\xibold|_{\ell^\infty} < \epsilonunder_1$ and $ \tautilde \in [0,  \epsilonunder_3/ k)$, 
$\phitilde =\phitilde[\bsigmaunder, \tautilde: k]$ satisfies the following.
\begin{enumerate}[label=(\roman*).]
\item  $F^{\phitilde}_{avg} = \frac{1}{k} +O\left(\frac{1}{k^2}\right)$ and $\sech^2 \sss_k \sim_C \frac{1}{k}$.  
\item \begin{enumerate}[label=(\alph*)]
	\item $ F^{\phitilde}_{1-}=2\tanh \sss_1 +O\left(\frac{1}{k^3}\right).$
	\item  $F^{\phitilde}_{i-}+F^\phitilde_{i-1+} = 2(\tanh \sss_i - \tanh \sss_{i-1}) +O\left(\frac{1}{k^2(k-i+1)}\right)$ for $i=2, \dots, k$.
	\item $F^{\phitilde}_{k+}=2(\tanh \stilde_k- \tanh \sss_k) + O\left(\frac{1}{k^2}\right)$.
	\end{enumerate}
\item $\left\|1-\phitilde(\sss): C^0\left( \{ \sss\in [0, \sss_k]\}\right) \right\| \le \frac{C}{k}\log k$ and  $\left|\log\frac{\phitilde(\sss_i)}{\phitilde(\sss_{i-1})}\right| \le \frac{C}{k}$ for $2\leq i \leq k$.    If $\bsigmaunder=\zerobold$, then also $1>\phitilde(\sss_1)> \cdots >\phitilde(\sss_k)$.  

\item There is a constant $C>0$ depending only on $\epsilonunder_1$ such that for any $1\leq i<j\leq k$, 
\begin{align*}
\sss_j - \sss_i &>  \left(1+\frac{C}{k}\right) \frac{j-i}{2k}
\quad \text{and} \quad 
\sech^2 \sss_i - \sech^2 \sss_j \le \frac{C(j-i)}{k}.
\end{align*}
\end{enumerate}
\end{prop}
\begin{proof}
Since $\phat[\sss_1; \bsigmaunder]$ depends continuously on $\sss_1$ on compact subsets of $\cyl$ and $\lim_{\tautilde \searrow 0} \phitilde[\bsigmaunder , \tautilde: k] = \phat[\bsigmaunder: k ]$, it follows from Proposition \ref{Prl} that for $\tautilde$ sufficiently small (depending on $k$), the estimates in \ref{Prlpoles} hold.  It remains to be seen that these estimates hold when $\tautilde \in [0, \epsilonunder_3/k)$.  

To this end, note that on $\{\sss\in [\sss_k, \infty)\}$, 
\begin{equation}
\label{EFtau}
\begin{gathered}
F^\phitilde_+(\sss)  = \frac{ \tautilde \partial \phie (\sss) + B_{k, \tautilde} \partial \phio(\sss)}{\tautilde \phie(\sss) +  B_{k, \tautilde} \phio(\sss)} 
= \frac{\sech^2 \sss}{\tanh \sss } (I) + (II) 
\quad \text{where} \quad 
\\ 
(I) := \frac{ B_{k, \tautilde} - \tautilde \sss}{ B_{k, \tautilde}+ \tautilde ( \coth \sss - \sss)}, 
\quad   
(II) := \frac{\tautilde}{ B_{k, \tautilde} + \tautilde (\coth \sss - \sss)}.  
\end{gathered}
\end{equation}
Evaluating \eqref{EFtau} at $\sss_k$, taking $\tautilde\in [0, \epsilonunder_3/k)$, and using that $\sss_k \sim_C \log k$, which follows from the fact that $\sss^{\phat[\bsigmaunder: k]}_k \sim_C \log k$ (recall \ref{Prl}.(i)) by using the flux monotonicity and that $\sss_1<a_{k, \bsigmaunder}$, we estimate that
 $0<(I) < 2$ and $|(II)| < \epsilonunder_3/k$.  Combining this with \eqref{EFtau} and \eqref{EFcomp}, it follows that 
\begin{align}
\label{EupFavgpol}
\sech^2 \sss_k \sim_{C\exp(|\bsigma|_{\ell^1}+ |\xibold|_{\ell^\infty})} F^\phitilde_\ave.
\end{align}
The rest of the proof proceeds in the same way as the proof of \ref{Prl}, so we point out only the differences.  The last equation in \eqref{Annuli} must be replaced with 
\begin{align}
F^\phitilde_{k+} = 2( \tanh \stilde_k - \tanh \sss_k ) + \int_{\sss_k}^{\stilde_k} \left( F^\phitilde_-(\sss)\right)^2 d\sss.
\end{align}
It follows that $\tanh \stilde_k = 1 - O\left(\frac{1}{k}\right)$.  Estimating, as in the proof of \ref{Prl}, we find $F^\phitilde_\ave = \frac{1}{k} + O\left(\frac{1}{k}\right)$.  The proofs of the remaining parts are nearly the same as the proof of \ref{Prl}, so we omit them.
\end{proof}

\begin{lemma}[cf. Proposition \ref{PODEest}]
\label{PODEestpol}
 Let $k\in \N$ and $\bsigmaunder = (\bsigma, \xibold)\in \R^{k-1}\times \R^k, \bsigmaunder'= (\bsigma', \xibold') \in \R^{k-1}\times \R^k$.  Suppose that $|\bsigma|_{\ell^1}+ |\xibold|_{\ell^\infty}< \epsilonunder_1/k$, $|\bsigma'|_{\ell^1}+ |\xibold'|_{\ell^\infty}< \epsilonunder_1/k$, and $\tautilde' \in [0,  \epsilonunder_3/k), \tautilde \in [0,  \epsilonunder_3/k)$.  Let $\phitilde = \phitilde[\bsigmaunder', \tautilde: k]$ and $\phitilde' = \phitilde[ \bsigmaunder, \tautilde': k]$. There is a constant $C>0$ independent of $k$ such that: \begin{enumerate}[label = (\roman*).]
\item $\displaystyle{  \left| \Fboldunder^{\phitilde'} - \Fboldunder^{\phitilde} \right|_{\ell^\infty} \le 
\frac{C}{k}\left(  | \bsigma' - \bsigma|_{\ell^1}+ |\xibold' - \xibold|_{\ell^\infty}+|\tautilde'- \tautilde|\right) }$.  
\\
\item   
$\displaystyle{  \max_{1\leq i \leq k} \left| \tanh \sss'_i - \tanh \sss_i \right| \le 
\frac{C}{k}\left(  | \bsigma' - \bsigma|_{\ell^1}+ |\xibold' - \xibold|_{\ell^\infty}+ |\tautilde' - \tautilde|\right) }$.  
\end{enumerate}
\end{lemma}

\begin{proof}
We first prove (i).  
As in the proof of \ref{PODEest}, it suffices to prove
\begin{align}
\label{EFdiff2}
\left| F^{\phitilde'}_1 - F^\phitilde_1 \right| \le \frac{C}{k}\left(  |\bsigma' - \bsigma |_{\ell^1} + |\xibold' - \xibold|_{\ell^\infty}\right).
\end{align}
Fix $k\in \N$, let $\epsilonunder_1$ be as in \ref{Prl}, and consider the map
\begin{align}
\Fcal ( F_1, \bsigmaunder, \tautilde) = \left(F^{\phitilde[\sss_1; \bsigmaunder]}_+(\sss_k) - F^{\phio}_+(\sss_k)\right) \phitilde[\sss_1; \bsigmaunder](\sss_k) \phio(\sss_k) + \tautilde,
\end{align} 
where $ \phitilde[\sss_1; \bsigmaunder]$ is as in \ref{dphitilde} and $\sss_1$ is chosen so that $F^\phitilde_1 = F_1$.  Recall from Remark \ref{rH} that 
\begin{equation}
\label{Ephidecompinf}
\begin{aligned}
\phitilde &= \phitilde(\sss_k)\phio(\sss_k) \left( F^\phio_+(\sss_k) - F^\phitilde_+(\sss_k)\right)\phie\\
&+ \phitilde(\sss_k)\phie(\sss_k)\left( - F^\phie_+(\sss_k) + F^\phitilde_+(\sss_k)\right)\phio
\end{aligned}
\end{equation}
By combining \eqref{Ephidecompinf} with \ref{Nphitildek}, we find that $\phitilde[\sss_1; \bsigmaunder] = \phitilde[\bsigmaunder, \tautilde: k]$ if and only if $\Fcal( F_1, \bsigmaunder, \tautilde) = 0$.
Now let $(F_1, \bsigmaunder, \tautilde) \in \Fcal^{-1}(\{ 0\})$ be arbitrary.  Estimating the partial derivatives of $\Fcal$ in a similar manner as in the proof of \ref{PODEest} using \ref{Csderiv} to estimate the partial derivatives of $\sss_k$, we estimate
\begin{align}
\label{EFestpoles}
\left.\frac{\partial \Fcal}{\partial F_1}\right|_{\pointppol} \sim_C k, \qquad 
\left| \left.\frac{\partial \Fcal}{\partial \sigma_i}\right|_{\pointppol} \right| \le C, \qquad
\left|\left. \frac{\partial \Fcal}{\partial \xi_i}\right|_{\pointppol} \right| \le \frac{C}{k}, \qquad
\left. \frac{\partial \Fcal}{\partial \tautilde} \right|_{\pointppol} = 1.
\end{align}
By the implicit function theorem, locally around $\pointppol \in \Fcal^{-1}\left( \{ 0\}\right)$, $\Fcal^{-1}\left( \{ 0\}\right)$ is a graph over $( \bsigmaunder, \tautilde)$ and moreover (abusing notation slightly),
\begin{equation}
\begin{aligned}
\label{Eiftpoles}
\left.\frac{\partial F_1}{\partial \sigma_i}\right|_{(\bsigmaunder, \tautilde)} &= - \left(\left. \frac{ \partial \Fcal}{ \partial F_1}\right|_{\pointppol}\right)^{-1} \left.\frac{ \partial \Fcal }{ \partial \sigma_i}\right|_{\pointppol}\\
\left.\frac{\partial F_1}{\partial \xi_j} \right|_{(\bsigmaunder, \tautilde)}&= - \left( \left.\frac{ \partial \Fcal}{ \partial F_1}\right|_{\pointppol}\right)^{-1}\left. \frac{ \partial \Fcal }{ \partial \xi_j}\right|_{\pointppol}\\
\left.\frac{\partial F_1}{\partial \tautilde}\right|_{(\bsigmaunder, \tautilde)}  &= -\left( \left.\frac{ \partial \Fcal}{ \partial F_1}\right|_{\pointppol}\right)^{-1}\left. \frac{ \partial \Fcal }{ \partial \tautilde}\right|_{\pointppol}.
\end{aligned}
\end{equation}
\eqref{EFdiff2} follows from this and the estimates \eqref{EFestpoles}.
(ii) follows from a similar, but omitted, calculation.
\end{proof}

\begin{lemma}[cf. Lemma \ref{Lphiavg}]
\label{Lphiavginf}
\label{dPhiinf}
Given $(\bsigmaunder, \tautilde) = (\bsigma, \xibold, \tautilde) \in \R^{k-1}\times \R^k \times \R_+$ as in \ref{Prlpoles} and $m$ as in \ref{Aratio}, there is a unique $\grouptwo$-invariant LD solution (recall \ref{LLDexistence})
\begin{align}
\label{EPhiinf}
\Phi = \Phi[\bsigmaunder, \tautilde: k, m] := \varphi[\Lhat ; \tau'],
\end{align}
characterized by the requirements that 
\begin{enumerate}[label=(\alph*)]
\item $\phi = \phi[ \bsigmaunder, \tautilde : k,m] : = \Phi_\ave$ is a multiple of $\phitilde[\bsigmaunder, \tautilde: k]$
\item $\Lhat = \Lhat [ \, \sbold[ \bsigmaunder, \tautilde: k]\, ; m \, ]$ (recall \ref{Linf})
\item $ \tau'_1 = 1$ (normalizing condition).
\end{enumerate}

Moreover, the following hold: 
 \begin{enumerate}[label=(\roman*).]
\item For $i\in \{1, \dots, k\}$ we have $\displaystyle{\tau'_i =  \frac{\phi(\sss_i)}{m} F^\phi_i}$.  Moreover $\tau'_i$ is independent of $m$ and satisfies
	 $\displaystyle{\tau'_i = \frac{\phitilde[\bsigmaunder, \tautilde: k](\sss_i)}{\phitilde[\bsigmaunder, \tautilde: k ](\sss_1)} e^{\sum_{l=1}^{i-1}\sigma_l}}$.
	
	\item $\displaystyle{ \taupol' = \frac{m}{F^{\phitilde[\bsigmaunder, \tautilde: k]}_1 \phitilde[\bsigmaunder, \tautilde: k](\sss_1)} \tautilde}$.
\item $\displaystyle{\phi = \frac{m}{\phitilde[\bsigmaunder, \tautilde: k](\sss_1) F^{\phitilde[\bsigmaunder, \tautilde: k]}_1} \phitilde[\bsigmaunder, \tautilde: k]}$.
\end{enumerate}

\end{lemma}
\begin{proof}
The proof of the uniqueness part proceeds as in the proof of \ref{Lphiavg}, so we only discuss the existence part.  As in the proof of Lemma \ref{Lphiavg}, it follows that $f: = \phi - \Phi_\ave$ is in $C^\infty_{|\sss|}(\cyl^{\sbold})\cap C^1_{|\sss|}(\cyl)$, hence therefore $f = c_{odd} \phio + c_{even} \phie$ for some $c_{odd}, c_{even}\geq 0$.  Since $f\in C^0_{|\sss|}(\cyl)$, $c_{odd} =0$. By the choice of $\tau'$ in (i)-(ii) above, it follows that $f$ is smooth at the poles, hence $c_{even}=0$ and therefore $\phi= \Phi_\ave$.  The proof of (iii) is exactly the same as in the proof of \ref{Lphiavg}.
\end{proof}

\begin{definition}[cf. Definitions \ref{dGhat}, \ref{DPhat}, \ref{dPprime}]
\label{dGphatinf}
Let $\Phi[\bsigmaunder, \tautilde:k, m]$ be as in \ref{dPhiinf}.  Define $\Ghat \in C^\infty_{sym}( \cyl \setminus L)$, $\Phat \in C^\infty_{|\sss|}(\cyl )$, and $\Phip, E' \in C^\infty_{sym}(\cyl)$ using Definitions \ref{dGhat}, \ref{DPhat} and \ref{dPprime} with $\Phi[\bsigmaunder, \tautilde:k,m]$ in place of $\Phi[\bsigmaunder: k,m]$. 
\end{definition}
\begin{remark}
\label{Rphatinf}
While the defining formula for $\Phat[\bsigmaunder, \tautilde: k, m]$ is the analogous to the one for $\Phat[\bsigmaunder: k, m]$, note that $\Phat[\bsigmaunder, \tautilde: k, m]$ is not extendible to $C^\infty_{|\sss|}(\Spheq)$ since $\phitilde[\bsigmaunder, \tautilde: k]$ is not smooth at the poles. 
\end{remark}

\begin{definition}[cf. Definition \ref{dkernel}]
\label{dkernelinf}
Given $\Lhat[\sbold;m]$ as in \ref{Linf}, define $\Vpol, \Wpol \in C^\infty_\sym(\Spheq)$ by requesting that $\Vpol$ is supported on $D^g_{L_{pol}}(2\delta)\setminus D^g_{L_{pol}}(\delta)$ and on 
$D^g_{p_N}(2\delta)$
\begin{align*}
\Vpol=\Vpol[m] = \Psibold\left[ \delta, 2\delta; \dbold^g_{p_N}\right]\left ( \cos \circ \dbold^g_{p_N}, 0 \right), 
\end{align*}
and that $\Wpol = \Wpol[m] := \Lcalp \Vpol$.  
 We define 
 \begin{align*}
 \skernel_\sym [ L_{pol}] &= \text{\emph{span}}( \Wpol) \subset C^\infty_\sym ( \Spheq),\\
 \skernelv_\sym [ L_{pol}] &= \left\{ v\in C^\infty_\sym(\Spheq) : \Lcalp v \in  \skernel_\sym [ L_{pol}]\right \}\subset  C^\infty_\sym ( \Spheq), \\
 \skernel_\sym [ \Lhat ] &= \skernel_\sym[ L] \oplus  \skernel_\sym [ L_{pol}]\subset C^\infty_\sym ( \Spheq),\\
 \skernelv_\sym[ \Lhat] &= \skernelv_\sym[L] \oplus \skernelv_\sym[L_{pol}]\subset C^\infty_\sym ( \Spheq), 
 \end{align*}
where $\skernel_\sym[ L], \skernelv_\sym[L]$ are as in \eqref{Ekhat}.  
\end{definition}

By the symmetries, clearly $ \skernelv_\sym [ L_{pol}] $ is spanned by $\Vpol$.
\begin{lemma}[cf. Lemma \ref{LV}]
\label{LVinf}
For $i=1, \dots, k$, $V_i, V'_i\in C^\infty_\sym (\cyl)$ satisfy (i) of \ref{LV}.  Moreover,
\begin{align}
\label{EVinfest}
\left\| \Vpol : C^j_\sym( \Spheq, \gtilde\, )\right \| \le C(j). 
\end{align}
Furthermore, $\Ecal_{\Lhat}$ is an isomorphism and $\left\|\Ecal^{-1}_{\Lhat}\right\|\leq C m^{2+\beta} k^\frac{2+\beta}{2}$(recall \ref{dEcal}), where $\left\|\Ecal^{-1}_{\Lhat}\right\|$ is the operator norm of $\Ecal^{-1}_{\Lhat}: \val_\sym[\Lhat] \rightarrow \skernelv_\sym[\Lhat]$ with respect to the $C^{2, \beta}(\Spheq, g)$ norm on the target and the maximum norm subject to the standard metric $g$ of $\Spheq$. 
\end{lemma}
\begin{proof}
The estimates on $V_i, V'_i$ follow exactly as in the proof of \ref{LV}.(i).  The estimate \eqref{EVinfest} follows the uniform bounds of the cutoff $\Psibold$ in the $\gtilde$ metric.  By the symmetries, $\val_\sym[L_{pol}]$ is one dimensional and $\Ecal_{L_{pol}}$ is an isomorphism. Moreover (abusing notation slightly) $\Ecal^{-1}_{L_{pol}}\left( (1, 0)\right) = V_{pol}$ and it is easy to see that $\| \Ecal^{-1}_{L_{pol}} \| \le C m^{2+\beta}$.

 It is clear that $\Ecal^{-1}_{\Lhat}$ splits naturally, i.e. $\Ecal^{-1}_{\Lhat} = \Ecal^{-1}_{L} + \Ecal^{-1}_{L_{pol}}$, so the estimate on $\|\Ecal_{\Lhat}\|$ follows from the above and the estimate on $\|\Ecal_{L}\|$ established in \ref{LV}.
\end{proof}

We now convert the LD solutions constructed above to MLD solutions.  By heuristic arguments which we omit we find that the overall scale $\tau_1$ should be given by
\begin{align}
\label{Etau1pol}
 \tau_1:= \frac{1}{m} e^{\zeta_1} e^{-\phi(\sss_1)} = \frac{1}{m} e^{\zeta_1} e^{- \frac{m}{F^\phi_1}}
\end{align}
and that the scale for the bridges at the poles should be given by 
\begin{align}
\label{Etautilded}
\tautilde = \tautilde[\zetaboldpol;m] : = e^\zetapol \frac{F^{\phat}_1}{m} B_{k, 0},
\end{align}
where $\zeta_1$ and $\zetapol$ are unbalancing parameters used to absorb error terms later.  The continuous parameters of the construction are then $\zetaboldpol : = (\zeta_1, \bsigmaunder, \zetapol) = (\zeta_1, \bsigma, \xibold, \zetapol) \in \R\times \R^{k-1}\times \R^{k}\times \R$, where we require
\begin{align}
\label{Epboundsinf}
(\zeta_1, \bsigmaunder) \quad \text{satisfy \eqref{Epbounds}}\quad \text{and} \quad  \left| \zetapol \right| \leq \cunder_1 \frac{\log m}{m},
\end{align}
where $\cunder_1>0$ will be fixed later and which we assume may be taken as large as necessary depending on $k$ bud independently of $m$.
With the overall scale $\tau_1$ having been chosen, we define the MLD solution
\begin{align}
\label{ELDpol}
\varphi = \tau_1 \Phi + \vunder
\end{align}
for some $\vunder \in \skernel_{\sym}[ \Lhat]$ uniquely determined by the matching condition $\Bcal_L\varphi=\zerobold$: 
By the definitions $\Bcal_L\varphi=\Bcal_L(\tauo\Phi)+\Ecal_L(\vunder)$. 
Using the invertibility of $\Ecal_L$ as in \ref{LV}.(ii),  
the matching condition is equivalent then to 
\begin{equation} 
\label{ELDvpol}
\vunder = -\Ecal^{-1}_L \Bcal_L (\tauo\Phi).  
\end{equation} 
To record in detail the dependence on the continuous parameters we have the following.

\begin{definition}
\label{dvarphiinf}
We assume $\zetaboldpol$ is given as in \ref{Epboundsinf}.  Let $\phi = \phi[\bsigmaunder, \tautilde: k, m]$, $\Phi = \Phi[ \bsigmaunder, \tautilde: k,m]$ and $\tau' = \tau'[\bsigmaunder: k]$ be as in \ref{Lphiavginf}.  We define then $\tau_1 = \tau_1[ \zetaboldpol; m]$ by \eqref{Etau1pol}, an MLD solution $\varphi = \varphi[\zetaboldpol; m ] $ of configuration $(\, \Lhat, \tau, w\, )$ (recall \ref{LLDexistence}) by \ref{ELDpol} and \ref{ELDvpol}, where $\Lhat = \Lhat[\sbold;m]$, $\sbold = \sbold[\bsigmaunder, \tautilde: k]$ (recall \ref{dphitilde}), $\tau = \tau[\zetaboldpol; m ] : \Lhat[\sbold;m ] \rightarrow \R_+$ is $\grouptwo$-invariant satisfying $\tau_i = \tau_1 \tau'_i$ for $i\in \{1, \dots, k\}$, $\taupol = \tau_1 \taupol'$, and $w = w[ \zetaboldpol; m]:= \Lcalp \vunder$.  Finally we define $\mubold = \mubold[\zetaboldpol; m] = (\mu_i)_{i=1}^k \in \R^k, \mubold' = \mubold'[\zetaboldpol; m] = (\mu'_i)_{i=1}^k \in \R^k$, and $\mupol \in \R$ by $\vunder = \sum_{i=1}^k\tau_i \mu_i V_i + \sum_{i=1}^k\tau_i \mu'_i V'_i + \tau_{pol}\mupol \Vpol$, which also implies $w = \sum_{i=1}^k \tau_i \mu_i W_i + \sum_{i=1}^k \tau_i \mu'_i W'_i + \taupol \mupol \Wpol$.
\end{definition}

\begin{lemma}
\label{Lmatchinginf}
Let $\varphi$ be as in \ref{dvarphiinf}.  
The equation $\Bcal_{\Lhat} \varphi = \zerobold$ is equivalent to the system of \eqref{Ephimatching1} and \eqref{Ephimatching2}   
for $i=1, \dots, k$, 
and the condition that (recall  \eqref{Ephitilde})
\begin{align}
\label{Ematchinginf}
0 = \frac{F^\phi_1}{m} \mupol - 1 + \frac{F^\phi_1}{m}\left( \zeta_1 + \log \frac{\taupol'}{4m} + 1\right) + \frac{F^\phi_1}{m} \frac{ B_{k, \tautilde}}{\tautilde}.
\end{align}
\end{lemma}
\begin{proof}
By Lemma \ref{Lmatching}, the equation $\Bcal_{L} \varphi= \zerobold$ is equivalent to the system of \eqref{Ephimatching1} and \eqref{Ephimatching2}.    
Recall that $\tau_{p_N} = \taupol = \tau_1 \frac{m}{F^\phi_1 \phitilde(\sss_1)} \tautilde$.   
By \ref{Linf}, $\Lhat = L \cup L_{pol}$.  
The condition that $\Bcal_{L_{pol}} \varphi = \zerobold$ is therefore equivalent to the conditions that
\begin{align}
\label{Evertpol}
\frac{1}{\taupol} \varphihat_{p_N}(p_N) +  \log \frac{\tau_1 \taupol'}{2} = 0, \quad  d_{p_N} \varphihat_{p_N} = 0. 
\end{align}
It is automatic from the symmetries that $d_{p_N} \varphihat_{p_N} = 0$.  By Definitions \ref{dGphatinf} and \ref{dvarphiinf} and \ref{Lphiavginf}.(iii), it follows that on $D^g_{p_N}(3 \delta)$,
\begin{align*}
\frac{1}{\taupol} \varphihat_{p_N} = \frac{F^\phi_1 \phitilde(\sss_1)}{\tautilde m} \Phi -  G^{\Sph^2} \circ \dbold^g_{p_N} +  \mupol \Vpol.
\end{align*}
Expanding $\Phi = \Pp+\Phat + \Phip$ (recall \ref{dGphatinf}), noting from \ref{dGhat} that $\Pp$ vanishes on $D^g_{p_N}(3 \delta)$, using \ref{DPhat} and \ref{dGphatinf} to expand $\Phat$, and \ref{LGp}.(iii) to expand $G^{\Sph^2} \circ \dbold^g_{p_N}$, we find that on $D^g_{p_N}(3 \delta)$,
\begin{align*}
\frac{1}{\taupol} \varphihat_{p_N} = \frac{B_{k, \tautilde}}{\tautilde} \phio + \phie + \frac{F^\phi_1 \phitilde(\sss_1)}{m\tautilde} \Phip -  ( \log 2 - 1) \phio -  \phie +   \mupol \Vpol.
\end{align*}
Evaluating at $p_N$, and using that $\Vpol(p_N) = 1$, and that $\Phip(p_N) = 0$, which follows from Lemma \ref{LPhip est}, we have 
\begin{align*}
\frac{1}{\taupol} \varphihat_{p_N} (p_N) = \frac{B_{k, \tautilde}}{\tautilde} - ( \log 2 - 1)  +  \mupol
\end{align*}
Adding $ \log \frac{\tau_1 \taupol'}{2}$ to both sides and using \ref{dvarphiinf} to expand $\tau_1$, we find that \eqref{Evertpol} is equivalent to
\begin{align*}
0 =  \mupol - \frac{m}{F^\phi_1}  +  \left( \zeta_1 + \log \frac{\taupol'}{4m} + 1\right) +  \frac{B_{k, \tautilde}}{\tautilde}.
\end{align*}
Multiplying through by $\frac{F^\phi_1}{m}$ completes the proof. 
\end{proof}

\begin{lemma}[Properties of the MLD solutions, cf. Lemma \ref{LLD}]
\label{LLDinf} 
Let $\zetaboldpol$ be as in \eqref{Epboundsinf} and $\varphi = \varphi[\zetaboldpol; m]$ be as in \ref{dvarphiinf}.
For $m$ large enough (depending only on $\cunder_1$), the following hold:
\begin{enumerate}[label=(\roman*).]
\item $\tauo = \tauo[\zetaboldpol; m]$ and $\mubold = \mubold[\zetaboldpol; m]$ depend continuously on $\zetaboldpol$.  

\item
 $\tauo[\zetaboldpol; m] \sim_{C(\cunder_1)} \tauo[\zerobold; m]$ and $C(\cunder_1)>1$ depends only on $\cunder_1$. 
 
\item (Matching estimates) There is an absolute constant $C$ independent of $\cunder_1$ such that 
	\begin{enumerate}[label=(\alph*)]
	\item  $\left|\zeta_1+\mu_1 \right| \le C$.
	\item  $ \left|  \bsigma - \frac{F^\phi_1}{m} \Dcal \mubold \right|_{\ell^\infty}\le C/m.$
	\item  $\left|\xibold+  \frac{2}{m}\mubold' \right|_{\ell^\infty} \le C/m.$
	\item  $\left | \zetapol -  \frac{F^\phi_1}{m} \mupol \right| \le C \log m/m.$
	\end{enumerate}
\item $\left\|  \varphi : C^{3, \beta}_{\sym}(\Spheq \setminus D^g_{\Lhat}(\delta'_{min}), g) \right\| \le \tau_{min}^{8/9}.$
\item  On $\Spheq \setminus D^g_{\Lhat}(\delta'_{min})$ we have $\tau_{max}^{1+\alpha/5} \le \varphi$.
\item  For all $p\in L$, $(\delta_p)^{-2} \left\| \varphihat_p: C^{2, \beta}\left( \partial D^g_p( \delta_p),\left( \delta_p\right)^{-2} g\right)\right\| \leq \tau_p^{1-\alpha/9}.$
\end{enumerate}
\end{lemma}
\begin{proof}
The bulk of the proof is nearly the same as that of Lemma \ref{LLD}, so we only prove the essentially new estimate, (iii).(d).  Note that by \ref{Lmatchinginf}, we have
\begin{align*}
0 = \frac{F^\phi_1}{m} \mupol - 1 + \frac{F^\phi_1}{m}\left( \zeta_1 + \log \frac{\taupol'}{4m} + 1\right) + \frac{F^\phi_1}{m} \frac{ B_{k, \tautilde}}{\tautilde}.
\end{align*}
 Using Lemma \ref{Lphiavginf}.(iii) to expand $\taupol'$ and $\tautilde$, we find
\begin{align*}
0 = \frac{F^\phi_1}{m}\mupol - 1 + \frac{F^\phi_1}{m} \left(-\log m + \zeta_1 + 1 + \zetapol + \log \frac{F^\phat_1 B_{k, 0}}{F^\phi_1 \phitilde(\sss_1)} \right) + \frac{B_{k, \tautilde} F^\phi_1}{B_{k, 0} F^\phat_1} e^{- \zeta_1}.
\end{align*}
Expanding 
\begin{align*}
 \frac{B_{k, \tautilde} F^\phi_1}{B_{k, 0} F^\phat_1} e^{- \zeta_1} = 1 + \frac{B_{k, \tautilde} F^\phi_1 - B_{k, 0} F^{\phat}_1}{B_{k, 0} F^\phat_1} - \zetapol + O(\zetapol^2),
\end{align*}
we find the matching equation is equivalent to
\begin{align*}
\frac{F^\phi_1}{m}\mupol - \zetapol = \frac{F^\phi_1}{m}\left(  \log m + O(1)\right) -  \frac{B_{k, \tautilde} F^\phi_1 - B_{k, 0} F^{\phat}_1}{B_{k, 0} F^\phat_1} + O(\zetapol^2).
\end{align*}
The estimate now follows from \eqref{Epboundsinf}, \ref{PODEestpol}, and \ref{Prlpoles}.
\end{proof}

\begin{theorem}
\label{Tmainpol}
There is an absolute constant $\cunder_1>0$ such that if $(k, m)\in \N^2$ with $m$ large enough in terms of $\cunder_1$ and $k$, there is ${\zetaboldpol} \in \R^{2k+1}$ satisfying \eqref{Epboundsinf} such that $\varphi[ \zetaboldpol; m]$ satisfies the conditions of Lemma \ref{LLDinf} and moreover there is $\upphihat \in C^\infty(\Mhat)$, where $\Mhat: = M[[\zetaboldpol]]$ such that
\begin{align*}
\left\| \upphihat\right \|_{2, \beta, \gamma, \Mhat} \le \tauhat^{1+ \alpha/4}_1,
\end{align*}
and further the normal graph $\Mhat_\upphihat$ is a genus $2km+1$ embedded minimal surface in $\Sph^3(1)$ which is invariant under the action of $\groupthree$ and has area $\text{Area}(\Mhat_\upphihat)\rightarrow 8 \pi$ as $m\rightarrow \infty$.  
\end{theorem}
\begin{proof}

Given $\zetaboldpol  = (\zeta_1, \bsigmaunder, \zetapol) \in \R\times \R^{2k-1} \times \R$ as in \ref{Epboundsinf}, we denote $\zerobold = (0, \dots, 0) \in \R^{2k+1}$ and 
\[ M[[ \zetaboldpol ]] : = M\left[ \Lhat[\sbold; m], \tau[\zetaboldpol; m], w[\zetaboldpol; m]\right], \quad \Lhat[[\zetaboldpol]] := \Lhat[\sbold; m].\] 

Define
\begin{align*}
B &=\left\{ \, v\in C^{2,\beta}_\sym(M[[\zerobold]]):\|v\|_{2,\beta,\gamma;M[[\zerobold]]} \, \le \, \tau_1[\zerobold; m]^{1+\alpha}\,\right\} 
\times [-\cunder_1,\cunder_1]\times  \left[-\frac{\cunder_1}{m},\frac{\cunder_1}{m}\right]^{2k-1} \\&\times\left [ - \cunder_1 \frac{\log m}{m}, \cunder_1 \frac{\log m}{m}\right]
 \subset 
C^{2,\beta}_\sym(\,M[[\zerobold]]\,)\times\R^{2k+1}.
\end{align*}
The proof is the same in structure as the proof of \ref{Tmain}; in particular, after carrying out steps (1)-(5) as in the proof of \ref{Tmain}, we define
\begin{align*}
\Jcal(v, \zetaboldpol) = \left( u_Q \circ \Fcal_{\zetaboldpol}, \zeta_1+\mutilde_1,  \bsigma - \frac{F^\phi_1}{m} \Dcal \muboldtilde , \xibold + \frac{2}{m}\muboldtilde', \zetapol -  \frac{F^\phi_1}{m} \mutilde_{pol}
\right),
\end{align*}
where $\sum_{i=1}^k \tau_i \mutilde_iW_i + \sum_{i=1}^k \tau_i \mutilde'_i W'_i + \taupol \mutilde_{pol} \Wpol = w+w_H + w_Q$.

It follows from Lemma \ref{LLDinf}.(iii) and the estimates on the norms of $w_H$ and $w_Q$ as in the proof of \ref{Tmain} that $\Jcal( B) \subset B$, so we may apply the Schauder fixed-point theorem as in the proof of \ref{Tmain}.
\end{proof}

\begin{remark}
\label{rtausizepol}
As $k\rightarrow \infty$, the sizes of the catenoidal bridges on each minimal surface in \ref{Tmainpol} tends to become uniform in the same sense as in the case where there are no bridges at the poles (recall \ref{rtausize}).  To see this, let $k \in \N$ and $\zetaboldpol \in \R^{2k+1}$ be as in \ref{Epboundsinf}.  By  \ref{Lphiavginf} and  \ref{dvarphiinf}, 
\begin{align*}
\frac{\taupol[\zetaboldpol;m]}{ \tau_1[\zetaboldpol; m]} = \taupol'[\zetaboldpol;m] = \frac{m}{F^{\phitilde}_1 \phitilde(\sss_1)} \tautilde = \frac{F^{\phat}_1 B_{k, 0}}{F^{\phitilde}_1 \phitilde(\sss_1)} e^{\zetapol},
\end{align*}
By \ref{PODEestpol}, \ref{Prlpoles}, and \ref{Epboundsinf}, it follows that
\begin{align}
\lim_{k\rightarrow \infty} \lim_{m\rightarrow \infty} \tauhat_{pol}/ \tauhat_{1} =1, 
\end{align}
where $\tauhat_{pol}/\tauhat_1$ is the ratio of the size of the bridges at the poles over the size of the bridges at the first jump latitude of a fixed point of $\Jcal$ as in \ref{Tmainpol}.  Then arguing   as in remark \ref{rtausize} to compare with the sizes of the rest of the bridges away from the poles, we conclude 
\begin{align}
\lim_{k\rightarrow \infty} \lim_{m\rightarrow \infty} \tauhat_{max}/\tauhat_{min} = 1,
\end{align}
where $ \tauhat_{max}/\tauhat_{min}$ is the ratio of the maximum size over the minimum size of the bridges associated with a fixed point of $\Jcal$ as in \ref{Tmainpol}.
\end{remark}
\subsection*{The case with necks along the equator}
\label{SS:equator}
$\phantom{ab}$
\nopagebreak

Here we construct doublings of $\Spheq$ where $\Lpar$ contains the equator circle.  Most of the construction is as before, so we outline the argument and describe in detail only those aspects which differ from the construction in Theorem \ref{Tmain}.  To begin, we define an expanded class of RLD solutions.

\begin{notation}
\label{Ns2}
Given $k\in \N$ and $0=\sss_0 <\sss_1<\cdots<\sss_k < \infty$, we denote
$\sbold = (\sss_0, \sss_1, \dots, \sss_k)\in \R^{k+1}$. 
\end{notation}

\begin{definition}[RLD solutions, cf. Definition \ref{RL}]
\label{D:RLD2}
We say $\phi\in C^0_{|\sss|} \left( \cyl\right)$ is a \emph{rotationally invariant linearized doubling solution} if $\phi$ is as in \ref{RL}, except $\sbold^\phi$ in \ref{RL}.(ii) may be as in either \ref{Dlpar} or \ref{Ns2}; in the latter case we require that \ref{RL}.(iii) holds for $i=0, \dots, k$.
\end{definition}

From now on we assume $\phi$ is an RLD solution as in \ref{D:RLD2} where $\sbold^\phi$ is as in \ref{Ns2}. 

\begin{definition}[Quantities associated to RLD solutions, cf. Definition \ref{RLquant}]
\label{RLquant2}
Let $\phi$ be an RLD solution as in \ref{D:RLD2}.  Define
\begin{multline}
\label{dFlist2}   
\Fboldunder^\phi := \left( F^\phi_{0-}, F^\phi_{0+}, F^\phi_{1-}, \dots, F^\phi_{k+} \right) \in \R^{2k+2}_+,
\\
\Fbold^\phi := ( F^\phi_i)_{i=0}^k  \in \R^{k+1}_+,
\quad \bsigma^\phi := (\sigma^\phi_i)_{i=0}^{k-1} \in \R^{k}, 
\quad \xibold^\phi = \left( \xi^\phi_i\right)_{i=1}^k \in \R^k,  
\end{multline}
where for $i=0, \dots, k$,  $j=0, \dots, k-1$, and $l = 1, \dots, k$, 
\begin{equation} 
\label{Exi2'}
F^\phi_{i\pm} := F^\phi_\pm (\sss_i), \quad 
F^\phi_i := F^\phi_+(\sss_i)+F^\phi_{-}(\sss_i), \quad 
e^{\sigma^\phi_j} = \frac{F^\phi_{j+1}}{F^\phi_j}, \quad 
\xi^\phi_l = \frac{F^\phi_{l+} - F^\phi_{l-}}{F^\phi_{l+} + F^\phi_{l-}}. 
\end{equation} 
We define $\bsigmaunder^\phi: = (\bsigma^\phi, \xibold^\phi) \in \R^{k}\times \R^{k}$ and call the entries of $\bsigmaunder^\phi$ the \emph{flux ratios} of $\phi$.   
If $\bsigmaunder^\phi =  \textbf{0}$ we call $\phi$ \emph{balanced}.        
Finally we define
\begin{align*}
F^\phi_{avg} : = \frac{1}{2(k+1)}\left|  \Fbold^\phi\right|_{\ell^1} =  \frac{1}{2(k+1)}\left| \Fboldunder^\phi\right|_{\ell^1}. 
\end{align*}
\end{definition}
Note that if $\phi$ is an RLD solution as in Definition \ref{D:RLD2}, the assumption that $\phi \in C^\infty_{|\sss|} \left( \cyl \setminus L_{par}[\sbold]\right)$ implies that $F^\phi_{0+} = F^\phi_{0-}$.  For this reason, we did not define $\xi^\phi_0$. 

We omit the proof of the next result, which is a straightforward modification of the proof of \ref{existence}.
\begin{lemma}[RLD existence and uniqueness, cf. Proposition \ref{existence}]
\label{eq existence}
Given $F\in (0, \infty)$ and 
$$ 
\bsigmaunder = (\bsigma, \xibold)  = 
\left( \, (\sigma_i )_{i=1}^{\infty} \, , \, (\xi_j)_{j=1}^\infty \, \right)\in \ell^1\left(\R^\N\right) \oplus \ell^{\infty}\left( \R^\N\right)  
$$
satisfying $|\xibold |_{\ell^\infty} < \frac{1}{10}$, there is a unique unit RLD solution $\phieq[F; \bsigmaunder]$ satisfying the following. 
\begin{enumerate}[label=(\alph*).]
\item $\sbold^\phieq$ is as in \ref{Ns2} and $F^{\phieq[F; \bsigmaunder]}_+(\sss_0) = F$.
\item $\bsigmaunder^{\phieq} = \left. \bsigmaunder\right|_{k}$ where $k$ is the number of jump latitudes of $\phi$ and $\left. \bsigmaunder\right|_k := \left( \, (\sigma_i)_{i=0}^{k-1}, (\xi_i)_{i=1}^k\, \right)$.
\end{enumerate}
 Moreover, the following hold.

 \begin{enumerate}[label=(\roman*).]
 \item $k[F; \bsigmaunder]$ is a nonincreasing function of $F$.  Further, for each $\bsigmaunder$ as above there exists a decreasing sequence $\{b_{0, \bsigmaunder}: = \infty, b_{1, \bsigmaunder}, b_{2, \bsigmaunder}, \dots\}$ such that $k[F; \bsigmaunder] = k$ if and only if $F \in [b_{k, \bsigmaunder}, b_{k-1, \bsigmaunder})$.  Moreover each $b_{k, \bsigmaunder}$ depends only on $\left. \bsigmaunder\right|_k $ (defined as above).

\item$\sss^\phieq_1, \dots, \sss^\phieq_k$ are increasing smooth functions of $F$ for fixed $\bsigmaunder$.

\item $\phieq[F; \bsigmaunder]$ is smooth at the poles if and only if $F = b_{k, \bsigmaunder}$ for some $k\geq 1$. 

\item The restriction of $\phieq[F; \bsigmaunder]$ on any compact subset $[0, \infty)$ depends continuously on $F$ and $\bsigmaunder$.
\end{enumerate}
\end{lemma}
We emphasize that the only substantial difference between RLD solutions $\phieq = \phieq[F; \bsigmaunder]$ as in \ref{eq existence} and RLD solutions $\phat = \phat[\sss_1; \bsigmaunder]$ as in \ref{existence} is that while $\phat[\sss_1;\bsigmaunder]$ coincides with $\phie$ on $\{ \sss \in [ -\sss^\phat_1, \sss^\phat_1]\}$, the expansion of $\phieq$ on $\{ \sss \in [ -\sss^{\phieq}_1, \sss^{\phieq}_1]\}$ with respect to the basis $\{ \phie, \phio\}$ contains a positive multiple of $\phio$.

\begin{notation} 
\label{Nphieq} 
Given $k$, $\bsigmaunder$ and $\left.\bsigmaunder\right|_k$ as in \ref{eq existence}, we define
\begin{align*}
\phieq[\bsigmaunder: k] : = \phieq[ \left.\bsigmaunder\right|_k : k]: = \phieq [ b_{k, \bsigmaunder}; \bsigmaunder], \qquad 
\sbold = \sbold[\bsigmaunder: k] := \sbold[ \left.\bsigmaunder\right|_k :k] : = \sbold^{\phieq[\bsigmaunder: k]}.
\end{align*}
By modifying Definition \ref{dphitilde}, Lemma \ref{Rtautilde}, and Definition \ref{Nphitildek} we also define for all appropriately small $\tautilde>0$ RLD solutions
$\phitildeq [ \bsigmaunder, \tautilde: k] = \phitildeq[ \sss_1; \bsigmaunder]$
which satisfy $A_k[ \sss_1; \bsigmaunder] = \tautilde$.
\end{notation}

Straightforward modifications of the respective statements and proofs of \ref{Prl} and \ref{PODEest} provide versions of those results which hold for smooth at the poles RLD solutions $\phieq$ and $\phitildeq$ as in \ref{Nphieq}.  Furthermore, as in Lemma \ref{Lphiavg} we convert each RLD solution $\phieq$ (or $\phitildeq$) as in \ref{Nphieq} into an LD solution whose non-oscillatory part is a multiple of $\phieq$ (or $\phitildeq$).  The only important difference is that (recall \ref{Lphiavg}.(c)) these LD solutions are normalized so that $\tau'(p_0):= \tau'_0 = 1$.

These LD solutions are decomposed and estimated via obvious modifications of the machinery developed in Section \ref{S:LD} used to estimate LD solutions $\Phi$ whose configurations $L$ do not contain points on the equator circle of $\Spheq$.

For constructions where $L$ contains points on the equatorial circle but not at the poles of $\Spheq$, the parameters of the construction are $\zetabold_{eq} =\left( \zeta_0, \bsigmaunder \right) = (\zeta_0, \bsigma, \xibold) \in \R\times \R^{k} \times \R^{k}$, where we require
\begin{align}
\label{Eparam2}
|\zeta_0| \leq \cunder_1, \quad |\bsigma|_{\ell^\infty} \leq \frac{\cunder_1}{m}, \quad |\xibold|_{\ell^\infty} \leq \frac{\cunder_1}{m}.
\end{align}
In cases where $L$ contains both points on the equator and at the poles, the parameters are
$\zetaboldpol_{eq} =\left( \zeta_0, \bsigmaunder, \zetapol \right) = (\zeta_0, \bsigma, \xibold, \zetapol) \in \R \times \R^k \times \R^k \times \R$ where we require that
\begin{align}
\label{Eparam3}
(\zeta_0, \bsigmaunder) \quad \text{satisfies}\quad  \eqref{Eparam2} 
\quad \text{and} 
\quad  |\zetapol| < \cunder_1 \log m /m.
\end{align}

Using the preceding, we modify the steps in Sections \ref{S:MLD} and \ref{S:Main} construct MLD solutions $\varphi[ \zetabold_{eq}; m]$ and $\varphi[ \zetaboldpol_{eq};m]$ and in turn smooth initial surfaces $M[[ \zetabold_{eq}]]$ and $M[[\zetaboldpol_{eq}]]$.  For the sake of brevity, we omit the proofs of the following and leave the routine modifications to the reader.

\begin{theorem}
\label{Tspheq}
There is an absolute constant $\cunder_1>0$ such that if $(k,m)\in \N^2$ satisfies Assumption \ref{Aratio}, there is $\zetaboldhat_{eq} \in \R^{2k+1}$ satisfying \eqref{Eparam2} 
 and moreover there is $\upphihat \in C^\infty(\Mhat)$, where $\Mhat: = M[[ \zetaboldhat_{eq}]]$ such that
\begin{align*}
\left\| \upphihat\right \|_{2, \beta, \gamma, \Mhat} \le \tauhat_0^{1+ \alpha/4},
\end{align*}
and further the normal graph $\Mhat_\upphihat$ is a genus $(2k+1)m-1$ embedded minimal surface in $\Sph^3(1)$ which is invariant under the action of $\groupthree$ and has area $\text{Area}(\Mhat_\upphihat)\rightarrow 8 \pi$ as $m\rightarrow \infty$.  
\end{theorem}

\begin{theorem}
\label{Tspheqpol}
There is an absolute constant $\cunder_1>0$ such that if $(k,m)\in \N^2$ satisfies Assumption \ref{Aratio}, there is $\widehat{\zetaboldpol}_{eq} \in \R^{2k+2}$ satisfying \eqref{Eparam3}
 and moreover there is $\upphihat \in C^\infty(\Mhat)$, where $\Mhat: = M[ [ \widehat{\zetaboldpol}_{eq} ]]$ such that
\begin{align*}
\left\| \upphihat\right \|_{2, \beta, \gamma, \Mhat} \le \tauhat_0^{1+ \alpha/4},
\end{align*}
and further the normal graph $\Mhat_\upphihat$ is a genus $(2k+3)m-1$ embedded minimal surface in $\Sph^3(1)$ which is invariant under the action of $\groupthree$ and has area $\text{Area}(\Mhat_\upphihat)\rightarrow 8 \pi$ as $m\rightarrow \infty$.  
\end{theorem}

\end{document}